\documentclass[11pt,a4paper,reqno]{amsart}
\usepackage{cancel}
 \usepackage[utf8]{inputenc}
 \usepackage[english]{babel}
 \usepackage[normalem]{ulem}
\usepackage{hyperref}
\usepackage{color,latexsym,amssymb,amsmath,amsthm,amsfonts,enumerate,pdfsync,float,lmodern,multicol,paralist,graphics,graphicx,epsfig,tikz,dsfont,mathrsfs,esint,enumitem,pifont,cleveref,mathtools,textcomp,subcaption}

\usepackage[left=2cm,right=2cm,top=2.25cm,bottom=2.25cm]{geometry}
\usepackage{mathrsfs}

\usepackage{subcaption}
\usepackage{footmisc}

\usepackage{pifont}

\usepackage{pgfplots}
\pgfplotsset{compat=1.18}
\usetikzlibrary{patterns}
\usepackage[title]{appendix}
\usepackage{bbm}
\usepackage{dsfont}

\newtheorem{theorem}{Theorem}[section]
\newtheorem{lemma}[theorem]{Lemma}

\newtheorem{remark}[theorem]{Remark}

\newtheorem{hypothesis}[theorem]{Hypothesis}

\newtheorem{proposition}[theorem]{Proposition}

\makeatother

\numberwithin{equation}{section}

\usepackage{cleveref}
\crefname{fig}{Figure}{Figure}
\crefname{lemma}{Lemma}{Lemmas}
 \def\j{{\mathcal{J}}}
 \def\D{{\mathcal{D}}}
\def\R{{\mathbb{R}}}

\newcommand{\dd}{\,\mathrm d}

\title[Propagation in integro-differential equations of ignition type]{Propagation rates in integro-differential \\ equations of ignition type}

\author{\'Emeric Bouin}
\address[E. Bouin]{CEREMADE - Université Paris-Dauphine, PSL Research University, UMR CNRS 7534, Place du Mar\'echal de Lattre de Tassigny, 75775 Paris Cedex 16, France.}
\email{bouin@ceremade.dauphine.fr}

\author{Jérôme Coville}
\address[J. Coville]{UR 546 Biostatistique et Processus Spatiaux, INRAE, Domaine St Paul Site Agroparc, F-84000 Avignon, France.}
\address{ICJ UMR 5208 - Universite Claude Bernard Lyon 1, Campus de la Doua, 69622 Villeurbanne, France.}
\email{jerome.coville@inrae.fr}

\author{Xi Zhang}
\address[X. Zhang]{School of Mathematics, Changsha University, Changsha, Hunan 410022, People's Republic of China.}
\address{School of Mathematics and Statistics, HNP-LAMA, Central South University, Changsha, Hunan 410083, People's Republic of China.}
\email{xizhangmath@gmail.com}

\begin{document}
\begin{abstract}
 This paper provides sharp rates of invasion in nonlinear integro-differential equations of ignition type. The framework covers both
integrable convolution kernels and singular
fractional-type kernels. While the emergence of a threshold on the decay $1+2s$ of the jump kernel is expected from earlier works, far fewer quantitative results were available in such a general framework. In particular, in addition to the spreading rates, our study provides information on the shape of the invasion profiles. Importantly, we obtain the first sharp spreading estimate in the
critical case $s=\frac12$, which separates linear-in-time and accelerated regimes.
 \end{abstract}


    	\maketitle

\tableofcontents

\section{Introduction}

Reaction–dispersion equations such as,
\begin{equation*}
\begin{cases}
    \partial_t u -\mathcal{D}[u]=f(u),&\text{on }\R^+\times \R,\\
    u(0,\cdot)=u_0,
\end{cases}
\end{equation*}
where $\D$ is the generator of a Lévy process, provide a fundamental framework for modelling the spread of biological, chemical, economic or physical phenomena in an environment. A central question for equations of this
type is the rate at which solutions propagate starting from a front-like initial datum $u_0$, and how this speed is influenced by the interplay between dispersal mechanisms and nonlinear reaction terms. While classical results for standard diffusion, such as those in the Fisher–KPP setting, predict propagation at a constant asymptotic speed \cite{AronsonWeinberger1975,Fisher1937,KPP1937},
it is now well understood that nonlocal dispersal effects can dramatically alter this picture, leading in some cases to accelerated spreading (see \textit{e.g.}\cite{Garnier2011,MR3817762} among many others).

Obtaining quantitative estimates of the spreading speed is thus nontrivial and reveals the interplay between the generator $\D$ and the form of the nonlinearity $f$. For bistable and monostable nonlinearities, the picture is now fairly clear. Indeed, propagation at finite speed is the only possibility when the nonlinearity is bistable, and we refer to \cite{bouin2024} for an up-to-date review.  For monostable nonlinearities, clear (optimal) dichotomies between constant speed and acceleration regimes have been obtained, see \cite{Garnier2011,MR3817762,cabre2013influence,Alfaro2017,Bouin2021,Coville2021,Gui2015,zhang2023optimal}.  The situation is less clear for nonlinearities of ignition type. Apart from the linear spreading situation (see \cite{Mellet2009,Chen1997,coville2007travelling,Shen2017,Shen2017a}) and the recent construction of fronts for a generic $\D$ \cite{bouin2024} which are well documented, the accelerated spreading situation is poorly understood. For
$\D=-(-\Delta)^s$ with $s<1/2$ and a generic $f$, sharp acceleration estimates
were obtained independently in \cite{Coville2021,zhang2023optimal}. A more general
result by two of the authors, subject to a restriction on the ignition
threshold, appears in \cite{Bouin2021ign}. Importantly, there are no results on spreading estimates for operators comparable
to $-(-\Delta)^{\frac{1}{2}}$, highlighting some gaps in previous approaches.

In this work, we investigate further propagation phenomena in reaction–dispersion models combining ignition-type nonlinearities with nonlocal dispersal operators driven by fat-tailed kernels. Our aim is to completely classify spreading rates for a broad class of operators $\D$. Ignition nonlinearities are characterised by a threshold effect: the reaction term vanishes below a critical value and becomes active only above it. This is typical of combustion phenomena. In local diffusion models, this typically leads to travelling fronts with finite speed. In contrast, fat-tailed kernels—whose slow decay at infinity allows for long-range interactions—are known to induce acceleration in monostable equations, raising the question of whether and how such effects persist in the presence of an ignition threshold.

We consider the following integro-differential ignition equations
\begin{equation}\label{ceq}
\left\{
\begin{aligned}
    &\partial_t u -\mathcal{D}[u]=f(u),&\R^+\times \R,\\
    &u(0,\cdot)=u_0,
\end{aligned}
\right.
\end{equation}
where the operator $\mathcal{D}[\cdot]$ is given by 
 \begin{equation*}
\mathcal{D}[\varphi]=P.V.\int_\R [\varphi(\cdot-y)-\varphi(\cdot)]J(y)dy,
 \end{equation*}
with $J$ as follows.
\begin{hypothesis}\label{kernel_hypothesis}
Let the kernel $J:\R^\star \mapsto [0,\infty)$ be a measurable even function. It satisfies the following conditions:
\begin{itemize}
    \item[(i)] Algebraic tails: For some $s>0$, there exist constants $\j_0>0$ and $R_0\ge 1$ such that 
     \begin{equation*}
    J\ge \j_0^{-1}|\cdot|^{-1-2s}\quad\text{on }[R_0,+\infty) \quad \text{and }\ J\le \j_0|\cdot|^{-1-2s}\quad \text{on }[1,+\infty).
     \end{equation*}
       \item[(ii)] Localized second moment boundedness: There exists a positive constant $\j_1$ such that
     \begin{equation*}
    \int_{|z|\le 1}z^2J(z) \dd z \le 2\j_1.
     \end{equation*}
\end{itemize}
\end{hypothesis}
Observe, for example, that with the sign convention used above, if
 \begin{equation*}
J(y)=c_{1,s}|y|^{-1-2s},
 \end{equation*}
then, up to the normalisation constant,
 \begin{equation*}
\mathcal D=-(-\Delta)^s.
 \end{equation*}

\noindent The nonlinearity satisfies the following hypothesis.
\begin{hypothesis}\label{hypo_f}
The function $f\in \mathscr{C}^{0,1}([0,1];\R)$ is an ignition nonlinearity if there is $\theta\in(0,1)$ such that 
   \begin{equation*}
       f=0 \quad\text{on }[0,\theta]\cup\{1\}\quad\text{and }\quad f>0\quad\text{on } (\theta,1).
   \end{equation*}
\end{hypothesis}
\noindent The Cauchy problem is complemented with a front-like initial datum.
\begin{hypothesis}\label{hy-initial-data}
    The initial data $u_0:\R\to [0,1]$ is uniformly continuous and satisfies
     \begin{equation*}
    \theta_0\mathds{1}_{(-\infty,0]}\le u_0\le \mathds{1}_{(-\infty,\xi]},
     \end{equation*}
    for some $\theta_0\in (\theta ,1]$ and $\xi\in\R^+$.
\end{hypothesis}
Our objective is to understand the mechanisms governing propagation in this setting and to determine whether accelerated spreading can occur despite the threshold structure of the nonlinearity. We establish quantitative estimates on the position of level sets and identify regimes in which the heavy tails of the dispersal kernel overcome the inhibiting effect of the ignition threshold.

To describe the behaviour of propagation, we introduce the following notation 
\begin{equation}\label{levelsets}
X_\lambda(t;\Phi)=\sup\left\{x\in\mathbb{R}:\Phi(t,x)\ge \lambda\right\}.
\end{equation}
Throughout the paper, $u$ denotes the unique bounded mild solution associated with the symmetric Lévy semigroup generated by $\D$. We only use the following standard properties of this solution: positivity and comparison principle with bounded piecewise smooth sub- and supersolutions. A discussion about well-posedness and comparison statement may be found in \cite{bouin2021sharp}.

The main result of this paper is the following. 
\begin{theorem}\label{thm_all_case}
Let $u$ be a solution to the Cauchy problem \eqref{ceq}, assuming Hypotheses \ref{kernel_hypothesis}, \ref{hypo_f} and \ref{hy-initial-data}. For any $\lambda\in(0,1)$ and $t$ large enough, 
\begin{equation*}
    X_\lambda(t;u) \asymp_\lambda\footnote{The notation $X\asymp_{\Lambda_1, \Lambda_2,\dots} Y$ means that there exists a positive constant $C_{\Lambda_1, \Lambda_2,\dots}$, depending only on some parameters $\Lambda_1$, $\Lambda_2$,\dots, such that $C_{\Lambda_1, \Lambda_2,\dots} Y\le X \le C^{-1}_{\Lambda_1, \Lambda_2,\dots} Y$.} \begin{cases}
        t^\frac{1}{2s}, & 0<s<\frac{1}{2},\\[5pt]
        t \ln(t), & s=\frac{1}{2},\\[5pt]
        t, &  s>\frac{1}{2}.
    \end{cases}
\end{equation*}
\end{theorem}
Some comments are now in order. The situation of linear spreading is actually rather well known and will not be discussed in this paper; it is included only to complete the trichotomy. It can be deduced from \cite{Gui2015,bouin2024}. The new contributions of the present paper concern $s\leq\frac{1}{2}$. This paper improves, generalises, and complements previous works, either giving sharper rates, simpler proofs, or more quantitative information on the shape of the invasion front. Finally, it is worth mentioning specifically that the case $s=\frac{1}{2}$ had not previously been covered and that we derive sharp estimates and information on the invasion profiles. 

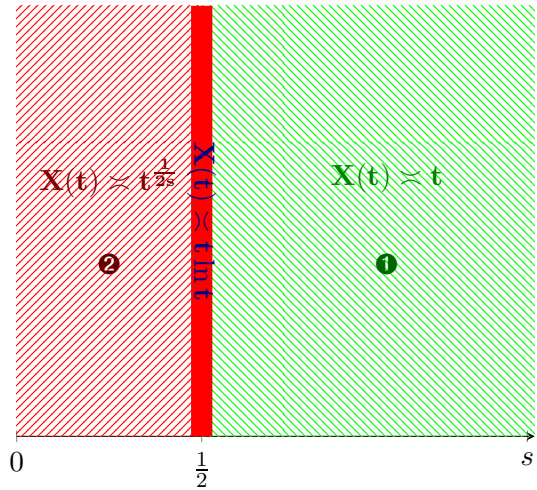
\begin{figure}[H]
     \centering
\begin{tikzpicture}
    \begin{axis}[
    axis x line=bottom,
    axis y line=none,
    xmin = 0,
    xmax = 1.4,
    ymin = 0,
    ymax = 1,
    xtick = {0,1/2,1.38},
    xticklabels = {0,$\frac{1}{2}$,$s$},
    ytick = \empty,
    ]
 
 \fill[
 pattern=checkerboard,   
 pattern color=green,
pattern=north west lines,
] (0.5,0) rectangle (2,1); 

\fill [
        domain=0:1, 
        samples=20,
        pattern color=red,
        pattern=north east lines,
        fill opacity=0.3,
]  (0,0) rectangle (0.5,1);
 \draw[red, line width=8,opacity=0.9] (axis cs:0.5,0) -- (axis cs:0.5,1); 

\node[anchor= center] at (axis cs:0.25,0.6) {$\textcolor{black!50!red}{\bf X(t)\asymp t^\frac{1}{2s}}$};
\node[anchor= center] at (axis cs:0.25,0.4) {\textcolor{black!60!red}{\ding{203}}};

\node[anchor=center] at (axis cs:1,0.6) {$\textcolor{black!50!green}{\bf X(t)\asymp t}$};

\node[anchor=center] at (axis cs:1,0.4) {\textcolor{black!60!green}{\ding{202}}};

\node[anchor=center] at (axis cs:0.5,0.5) {\rotatebox{-90}{$\textcolor{black!50!blue}{\bf X(t)\asymp t\ln t}$}};

\end{axis}
\end{tikzpicture}    \caption{Summary of the results: rates of propagation for the algebraic kernels $J$ satisfying Hypothesis \ref{kernel_hypothesis}. In the green zone \textcolor{black!60!green}{\ding{202}}, the solution to \eqref{ceq} propagates at a linear-in-time rate. In the red zone \textcolor{black!60!red}{\ding{203}}, acceleration occurs: $X_\lambda(t)\asymp t^\frac{1}{2s}$. In the critical case $s=\frac{1}{2}$, acceleration occurs: $X_\lambda(t)\asymp t\ln t$.}
    \label{diagram}
 \end{figure}

The origin of the three propagation regimes can be understood through the truncated first moment of the jump kernel,
 \begin{equation*}
M_1(R):=\int_1^R yJ(y)\,dy.
  \end{equation*}
Under \Cref{kernel_hypothesis},
 \begin{equation*}
M_1(R)\asymp R^{1-2s}\quad\text{if }0<s<\frac12,
 \qquad
 M_1(R)\asymp\ln R\quad\text{if }s=\frac12,
  \end{equation*}
whereas $M_1(R)$ remains bounded when $s>\frac12$. If $R(t)$ denotes the characteristic width and $X(t)$ the rough location of a propagation profile, then heuristically $|\partial_xu|$ is of order
$R(t)^{-1}$, while $|\partial_tu|$ is of order $X'(t)R(t)^{-1}$. The contribution of the dispersal operator is thus roughly of order $M_1(X(t))R(t)^{-1}$ by a first order approximation of the integral. Balancing these two quantities leads to
 \begin{equation*}
 X'(t)\asymp X(t)^{1-2s},\qquad
 X'(t)\asymp\ln X(t),\qquad
 X'(t)\asymp1
 \end{equation*}
 in the three respective regimes, and hence to the scales
  \begin{equation*}
X(t)\asymp t^{\frac{1}{2s}},\qquad
X(t)\asymp t\ln t,\qquad
X(t)\asymp t.
 \end{equation*}
This heuristic also explains why the critical case cannot be obtained
by merely setting $s=\frac12$ in the subcritical construction. Of course, this heuristic is not related to the form of the nonlinearity: the main contribution of the paper is to
prove that these heuristic scales remain valid for
ignition nonlinearities. Furthermore, \Cref{fig:sol} provides a numerical illustration of \Cref{thm_all_case}. The numerical results clearly reveal the transition from accelerated propagation for $s=0.25$ and $s=0.5$ to propagation with a linear-in-time speed for $s=0.75$. In particular, the curves displayed in the second row of \Cref{fig:sol} progressively approach a nearly horizontal profile, providing further numerical evidence for the asymptotic propagation regimes predicted by \Cref{thm_all_case}.

\begin{figure}[H]
\centering 
    \begin{subfigure}[b]{0.3\textwidth} 
        \includegraphics[width=\textwidth]{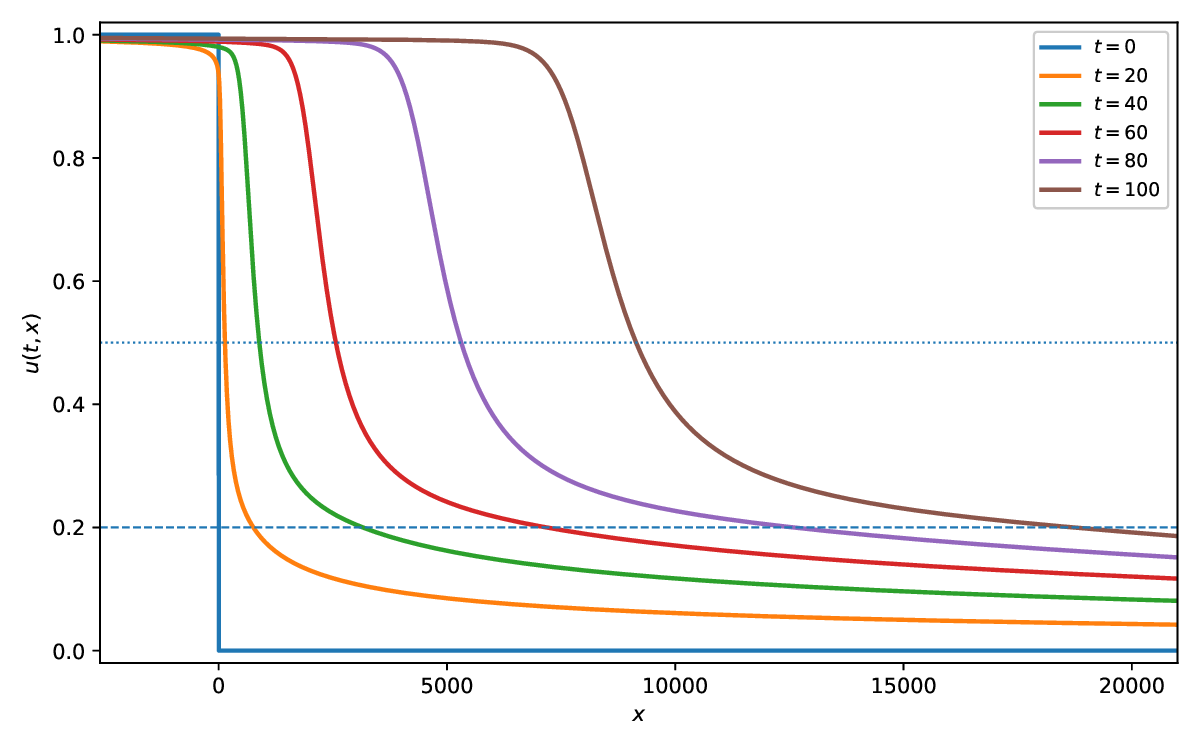} \caption{$s=0.25$} 
    \end{subfigure} 
    \quad
    \begin{subfigure}[b]{0.3\textwidth}  
        \includegraphics[width=\textwidth]{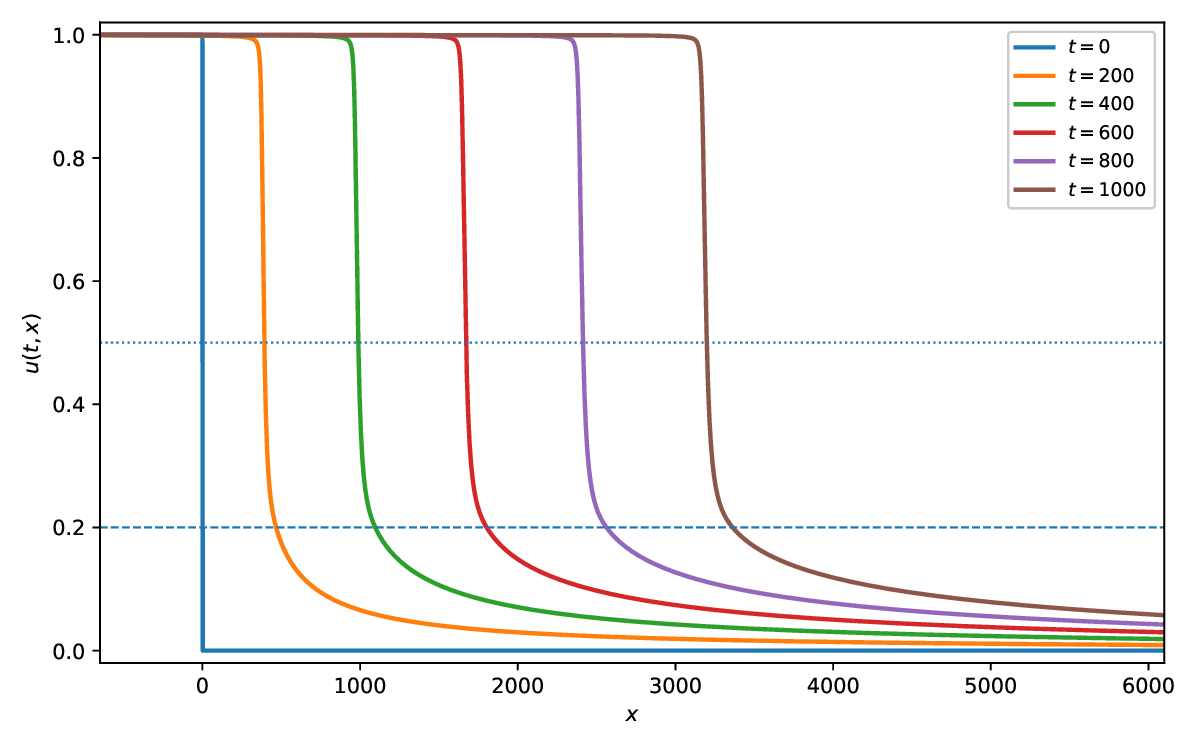} \caption{$s=0.50$}  
    \end{subfigure} 
       \quad
    \begin{subfigure}[b]{0.3\textwidth}  
        \includegraphics[width=\textwidth]{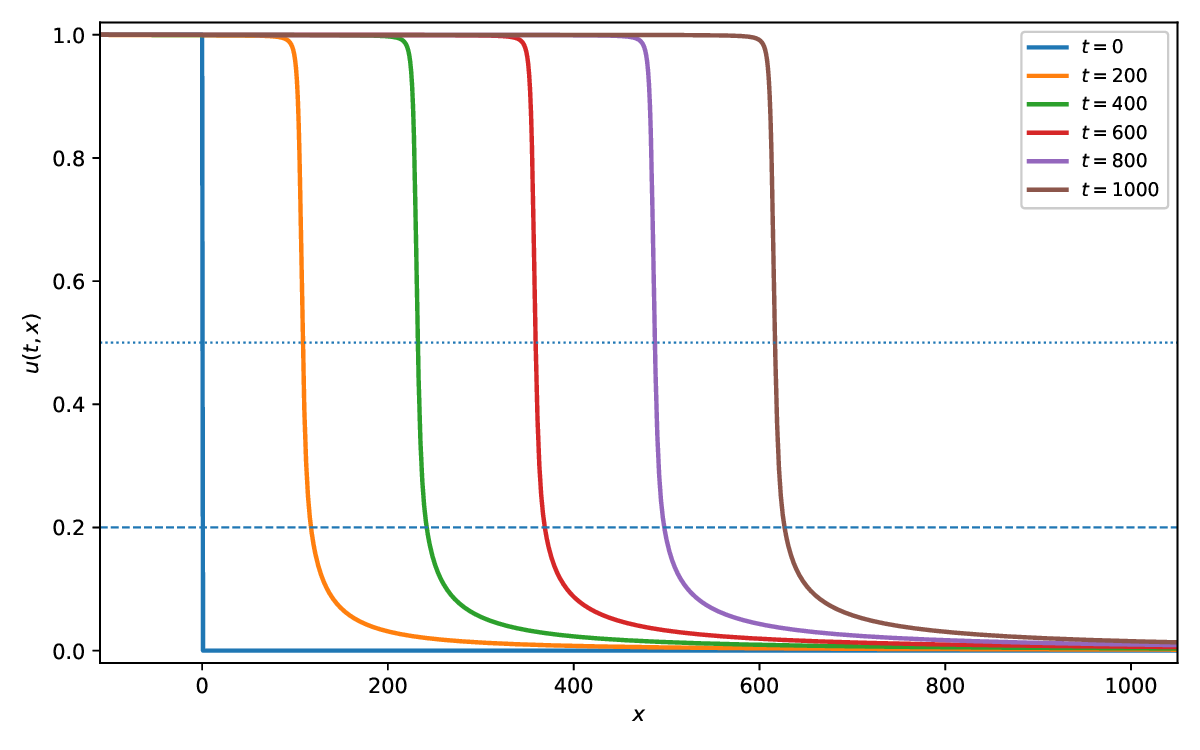} \caption{$s=0.75$}  
    \end{subfigure}  \\
       \begin{subfigure}[b]{0.3\textwidth} 
        \includegraphics[width=\textwidth]{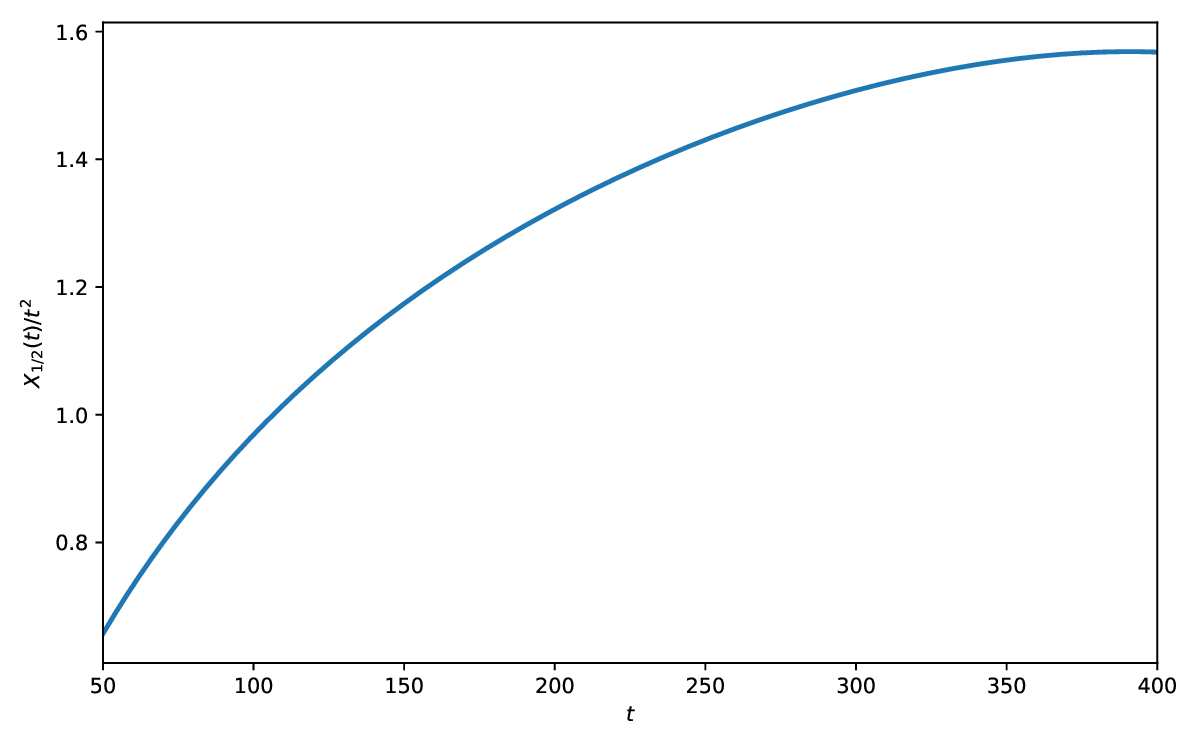} \caption{$\frac{X_{0.5}(t)}{t^2},\quad s=0.25$} 
    \end{subfigure} 
    \quad
    \begin{subfigure}[b]{0.3\textwidth}  
        \includegraphics[width=\textwidth]{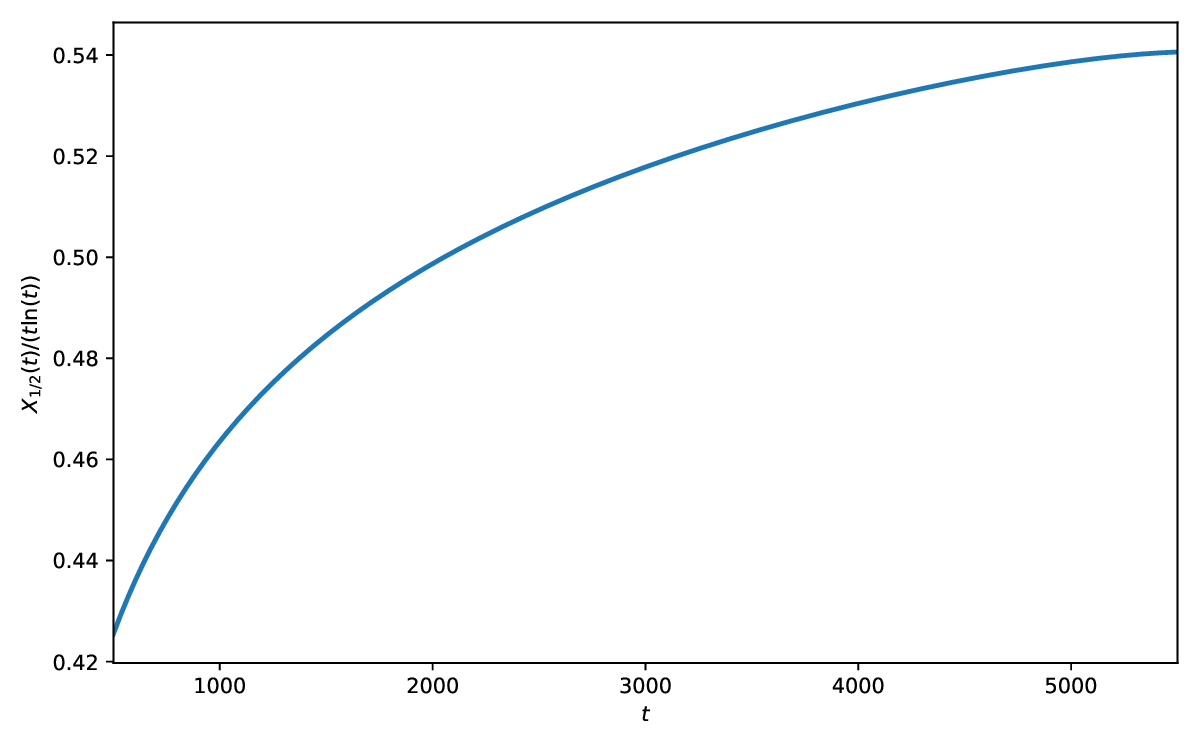} \caption{$\frac{X_{0.5}(t)}{t\ln(t)},\quad s=0.50$}  
    \end{subfigure} 
       \quad
    \begin{subfigure}[b]{0.3\textwidth}  
        \includegraphics[width=\textwidth]{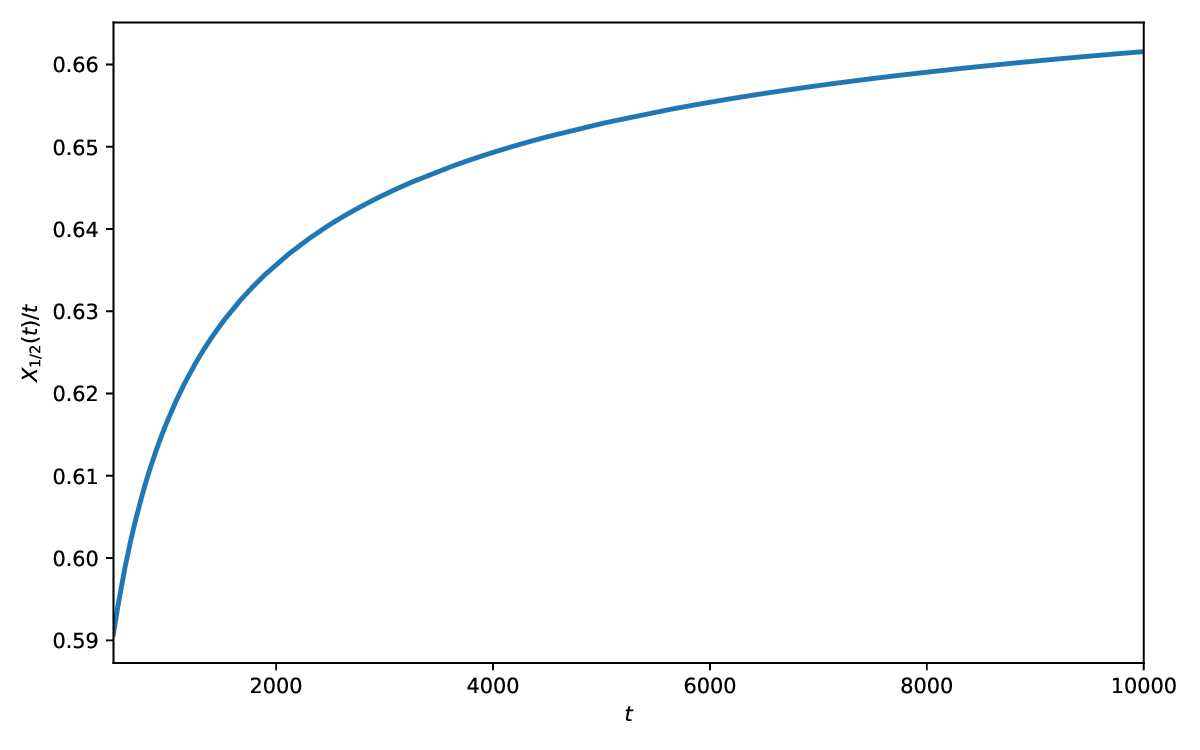} \caption{$\frac{X_{0.5}(t)}{t},\quad s=0.75$}  
    \end{subfigure} 
\caption{ \small
Numerical simulation to \eqref{ceq} with the initial datum $u_0 =\mathds{1}_{(-\infty,0]}$ and the ignition nonlinearity $f=(\cdot-0.2)^2(1-\cdot)\mathds{1}_{[0.2,1]}$. The dispersal kernel is chosen as
        $J(x) =\frac{\Gamma\left(s+\frac12\right)}{\sqrt{\pi}\,\Gamma(s)}(1+x^2)^{-\frac{1+2s}{2}}$ so that $\int_{\mathbb R}J(x)\,dx=1$ and
    $J(x)\sim \frac{\Gamma\left(s+\frac12\right)}{\sqrt{\pi}\,\Gamma(s)}|x|^{-1-2s}$ as $|x|\to\infty$.
    The top row shows the solution profiles $x \mapsto u(t,x)$ at different times $t$ for $s=0.25$, $s=0.5$, and $s=0.75$, respectively. The bottom row shows the evolution in time of the corresponding normalised front positions $\frac{X_{0.5}(t)}{t^2}$, $\frac{X_{0.5}(t)}{t\ln(t)}$, and $\frac{X_{0.5}(t)}{t}$. 
}
\label{fig:sol}
\end{figure} 

We now comment more precisely on the difficulties and novelties of the paper. To prove matching lower and upper bounds, we construct pairs of matching piecewise sub- and super- solutions, respectively. A schematic plot to help the reader may be found in \Cref{fig:scheme}.

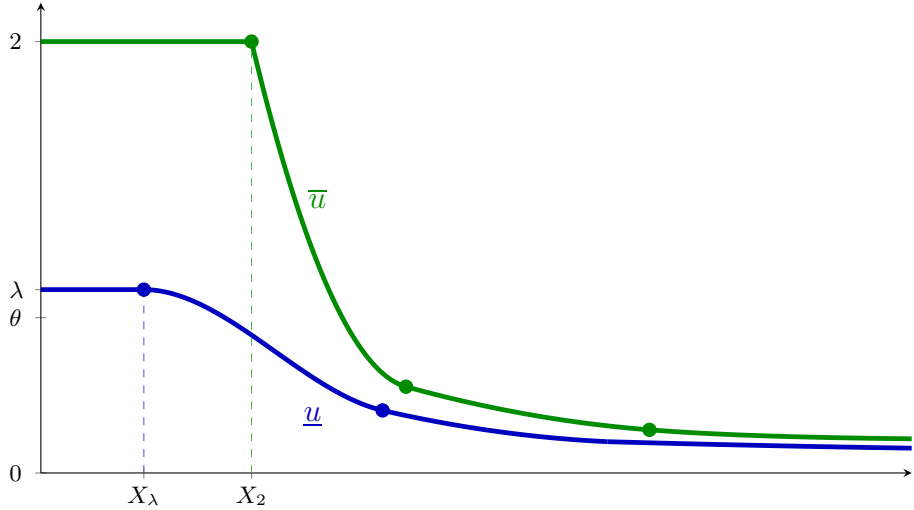
\begin{figure}[H]
     \centering
\begin{tikzpicture}
\colorlet{supercolor}{green!55!black}
\colorlet{subcolor}{blue!75!black}

\pgfmathsetmacro{\thetaLevel}{0.72}
\pgfmathsetmacro{\ThetaLevel}{0.85}
\pgfmathsetmacro{\thetazero}{0.98}

\pgfmathsetmacro{\XA}{2.55}   
\pgfmathsetmacro{\XB}{4.20}   
\pgfmathsetmacro{\XC}{6.80}  
\pgfmathsetmacro{\Xmax}{9.60}
\pgfmathsetmacro{\yA}{2.0}
\pgfmathsetmacro{\yB}{0.40}
\pgfmathsetmacro{\yC}{0.20}
\pgfmathsetmacro{\mB}{-0.12}
\pgfmathsetmacro{\Lone}{\XB-\XA}
\pgfmathsetmacro{\Ltwo}{\XC-\XB}
\pgfmathsetmacro{\aone}{(\yA+\mB*\Lone-\yB)/(\Lone*\Lone)}
\pgfmathsetmacro{\bone}{\mB-2*\aone*\Lone}
\pgfmathsetmacro{\atwo}{(\yC-\yB-\mB*\Ltwo)/(\Ltwo*\Ltwo)}
\pgfmathsetmacro{\btwo}{\mB}
\pgfmathsetmacro{\mC}{2*\atwo*\Ltwo+\btwo}
\pgfmathsetmacro{\alphaR}{-(\mC)/\yC}
\pgfmathsetmacro{\alphaFast}{0.70}
\pgfmathsetmacro{\alphaSlow}{0.001}
\pgfmathsetmacro{\tailWeight}{(\alphaFast-\alphaR)/(\alphaFast-\alphaSlow)}

\pgfmathsetmacro{\XSUBTwo}{1.40}   
\pgfmathsetmacro{\XSUBOne}{3.95}   
\pgfmathsetmacro{\XSUBThree}{6.35}
\pgfmathsetmacro{\XSUBEnd}{9.60}

\pgfmathsetmacro{\ySubLeft}{\ThetaLevel}
\pgfmathsetmacro{\ySubOne}{0.29}
\pgfmathsetmacro{\ySubThree}{0.145}
\pgfmathsetmacro{\ySubEnd}{0.115}

\pgfmathsetmacro{\mSubLeft}{0.0}
\pgfmathsetmacro{\mSubOne}{-0.10}

\pgfmathsetmacro{\hConv}{\XSUBThree-\XSUBOne}
\pgfmathsetmacro{\aConv}{(\ySubThree-\ySubOne-\mSubOne*\hConv)/(\hConv*\hConv)}

\pgfmathsetmacro{\hTail}{\XSUBEnd-\XSUBThree}
\pgfmathsetmacro{\mSubTailStart}{-0.012}
\pgfmathsetmacro{\aTail}{(\ySubEnd-\ySubThree-\mSubTailStart*\hTail)/(\hTail*\hTail)}

\pgfmathsetmacro{\thetaepsLevel}{\ySubThree} 

\begin{axis}[
    axis lines=left,
    width=13.1cm,
    height=7.8cm,
    xmin=0.3,
    xmax=\Xmax,
    ymin=-0.0,
    ymax=2.18,
    xtick={\XSUBTwo,\XA},
    xticklabels={$X_\lambda$,$X_2$},
    xticklabel style={font=\footnotesize,align=center,yshift=1pt},
     ytick={0,\thetaLevel,\ThetaLevel,\yA},
    yticklabels={$0$,$\theta$, $\lambda$,$2$},
    yticklabel style={font=\footnotesize,text width=0.1cm,align=left,xshift=-3pt},
    samples=280,
    clip=false,
    declare function={
      wl(\x)=\aone*(\x-\XA)^2+\bone*(\x-\XA)+\yA;
      wm(\x)=\atwo*(\x-\XB)^2+\btwo*(\x-\XB)+\yB;
      wr(\x)=\yC*((1-\tailWeight)*exp(-\alphaFast*(\x-\XC))+\tailWeight*exp(-\alphaSlow*(\x-\XC)));
      h00(\t)=2*(\t)^3-3*(\t)^2+1;
      h10(\t)=(\t)^3-2*(\t)^2+(\t);
      h01(\t)=-2*(\t)^3+3*(\t)^2;
      h11(\t)=(\t)^3-(\t)^2;
      t1(\x)=(\x-\XSUBTwo)/(\XSUBOne-\XSUBTwo);
      subsegA(\x)=h00(t1(\x))*\ySubLeft + h10(t1(\x))*(\XSUBOne-\XSUBTwo)*\mSubLeft + h01(t1(\x))*\ySubOne + h11(t1(\x))*(\XSUBOne-\XSUBTwo)*\mSubOne;
      subsegB(\x)=\aConv*(\x-\XSUBOne)^2+\mSubOne*(\x-\XSUBOne)+\ySubOne;
      subtail(\x)=\aTail*(\x-\XSUBThree)^2+\mSubTailStart*(\x-\XSUBThree)+\ySubThree;
    }
]

\addplot[supercolor,line width=1.8pt,domain=0.3:\XA] {\yA};
\addplot[supercolor,line width=1.8pt,domain=\XA:\XB] {wl(x)};
\addplot[supercolor,line width=1.8pt,domain=\XB:\XC] {wm(x)};
\addplot[supercolor,line width=1.8pt,domain=\XC:\Xmax] {wr(x)};
\addplot[only marks,mark=*,mark size=2.5pt,supercolor] coordinates {(\XA,\yA) (\XB,\yB) (\XC,\yC)};
\node[supercolor, font=\large,anchor=south west] at (axis cs:3.05,1.18) {$\overline{u}$};

\addplot[subcolor,line width=1.8pt,domain=0.30:\XSUBTwo] {\ySubLeft};
\addplot[subcolor,line width=1.8pt,domain=\XSUBTwo:\XSUBOne] {subsegA(x)};
\addplot[subcolor,line width=1.8pt,domain=\XSUBOne:\XSUBThree] {subsegB(x)};
\addplot[subcolor,line width=1.8pt,domain=\XSUBThree:\XSUBEnd] {subtail(x)};
\addplot[only marks,mark=*,mark size=2.5pt,subcolor] coordinates {(\XSUBTwo,\ySubLeft) (\XSUBOne,\ySubOne)};
\node[subcolor,font=\large, anchor=north west] at (axis cs:3.00,0.35) {$\underline{u}$};
\draw[dashed,supercolor!65] (axis cs:\XA,0) -- (axis cs:\XA,\yA);

\draw[dashed,subcolor!65] (axis cs:\XSUBTwo,0) -- (axis cs:\XSUBTwo,\ySubLeft);
\end{axis}
\end{tikzpicture}
    \caption{Schematic plot of the sub- and supersolutions. More accurate drawings are displayed in \Cref{fig:diagram-subsolution} and \Cref{fig:piecewise-supersolution}, where all the definitions of parameters are given.}
    \label{fig:scheme}
\end{figure}
The subsolutions are front-like and composed of a flat left part, a leading-edge profile which behaves like the integrated tail of the jump kernel $J$ (that is $\int_{y}^{+\infty} J(z)  \dd z  \asymp y^{-2s}$), and a middle part that connects the two former pieces. 
The choice of tail is dictated by the fact that the ignition nonlinearity vanishes far ahead of the propagating front. Since the tail flattens with time due to the acceleration (see for example \cite{bouin2023simple}), the middle part is a slowly varying transition layer whose width increases with time to connect the back and the front parts of the subsolution properly. A schematic view is displayed in \Cref{fig:diagram-subsolution}. An important issue in the construction is the control of the dispersal operator acting on the subsolution. To avoid a blow-up from below, it is mandatory to smooth out the concave corner between the left flat part and the middle profile, and ensure that the right junction between the middle profile and the tail is convex. These requirements add supplementary difficulties. 

The supersolutions also have this piecewise structure, but require four parts. This time, to avoid a blow up from above of the dispersal part, we have to work with smooth enough (typically $\mathscr{C}^{1,1}$) profiles. This adds significant complications compared to related works on the topic. Note that a full $\mathscr{C}^2$ regularity is not required since $s\leq \frac12$: local $\mathscr{C}^{1,1}$ regularity is sufficient because the second differences $\delta_z$ are then at most of order $z^2$, which is exactly compatible with the localised second-moment assumption in \Cref{kernel_hypothesis}. This is crucial in our approach. The far left part is still flat, at a higher level (typically $2$, but any value larger than one would fit). It is linked to a left part, which has the form of a solution to a reactive ODE in time, with the good spreading rate in space: typically the nonlinearity is active in this region and plays the main role. The edge part has the expected decay for the linearised problem around zero. A major difficulty is to find a middle part that joins the left and the right parts and has the desired slopes in space at the two junctions (so that the Lipschitz regularity of the supersolution is ensured). A schematic view is displayed in \Cref{fig:piecewise-supersolution}.

Let us insist on the fact that supersolutions are more difficult to make than subsolutions is due to the asymmetry between sign requirements between the two, and this is related to the difference of regularity targeted between the two. For a subsolution, we are looking for a bound from below on $\D[u]$. A concave corner may create an infinite negative contribution and is therefore incompatible with the lower estimates needed for a subsolution. Thus concave zones have to be (at least) $\mathscr{C}^2$. By contrast, a convex corner has the good (positive) sign in a lower estimate and can be kept. For a supersolution, we look for a bound from above on $\D[u]$. The situation is reversed: a convex corner can produce an uncontrolled positive contribution. It is therefore necessary to connect the pieces with continuity of the derivative and obtain local regularity $\mathscr{C}^{1,1}$. Since the supersolution is really likely to be convex at infinity in space, this is the reason for the more elaborate four-piece construction of the upper barrier.

Note that constructions in the case $s=\frac12$ are more involved because of particular consequences of the logarithmic growth of the truncated moment. This is what leads us to a separate treatment: note in particular that in this case, the spreading rate is $t \ln(t)$ whereas the profile has shape $\frac{x}{t}$ at the edge. This is an additional subtlety of this boundary regime.  

We believe that our construction is broader, more precise and simpler than the one in \cite{zhang2023optimal} that deals with a particular case only. Moreover, the construction of accurate sub- and supersolutions gives some insight of the shape of the true solution, since they have comparable behaviours in time.

It is worth mentioning that the fact that the sub- and supersolutions we construct are valid for eventually large (but finite) times only is not an obstacle. Indeed, following the approach in \cite{bouin2024acceleration} and employing the parabolic comparison principle together with the flattening estimate from \cite{bouin2023simple}, one may obtain an estimate of the solution of the Cauchy problem at a later time which is compatible with the subsolutions. A very simple estimate from above shows that this is also the case for the supersolutions, and the full spreading result of \Cref{thm_all_case} then follows. Since the argument is standard, we omit the details of this additional argument in the rest of the paper. Moreover, it is worth recalling that once the solution of the Cauchy problem \eqref{ceq} has reached a level strictly above the ignition threshold on a sufficiently large interval, it reaches any prescribed level $\lambda<1$ after a uniformly bounded additional time: we do omit the details of this step later on.

The paper is organised as follows. In Section 2, we provide basic estimates on the operator $\D$. Section 3 is devoted to proving \Cref{thm_all_case} when $s < \frac{1}{2}$. Section 4 is devoted to proving \Cref{thm_all_case} when $s = \frac{1}{2}$. 

\section{Preliminaries: basic estimates on \texorpdfstring{$\D$}{D}}

\subsection{Estimates from below}
Since we aim to estimate $\mathcal{D}[\cdot]$ acting on a function that may have convex corners (rather than assuming global $\mathscr{C}^2$ regularity), we refine the estimate obtained in \cite{bouin2021sharp,bouin2024acceleration} as follows.

\begin{proposition}\label{prop-D[phi]} Assume that $v:\mathbb{R}\to [0,1]$ is decreasing, and there are $\xi_1,\xi_2\in \mathbb{R}$ such that $\xi_2-\xi_1\ge R_0$, and $v\equiv \Theta\in(0,1)$ on $(-\infty,\xi_1]$ and convex on $[\xi_2,\infty)$. Let $\Omega \subset[\xi_2+1,\infty)$ be closed, and assume that $v\in \mathscr{C}(\mathbb{R})\cap \mathscr{C}^2(\mathbb{R}\setminus\Omega)$.  We have the following estimates for $\mathcal{D}[v]$:
    \par I. Take $x\in(-\infty,\xi_2+1]$. For any $B_1>1$, 
     \begin{equation*}
\D[v](x)\ge -v(x) \int_{B_1}^{\infty}J(z) \dd z - \left(\j_1+\int_1^{B_1}z^2J(z) \dd z \right) \sup_{y\in[x-B_1,x+B_1]\setminus \Omega} [v''(y)]^-.
 \end{equation*} 
    \par II. Take $x\in[\xi_2+1,+\infty)$. For any $B_2>1$,
     \begin{equation*}
   \D[v](x)\ge y\int_1^{B_2}zJ(z) \dd z -\frac{\j_0 v(x)}{2sB_2^{2s}}+ \frac{\Theta-v(x)}{2s\j_0[x-\xi_1]^{2s}},
     \end{equation*} 
    for any $y\in \partial v(x)$, where $\partial v(x)$ is the subdifferential\footnote{Let $g:I\to\R$ be a convex function on an open interval $I\subset\R$. The {\bf \textit{subdifferential}} of $g$ at $x\in I$ defined by $\partial g(x)=\{h\in \R:g(y)\ge g(x)+ h (y-x),\ \forall y\in I\}$. Moreover, $g$ is differentiable at $x$ if and only if the subdifferential is singleton set, that is, $\partial g(x)=\{g'(x)\}.$   } of $v$ at $x$.
\end{proposition}
\begin{proof} [{\bf Proof of Proposition \ref{prop-D[phi]}}] 
Let us prove Proposition \ref{prop-D[phi]} in the two spatial regions: $(-\infty,\xi_2+1]$ and $[\xi_2+1,+\infty)$, respectively.
\par\medskip\noindent \# {\bf Proof of I. Take $x\in(-\infty,\xi_2+1]$.} 
For any $B_1>1$, since $v$ is decreasing and $J$ is symmetric, we write 
 \begin{equation*}
\begin{aligned}
\mathcal{D}[v](x)
&\ge \frac{1}{2}\int_{|z|\le B_1}[v(x+z)+v(x-z)-2v(x)]J(z) \dd z  +\int_{z\ge B_1}[v(x+z)-v(x)]J(z) \dd z \\
&:=I_1+I_2.
\end{aligned}
 \end{equation*}
\par 
To estimate $I_1$, we first claim that
\begin{equation}\label{eq:local-symmetric-bound}
v(x-z)+v(x+z)-2v(x)\ge -L_{x,B_1}z^2,\qquad |z|\le B_1,
\end{equation}
where
 \begin{equation*}
L_{x,B_1}:=\sup_{y\in[x-B_1,x+B_1]\setminus\Omega}[v''(y)]^-.
 \end{equation*}
Indeed, since $x-B_1<\xi_2$, $\Omega\subset[\xi_2+1,\infty)$,
and $v$ is convex on $[\xi_2,\infty)$, we have
 \begin{equation*}
L_{x,B_1}=\sup_{y\in[x-B_1,\min\{x+B_1,\xi_2\}]}
[v''(y)]^-<\infty.
 \end{equation*}
Define
 \begin{equation*}
V(y)
:=
v(y)+\frac{L_{x,B_1}}2(y-x)^2,
\qquad y\in[x-B_1,x+B_1].
 \end{equation*}
By the definition of $L_{x,B_1}$,
 \begin{equation*}
V''(y)=v''(y)+L_{x,B_1}\ge0
 \end{equation*}
on $[x-B_1,\min\{x+B_1,\xi_2\}]$. Hence $V$ is convex
there. If $x+B_1>\xi_2$, then the convexity of $v$ on
$[\xi_2,\infty)$ implies that $V$ is also convex on
$[\xi_2,x+B_1]$. Moreover, since
$\Omega\subset[\xi_2+1,\infty)$, the function $V$ is
$C^1$ at $\xi_2$. Therefore, $V$ is convex on the whole interval
$[x-B_1,x+B_1]$. It follows from the midpoint inequality that
 \begin{equation*}
2V(x)\le V(x-z)+V(x+z),
\qquad |z|\le B_1.
 \end{equation*}
Using
 \begin{equation*}
V(x)=v(x),
\qquad
V(x\pm z)
=
v(x\pm z)+\frac{L_{x,B_1}}2z^2,
 \end{equation*}
we obtain \eqref{eq:local-symmetric-bound}.
\par
For $I_1$, using \eqref{eq:local-symmetric-bound}, we get 
 \begin{equation*}
\begin{aligned}   
I_1\ge-\frac{L_{x,B_1}}{2}\int_{-B_1}^{B_1} z^2J(z) \dd z  \ge -\left(\j_1+\int_1^{B_1}z^2J(z) \dd z \right) \sup_{y\in[x-B_1,x+B_1]\setminus \Omega} [v''(y)]^-.
\end{aligned}
 \end{equation*} 
For $I_2$, using $v\ge 0$ on $\R$,
 \begin{equation*}
I_2\ge -v(x)\int_{B_1}^{+\infty}J(z) \dd z .
 \end{equation*}
As a result,
 \begin{equation*}
\D[v](x)\ge -v(x) \int_{B_1}^{+\infty}J(z) \dd z - \left(\j_1+\int_1^{B_1}z^2J(z) \dd z \right)\sup_{y\in[x-B_1,x+B_1]\setminus \Omega} [v''(y)]^-,
 \end{equation*}
for all $x\in (-\infty,\xi_2+1]$.
\medskip
\par
\noindent\# {\bf Proof of II. Take $x\in[\xi_2+1,+\infty)$.}
By the definition of $\mathcal{D}[\cdot]$, we write 
 \begin{equation*}
\begin{aligned}
\mathcal{D}[v](x)&=\int_{1}^{+\infty}[v(x+z)-v(x)]J(z) \dd z 
+P.V.\int_{-1}^{1}[v(x+z)-v(x)]J(z) \dd z \\
&\qquad\qquad\qquad+\int_{-\infty}^{-1}[v(x+z)-v(x)]J(z) \dd z :=II_1+II_2+II_3.
\end{aligned}
 \end{equation*}
\par 
\#\# Let us estimate $II_1$.  Since $v$ is convex on $[\xi_2,+\infty)$, it follows from the definition of the subdifferential at $x$ that
 \begin{equation*}
v(x+z)-v(x)\ge yz,
 \end{equation*}
for all $z\ge 0$ and $y\in\partial v(x)$.
As a result, for any $B_2>1$, we arrive at 
 \begin{equation*}
\begin{aligned}
II_1&=\int_{1}^{B_2}[v(x+z)-v(x)]J(z) \dd z +\int_{B_2}^{+\infty}[v(x+z)-v(x)]J(z) \dd z \\
&\ge y\int_1^{B_2}zJ(z) \dd z -\frac{\j_0}{2sB_2^{2s}} v(x),
\end{aligned}
 \end{equation*}
for all $y\in\partial v(x)$.
\par
\#\# Let us estimate $II_2$. 
By the convexity of $v$ on $[\xi_2,+\infty)$, for all $x\in[\xi_2+1,+\infty)$ and $z\in[-1,1]$, we have 
 \begin{equation*}
\frac{v(x+z)+v(x-z)}{2}\ge v(x).
 \end{equation*}
It then follows from the symmetry of $J$ that 
 \begin{equation*}
II_2= \frac{1}{2}\int_{-1}^{1}[v(x-z)+v(x+z)-2v(x)]J(z) \dd z \ge 0.
 \end{equation*}
\par \#\# Let us estimate $II_3$. For $x\in[\xi_2+1,+\infty)$, since $v$ is decreasing on $\R$ and $v\equiv\Theta$ on $(-\infty,\xi_1]$,   using $x-\xi_1\ge \xi_2-\xi_1+1\ge R_0+1$ and \Cref{kernel_hypothesis}, we have
 \begin{equation*}
II_3\ge \int_{-\infty}^{\xi_1-x}[v(x+z)-v(x)]J(z) \dd z \ge [\Theta-v(x)]\int_{-\infty}^{\xi_1-x}J(z) \dd z \ge \frac{\Theta-v(x)}{2s\j_0[x-\xi_1]^{2s}}.
 \end{equation*}
As a result, for all $x\in [\xi_2+1,+\infty)$, we have
     \begin{equation*}
    \mathcal{D}[v](x)\ge y\int_1^{B_2}zJ(z) \dd z -\frac{\j_0v(x)}{2sB_2^{2s}}+ \frac{\Theta-v(x)}{2s\j_0[x-\xi_1]^{2s}},
     \end{equation*}
    for all $y\in \partial v(x)$.
\end{proof}

\subsection{Estimates from above}
We list and prove below some useful lemmas to estimate the operator $\mathcal{D}$ acting on a piecewise function which is $\mathscr{C}^{1,1}_{\textrm{\textrm{loc}}}$.

\begin{lemma}\label{lemma-c11}
Denote $\delta^2_zv:=v(\cdot+z)+v(\cdot-z)-2v$ for $z\in \R$. 
Let $\eta\in\R$ and $\rho>0$. Suppose that
 \begin{equation*}
v(x)=
\begin{cases}
v_l(x),& x\le \eta\\[1mm]
v_r(x),& x\ge \eta,
\end{cases}
 \end{equation*}
where $v_l\in \mathscr{C}^{1,1}([\eta-\rho,\eta])$ and $v_r\in \mathscr{C}^{1,1}([\eta,\eta+\rho])$ are such that $v_l(\eta)=v_r(\eta)$
and $v_l'(\eta)=v_r'(\eta)$.  Then $v\in \mathscr{C}^{1,1}([\eta-\rho,\eta+\rho])$. Moreover, for every $x\in [\eta-\rho,\eta+\rho]$ and $z\in\mathbb R$ such that $x\pm|z|\in [\eta-\rho,\eta+\rho],$ one has
 \begin{equation*}
\left|\delta^2_zv(x)\right|
\le L_\eta z^2,
 \end{equation*}
where
$L_\eta:=\max\{
\operatorname{Lip}_{[\eta-\rho,\eta]}
v_l',
\operatorname{Lip}_{[\eta,\eta+\rho]}
v_r'\}$.
\end{lemma}

\begin{proof}[{\bf Proof of \Cref{lemma-c11}}]
 We first
show that the derivative of the piecewise function $v$ is Lipschitz
continuous on the whole interval $[\eta-\rho,\eta+\rho]$. On each side of $\eta$, the claim follows directly from the definition of $L_\eta$. It remains to consider points lying on opposite sides of
the junction. 
\par Let $x<\eta<y.$ Since
$v_l'(\eta)=v_r'(\eta),$ we have
 \begin{equation*}
\begin{aligned}
|v'(y)-v'(x)|
&=
|v_r'(y)-v_l'(x)|  \\
&\le
|v_r'(y)-v_r'(\eta)|
+
|v_l'(\eta)-v_l'(x)|  \\
&\le
\operatorname{Lip}_{[\eta-\rho,\eta]}
v_r'(y-\eta)+
\operatorname{Lip}_{[\eta,\eta+\rho]}
v_l'(\eta-x)  \\
&\le
L_\eta (y-x).
\end{aligned}
 \end{equation*}
Hence for all $x,y\in[\eta-\rho,\eta+\rho]$,
 \begin{equation*}
|v'(y)-v'(x)|\le L_\eta|y-x|.
 \end{equation*}
Thus,
 \begin{equation*}
v\in \mathscr{C}^{1,1}([\eta-\rho,\eta+\rho])
\quad\text{and}\quad \operatorname{Lip}_{[\eta-\rho,\eta+\rho]}v'\le L_\eta.
 \end{equation*}

Now let $x\in[\eta-\rho,\eta+\rho]$ and $z\in\mathbb R$ be such that $x-|z|$ and $x+|z|$ both belong to $[\eta-\rho,\eta+\rho].$
Since $\delta_z^2v=\delta_{|z|}^2v$, it suffices to consider $z>0$.
Using the fundamental theorem of calculus, we write
 \begin{equation*}
v(x+z)-v(x)
=
\int_0^z v'(x+r)\,dr
\quad\text{and}\quad
v(x)-v(x-z)
=
\int_0^z v'(x-r)\,dr.
 \end{equation*}
Therefore, by the Lipschitz continuity of $v'$,
 \begin{equation*}
|\delta_z^2v(x)|
\le
\int_0^z
|v'(x+r)-v'(x-r)|\,dr \le
\int_0^z 2L_\eta r\,dr =
L_\eta z^2.
 \end{equation*}
This completes the proof.
\end{proof}

\begin{remark}
If the two pieces are piecewise $\mathscr{C}^2$, the same statement follows with
 \begin{equation*}
L_\eta=
\max\left\{
\|v_l''\|_{L^\infty(\eta-\rho,\eta)},
\|v_r''\|_{L^\infty(\eta,\eta+\rho)}
\right\}.
 \end{equation*}
\end{remark}

\begin{lemma}\label{D-bounded}

Assume that $J$ satisfies Hypothesis \ref{kernel_hypothesis} with some $s>0$. Let $v\in \mathscr{L}^\infty(\R)$ and suppose that for $\Omega\subset \R$, 
\begin{equation*} [v]_{2,\infty;\Omega} := \sup_{x\in\Omega,\ 0<|z|\le1} \frac{|\delta^2_zv(x)| }{z^2} <\infty. \end{equation*}
Then $\mathcal{D}[v]\in \mathscr{L}^\infty(\Omega)$, and 
 \begin{equation*}
\|\mathcal{D}[v]\|_{\mathscr{L}^\infty(\Omega)}\le \j_1[v]_{2,\infty;\Omega}+\frac{2\j_0}{s}\|v\|_{\mathscr{L}^\infty(\R)}.
 \end{equation*}
\end{lemma}
\begin{proof}[{\bf Proof of \Cref{D-bounded}}]
    By the definition of $\mathcal{D}[\cdot]$, we write 
     \begin{equation*}
    \mathcal{D}[v](x)=\left(P.V.\int_{|y|\le 1}+\int_{|y|\ge 1}\right)[v(x-y)-v(x)]J(y)dy:=I_1+I_2.
     \end{equation*}
    Using the symmetry of $J$, one has 
     \begin{equation*}
    |I_1|\le \frac{1}{2}\int_{|y|\le 1}|\delta^2_yv(x)|J(y)dy=\frac{1}{2}\int_{|y|\le 1}\frac{|\delta^2_yv(x)|}{y^2}y^2J(y)dy\le \j_1[v]_{2,\infty;\Omega}.
     \end{equation*}
    It follows from $J(y)\le \j_0|y|^{-1-2s}$ for $|y|\ge 1$ that 
     \begin{equation*}
    |I_2|\le \frac{2\j_0}{s}\|v\|_{\mathscr{L}^\infty(\R)}.
     \end{equation*} 
    Therefore,
     \begin{equation*}
\|\mathcal{D}[v]\|_{\mathscr{L}^\infty(\Omega)}\le \j_1[v]_{2,\infty;\Omega}+\frac{2\j_0}{s}\|v\|_{\mathscr{L}^\infty(\R)}.
 \end{equation*}
    \end{proof}

\begin{proposition}\label{gen-upper-bound-of-D}
 Assume $s\in\left(0,\frac{1}{2}\right)$. For some $A\in\R$, let $v\in \mathscr{C}^{1,1}_{\textrm{\textrm{loc}}}([A,\infty))\cap \mathscr{L}^\infty(\R)$ be non-increasing and non-negative on $\R$, and convex on $[A,+\infty)$.
Then for $x\ge A+1$, we have
     \begin{equation*}
    \D[v](x)\le \j_1 \mathcal{Q}[v](x)+
          \frac{\mathcal{C}_{0}}{(x-A)^{2s}},
     \end{equation*}
    where 
    $\mathcal{Q}[v](x):=\sup_{0<|z|\le1}\frac{|\delta^2_zv(x)|}{z^2}$
    and $\mathcal{C}_{0}:= (1-2s)^{-1}(2s)^{-1}\j_0\|v\|_{\mathscr{L}^\infty(\R)}.$
   
\end{proposition}
\begin{proof}[{\bf Proof of \Cref{gen-upper-bound-of-D}}]
By the definition of $\D[\cdot]$, since $v$ is non-increasing, we write 
 \begin{equation*}
\begin{aligned}
\D[v](x)&\le P.V.\int_{-1}^1[v(x-y)-v(x)]J(y)dy+\int_{1}^{x-A}[v(x-y)-v(x)]J(y)dy\\
&\qquad\qquad+\|v\|_{\mathscr{L}^\infty(\R)}\int_{x-A}^{+\infty}J(y)dy:=I_1+I_2+I_3.
\end{aligned}
 \end{equation*}
 \par {\bf \# Let us estimate $I_1$.} It follows from $J(\cdot)=J(-\cdot)$ that
 \begin{equation*}
I_1=\frac{1}{2}\int_{-1}^{1}[v(x-y)+v(x+y)-2v(x)]J(y)dy\le \j_1 \mathcal{Q}[v](x).
 \end{equation*}
   \par {\bf \# Let us estimate $I_2$.}  For $1\le y\le x-A$, the convexity of $v$ on $[A,\infty)$ gives 
    \begin{equation*}
   v(x-y)-v(x)\le \frac{v(A)-v(x)}{x-A}y\le \frac{\|v\|_{\mathscr{L}^\infty(\R)}}{x-A}y.
    \end{equation*}

   And thus, using $J(y)\le \j_0y^{-1-2s}$ for $y\ge 1$,
 \begin{equation*}
I_2\le \frac{\|v\|_{\mathscr{L}^\infty(\R)}\j_0}{x-A}\int_{1}^{x-A}y^{-2s}dy\le
    \j_0(1-2s)^{-1}\|v\|_{\mathscr{L}^\infty(\R)}(x-A)^{-2s}.
 \end{equation*}
   \par {\bf \# Let us estimate $I_3$.} Using $J(y)\le \j_0y^{-1-2s}$ for $y\ge 1$ again,
    \begin{equation*}
   I_3\le \frac{\j_0\|v\|_{\mathscr{L}^\infty(\R)}}{2s(x-A)^{2s}}.
    \end{equation*}
    Therefore, combining the estimations of $I_1$, $I_2$ and $I_3$, it complete the proof.
\end{proof}

\section{The case \texorpdfstring{$s < \frac{1}{2}$}{s<1/2}}

\subsection{Definition and properties of a suitable subsolution}
To derive a lower bound, we look for a function $\underline u:\mathbb{R}^+\times \mathbb{R}\to [0,\theta_0]$ satisfying 
 \begin{equation*}
\partial_t\underline u-\D[\underline u]-f(\underline u)\le 0\quad\text{and}\quad \underline u\le \theta_0,
 \end{equation*}
where $\theta_0\in (\theta,1]$ is given by \Cref{hy-initial-data}.
Since $f(z)$ equals zero for $z\in[0,\theta]$, one may construct a function $\underline u$ with values close to zero such that
\begin{equation*}
\partial_t\underline u-\D[\underline u]\le 0.
\end{equation*}

In view of part II of Proposition \ref{prop-D[phi]}, for any $\Theta\in(\theta,\theta_0)$, it is reasonable to seek $\underline u$ satisfying  
\begin{equation}
\label{right-b}
\partial_t \underline u\lesssim [x-X_\Theta(t;\underline u)]^{-2s}\quad\text{for all } x\ge X_{\lambda}(t;\underline u),
\end{equation}
where $\lambda\in(0,\theta)$ is small. Based on such an observation, we can further construct an explicit global subsolution.

\par 
To smoothly connect the constant plateau with the convex leading edge, we first introduce a suitable monotone mollifier, which will be used in both cases $s<\frac{1}{2}$ and $s=\frac{1}{2}$.
\begin{lemma}\label{lem:smooth-truncation}
Let $0<\theta_1<\theta_2\le 1$ and $\theta_1<\Theta< 1.$
Then there exists $m\in \mathscr{C}^2([0,1];[0,\Theta])$ such that
 \begin{equation*}
m(y)=
\begin{cases}
y, & y\in[0,\theta_1],\\[2mm]
\Theta, & y\in[\theta_2,1],
\end{cases}
\quad \text{and}\quad
0\le m'\le M,
\qquad
m''\ge-C
\quad\text{on }[0,1],
 \end{equation*}
for some constants $M,C>0$ depending only on
$\theta_1,\theta_2$, and $\Theta$.
\end{lemma}

\begin{proof}[{\bf Proof of \Cref{lem:smooth-truncation}}]
Set
 \begin{equation*}
q:=\frac{\Theta-\theta_1}{\theta_2-\theta_1}>0.
 \end{equation*}
We first construct a function $G_q\in \mathscr{C}^1([0,1])$ satisfying
\begin{equation}\label{eq:Gq-properties}
G_q\ge0,\qquad
G_q(0)=1,\qquad
G_q(1)=0,\qquad
G_q'(0)=G_q'(1)=0,
\qquad
\int_0^1G_q(z)\, \dd z =q.
\end{equation}

If $0<q\le\frac12$, let $p:=\frac2q-2\ge2$ and define
$G_q(z):=(1-z)^p(1+pz)$.
Then
 \begin{equation*}
G_q'(z)=-p(p+1)z(1-z)^{p-1}\le0\quad\text{and}\quad
\int_0^1G_q(z)\, \dd z 
=
\frac2{p+2}
=q.
 \end{equation*}
Hence \eqref{eq:Gq-properties} holds.

If $q>\frac12$, define
 \begin{equation*}
H(z):=1-3z^2+2z^3,
\qquad
\psi(z):=30z^2(1-z)^2,
 \end{equation*}
and set
 \begin{equation*}
G_q(z):=
H(z)+\left(q-\frac12\right)\psi(z).
 \end{equation*}
Since
 \begin{equation*}
H,\psi\ge0,\qquad
\int_0^1H(z)\, \dd z =\frac12,
\qquad
\int_0^1\psi(z)\, \dd z =1,
 \end{equation*}
and
 \begin{equation*}
H(0)=1,\quad H(1)=0,\quad
H'(0)=H'(1)=0,
 \end{equation*}
while
 \begin{equation*}
\psi(0)=\psi(1)=\psi'(0)=\psi'(1)=0,
 \end{equation*}
the properties in \eqref{eq:Gq-properties} again follow.

We now define
 \begin{equation*}
m(y):=
\begin{cases}
y,
&0\le y\le \theta_1,\\[3mm]
\theta_1+(\theta_2-\theta_1)
\displaystyle
\int_0^{\frac{y-\theta_1}{\theta_2-\theta_1}}G_q(z)\, \dd z ,
&\theta_1<y<\theta_2,\\[5mm]
\Theta,
&\theta_2\le y\le1.
\end{cases}
 \end{equation*}
Since
 \begin{equation*}
(\theta_2-\theta_1)\int_0^1G_q(z)\, \dd z 
=
(\theta_2-\theta_1)q
=
\Theta-\theta_1,
 \end{equation*}
the three pieces agree continuously at $\theta_1$ and $\theta_2$.

For $y\in(\theta_1,\theta_2)$,
 \begin{equation*}
m'(y)
=
G_q\left(\frac{y-\theta_1}{\theta_2-\theta_1}\right)
\quad\text{and}\quad m''(y)
= \frac1{\theta_2-\theta_1}
G_q'\left(\frac{y-\theta_1}{\theta_2-\theta_1}\right).
 \end{equation*}
By \eqref{eq:Gq-properties}, these derivatives match the derivatives of the
affine and constant pieces at $\theta_1$ and $\theta_2$. Hence $m\in\mathscr{C}^2([0,1]).$
\par
Moreover, $G_q\ge0$, so $m$ is nondecreasing. Since $m(0)=0$ and $m(1)=\Theta$,
we have
 \begin{equation*}
0\le m\le\Theta
\quad\text{on }[0,1].
 \end{equation*}
Finally, since $G_q,G_q'\in C([0,1])$, the constants
 \begin{equation*}
M:=\max_{z\in[0,1]}G_q(z)\quad \text{and}\quad C:=
\frac1{\theta_2-\theta_1}
\max_{z\in[0,1]}[G_q'(z)]^-
 \end{equation*}
are finite. Therefore,
 \begin{equation*}
0\le m'\le M,
\qquad
m''\ge-C
\quad\text{on }[0,1].
 \end{equation*}
This completes the proof.
\end{proof}

For any $\Theta\in(\theta,\theta_0)$, using \Cref{lem:smooth-truncation}, one may find $m\in \mathscr{C}^2([0,1];[0,\Theta])$ satisfying 
\par
 \begin{equation*}
m(y)=\begin{cases}
    y, &y\in \left[0,\frac{\Theta+\theta}{2}\right],\\[5pt]
    \Theta, &y\in\left[\frac{1+\theta_0}{2},1\right],
\end{cases}
 \end{equation*}
and 
\begin{equation}
    \label{smf}
0\le m'\le M_\theta \quad\text{and}\quad m''\ge -C_\theta\quad \text{on }[0,1],
\end{equation}
for some $M_\theta$ and $C_\theta>0$ depending only on $\theta$, $\Theta$ and $\theta_0$. Denote 
 \begin{equation*}
\kappa:=\frac{2s}{\Theta+\theta}\varepsilon^{-1}e^{-\frac{1}{\varepsilon}}\quad\text{and}\quad\theta_\varepsilon :=e^{-\frac{1}{\varepsilon}},
 \end{equation*}
and choose $\varepsilon\in\left(0,\varepsilon_0\right)$ with $\varepsilon_0:=\min\left\{1,\left(\ln\frac{1}{\theta}\right)^{-1},\frac{1}{2}\left(\ln\left(1+\frac{2s}{\Theta+\theta}\right)\right)^{-1}\right\}$ such that $$\kappa<1\quad\text{and}\quad\theta_\varepsilon<\theta.$$
For $\lambda\in (0,1]$, denote
 \begin{equation*}
X_\lambda(t)=\left\{
\begin{aligned}
&\left(1+\kappa\left(\frac{\theta_\varepsilon }{\lambda }\right)^\frac{1}{2s}\right)t^\frac{1}{2s}, &\lambda\in(0,\theta_\varepsilon ],\\
    &\left[\kappa+\left(\frac{\Theta+\theta }{ \Theta+\theta + (\lambda-\theta_\varepsilon)\varepsilon}\right)^{\frac{1}{\varepsilon^2}}\right]t^\frac{1}{2s},&\lambda\in(\theta_\varepsilon ,1].
\end{aligned}\right.
 \end{equation*}
For $t\gg1$, we define
 \begin{equation*}
\underline u(t,x):=\begin{cases}
\Theta, & x\le X_{\frac{1+\theta_0}{2}}(t),\\
   m\circ  \phi_m(t,x), &X_{\frac{1+\theta_0}{2}}(t)\le x\le X_{\theta_\varepsilon }(t),\\
  \phi_r(t,x), & x\ge X_{\theta_\varepsilon }(t),
\end{cases} 
 \end{equation*}
where
 \begin{equation*}
\phi_m(t,x)=\frac{\Theta+\theta}{\varepsilon} \left[\left(\frac{t^\frac{1}{2s}}{x-\kappa t^\frac{1}{2s}}\right)^{\varepsilon^2}-1\right]+\theta_\varepsilon \qquad \text{and} \qquad \phi_r(t,x)=\frac{\kappa^{2s}\theta_\varepsilon t}{(x-t^\frac{1}{2s})^{2s}}.
 \end{equation*}
 Observe that $X_{\Theta}(t;\underline{u})=X_{\frac{1+\theta_0}{2}}(t)$ and $\lim_{t\to\infty}X_{\lambda_1}(t)-X_{\lambda_2}(t)=\infty$ for any $0<\lambda_1<\lambda_2\le\frac{1+\theta_0}{2}$. 
 Since $\Theta$ can be chosen arbitrarily in $(\theta,\theta_0)$, the above construction, together with the propagation result for the remaining levels, yields the desired estimate for every $\lambda\in(0,1)$.

A summary of the construction is displayed in \Cref{fig:diagram-subsolution}.
\begin{figure}[H]
     \centering
\begin{tikzpicture}
\pgfmathsetmacro{\t}{e+0.5}
\pgfmathsetmacro{\et}{0.05/ ln(\t )}
\pgfmathsetmacro{\Xt}{\t * ln(\t)}
\pgfmathsetmacro{\pgamma}{0.8/(1-e^(-0.05/4))}  
\pgfmathsetmacro{\tlambda}{\pgamma*e^(-0.05/4)}  
\pgfmathsetmacro{\Rt}{\Xt /(1-( \Xt )^(-1/2))-\Xt-( \t )^(1/4)} 
\pgfmathsetmacro{\Xthetathree}{(\Xt )/(1-(\Xt )^(-1/2))}
\pgfmathsetmacro{\Xthetatwo}{\Xt +\t ^ ( ln(\pgamma /(0.8+\tlambda -0.3 ))/0.05) + \Rt}

\pgfmathsetmacro{\Xthetaone}{\Xt +\t ^ ( ln(\pgamma /(0.4+\tlambda -0.3))/0.05) + \Rt}
    \begin{axis}[
    axis lines = left,
    ymin = 0,
    ymax = 0.9,
    xtick = {\Xthetatwo,\Xthetathree},
    xticklabels = {$X_{\frac{1+\theta_0}{2}}(t)$,$X_{\theta_\varepsilon }(t)$},
     ytick = {0.3-0.01,0.4,0.6,0.8},
    yticklabels = {$\theta_\varepsilon $,$\frac{\Theta+\theta}{2}$,$\Theta$,$\theta_0$},
    width = 8.5cm,   
    height = 5.5cm,
    ]

      \addplot[ 
    color = black,
    domain = 0.98*\Xthetatwo : \Xthetatwo,
    samples = 100,
    line width =1.5,
    ] {0.6};
    
       \addplot[ 
    color = red,
   domain = \Xthetatwo : \Xthetaone,
    samples = 100,
    line width =1.5,
    ] {(0.6-0.4)*sin(deg(pi/2*((\pgamma / (x-\Rt -\Xt )^(\et)-\tlambda + 0.3) -0.4)/(0.8-0.4)))+0.4};
       \addplot[ 
    color = red,
   domain = \Xthetaone : \Xthetathree,
    samples = 100,
    line width =1.5,
    ] {(\pgamma / (x-\Rt -\Xt )^(\et)-\tlambda + 0.3)};
    
    \addplot[ 
    color = black,
   domain = \Xthetathree : 1.06 * \Xthetathree,
    samples = 100,
    line width =1.5,
    ] {0.29 /(x^1.3+1-\Xthetathree^1.3)};

\draw[dashed] (axis cs:\Xthetatwo,0) -- (axis cs:\Xthetatwo,0.6);
\draw[dashed] (axis cs:\Xthetathree,0) -- (axis cs:\Xthetathree,0.3-0.01);
\draw[dashed] (axis cs:0,0.3-0.01) -- (axis cs:\Xthetathree,0.3-0.01);
\pgfmathsetmacro{\Xthetamid}{(\Xthetatwo + \Xthetathree)/2}
\node[rotate=-34] at (axis cs: \Xthetamid, 0.5) [above] {\color{red}{$m\circ  \phi_m$}};

\pgfmathsetmacro{\Xthree}{1.03*\Xthetathree}
\node[rotate=-2] at (axis cs: \Xthree, 0.2) [above] {$\phi_r$};

\end{axis}
\end{tikzpicture}
    \caption{Graph of the function $\underline u(t,\cdot)$ for some time $t$.}
    \label{fig:diagram-subsolution}
 \end{figure}
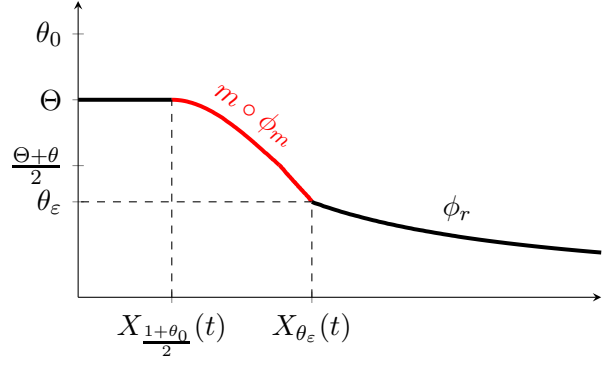

\begin{lemma}\label{lemma-s<1/2}
The function $t \mapsto\underline{u}(t,\cdot)$ is nondecreasing in $t\in \R^+$, $x \mapsto\underline{u}(x,\cdot)$ is nonincreasing in $x\in \R$ and convex in $x\in \left[X_{\frac{\Theta+\theta}{2}}(t),\infty\right)$.
Moreover,  for all $t>0$, $\underline{u}(t,\cdot)\in  \mathscr{C}^1(\R)\cap\mathscr{C}^2(\R\setminus\{X_{\theta_\varepsilon}(t)\})$.
\end{lemma}

\begin{proof}[{\bf Proof of Lemma \ref{lemma-s<1/2}}]
   Some direct calculations show that
\begin{equation}
   \label{dwb2}
\left\{\begin{aligned}
&\partial_t  \phi_m(t,x) =\frac{ \varepsilon (\Theta+\theta +(\phi_m-\theta_\varepsilon) \varepsilon)}{2s t}\left(1+\frac{\kappa t^{\frac{1}{2s}}}{x-\kappa t^\frac{1}{2s}}\right) >0,
\\
&\partial_x  \phi_m(t,x) =-\frac{\varepsilon (\Theta + \theta + (\phi_m-\theta_\varepsilon) \varepsilon)}{x-\kappa t^\frac{1}{2s}}<0,\\
&\partial_{xx}  \phi_m(t,x)=\frac{\varepsilon(1+\varepsilon^2)(\Theta + \theta + (\phi_m-\theta_\varepsilon) \varepsilon)}{\left(x-\kappa t^\frac{1}{2s}\right)^2}\ge 0,
\end{aligned}\right.
\end{equation}
and
\begin{equation}
   \label{dphir}
\left\{
\begin{aligned}
    &\partial_t \phi_r(t,x)=\frac{\kappa^{2s}\theta_\varepsilon }{(x-t^\frac{1}{2s})^{2s}}\left(1+\frac{ t^\frac{1}{2s}}{x-t^\frac{1}{2s}}\right)>0,\\
    &\partial_x \phi_r(t,x)=-\frac{2s\kappa^{2s}\theta_\varepsilon t}{(x-t^\frac{1}{2s})^{2s+1}}<0,\\
    &\partial_{xx} \phi_r(t,x)=\frac{2s(2s+1)\kappa^{2s}\theta_\varepsilon t}{(x-t^\frac{1}{2s})^{2s+2}}>0.
\end{aligned}
\right.
\end{equation}
As a result, we get that the functions $\phi_m$ and $\phi_r$ are increasing in $t$, and decreasing and convex in $x$. 
\par
Using $X_{\theta_\varepsilon}(t)=(1+\kappa)t^\frac{1}{2s}$, notice that
 \begin{equation*}
\partial_x \phi_m(t,X_{\theta_\varepsilon}(t))=- \frac{(\Theta+\theta)\varepsilon}{t^{\frac{1}{2s}}} \quad\text{and}\quad\partial_x \phi_r(t,X_{\theta_\varepsilon}(t))=-\frac{2s \theta_\varepsilon}{\kappa t^{\frac{1}{2s}}}.
 \end{equation*}
Therefore, using the definition of $\theta_\varepsilon$ and $\kappa$, we have
 \begin{equation*}
\phi_m(t,X_{\theta_\varepsilon }(t))
=  \phi_r(t,X_{\theta_\varepsilon }(t))\quad\text{and}\quad \partial_x \phi_m(t,X_{\theta_\varepsilon }(t))= \partial_x \phi_r(t,X_{\theta_\varepsilon }(t)).
 \end{equation*}
Then, the conclusion follows directly from the definition of $\underline{u}$.
\end{proof}

\par
 Before proving that $\underline u$ is a subsolution to \eqref{ceq}, we list some useful properties of $\underline u$.
\begin{proposition}
\label{prop-phi-subcritical}
 We have the following derivative estimates:
\begin{enumerate}[label=(\roman*)] 
    \item  We have
 \begin{equation*}
\underset{t \to  \infty}{\lim}\sup_{x\in\left[X_{\frac{1+\theta_0}{2}}(t),X_{\theta_\varepsilon}(t)\right]} |\partial_t \underline{u}(t,x)| = 0.
 \end{equation*}
And for all $(t,x)\in\R^+\times\left[X_{\frac{\Theta+\theta}{2}}(t),\infty\right)$, if $\varepsilon\in(0,\varepsilon_0)$, then
 \begin{equation*}
\partial_t \underline{u}(t,x)\le 
 \frac{3}{2}\left(\frac{\Theta+\theta}{2s}+1\right)\varepsilon \left[x-X_{\frac{1+\theta_0}{2}}(t)\right]^{-2s}.
 \end{equation*}
\item We have
 \begin{equation*}
\lim_{t\to\infty}\sup_{x\in\left[X_{\frac{1+\theta_0}{2}}(t),X_{\theta_\varepsilon}(t)\right]}[\partial_{xx}\underline{u}(t,x)]^-=0.
 \end{equation*}
\end{enumerate}
\end{proposition}

\begin{proof}[{\bf Proof of Proposition \ref{prop-phi-subcritical}}]
\textbf{\# Proof of (i).}
\#\# \textbf{Take $x\in\left[X_{\frac{1+\theta_0}{2}}(t),X_{\theta_\varepsilon}(t)\right]$.} Using \eqref{dwb2}, we write
 \begin{equation*}
\partial_t \underline{u}(t,x)=(m'\circ  \phi_m(t,x)) \partial_t \phi_m(t,x)=(m'\circ  \phi_m(t,x)) \frac{ \varepsilon (\Theta+\theta +(\phi_m-\theta_\varepsilon) \varepsilon)}{2st}\left(1+\frac{\kappa t^{\frac{1}{2s}}}{x-\kappa t^\frac{1}{2s}}\right).
 \end{equation*}
It follows from \eqref{smf}  that 
 \begin{equation*}
0\le\partial_t \underline{u}(t,x)\le M_{\theta}\frac{ \varepsilon (\Theta+\theta +(\phi_m-\theta_\varepsilon) \varepsilon)}{2st}\left(1+\frac{\kappa t^{\frac{1}{2s}}}{x-\kappa t^\frac{1}{2s}}\right).
 \end{equation*}
Thus, for all $x\in\left[X_{\frac{1+\theta_0}{2}}(t),X_{\theta_\varepsilon}(t)\right]$,
 \begin{equation*}
|\partial_t \underline{u}(t,x)|\lesssim t^{-1} \xrightarrow[t\to\infty]{} 0.
 \end{equation*} 
 \par
\#\# \textbf{Take $x\in\left[X_{\frac{\Theta+\theta}{2}}(t),X_{\theta_\varepsilon }(t)\right]$.} 
Using $X_{\lambda}(t)= \left[\kappa+\left(\frac{\Theta+\theta}{\Theta+\theta + (\lambda-\theta_\varepsilon)\varepsilon}\right)^{\frac{1}{\varepsilon^2}}\right]t^\frac{1}{2s}$ for $\lambda\in \left[\theta_\varepsilon,\frac{\Theta+\theta}{2}\right]$ and $m'\equiv 1$ on $(0,\frac{\Theta+\theta}{2})$, we have
 \begin{equation*}
\left[x-X_{\frac{1+\theta_0}{2}}(t)\right]^{2s}\partial_t\underline{u}(t,x)\le \frac{(\Theta+\theta)\left(\varepsilon+2\right)\varepsilon}{4s}\left(1+\kappa\left(1+\frac{\varepsilon}{2}\right)^\frac{1}{\varepsilon^2}\right).
 \end{equation*}
Notice that
 \begin{equation*}
\left(1+\frac{\varepsilon}{2}\right)^\frac{1}{\varepsilon^2}=  \exp\left\{\frac{\ln\left(1+\frac{\varepsilon}{2}\right)}{\varepsilon^2}\right\}\le e^{\frac{1}{2\varepsilon}}.
 \end{equation*}
Thus, by $e^{-\frac{1}{2\varepsilon}}\le \varepsilon<1$ and 
$\kappa=\frac{2s}{\Theta+\theta}\varepsilon^{-1}e^{-\frac{1}{\varepsilon}}$,
 \begin{equation*}
\left[x-X_{\frac{1+\theta_0}{2}}(t)\right]^{2s}\partial_t\underline{u}(t,x)
\le \frac{3(\Theta+\theta)}{4s}\varepsilon+\frac{3}{2} e^{-\frac{1}{2\varepsilon}}
\le \frac{3}{2}\left(\frac{\Theta+\theta}{2s}+1\right)\varepsilon.
 \end{equation*}
Therefore, we arrive at 
 \begin{equation*}
\partial_t\underline{u}(t,x)\le \frac{3}{2}\left(\frac{\Theta+\theta}{2s}+1\right)\varepsilon\left[x-X_{\frac{1+\theta_0}{2}}(t)\right]^{-2s}.
 \end{equation*}
\par \#\# \textbf{Take $x\in[X_{\theta_\varepsilon }(t),+\infty)$.} 
Since $X_{\frac{1+\theta_0}{2}}(t)\ge \kappa  t^\frac{1}{2s}$ and $\kappa< 1$ for $\varepsilon\le\varepsilon_0$,
\begin{equation}\label{sc-x-t-x-t}
x-X_{\frac{1+\theta_0}{2}}(t)\le 
t^\frac{1}{2s}+x-(\kappa+1) t^\frac{1}{2s} \le
t^\frac{1}{2s}+\kappa^{-1}\left[x-(\kappa+1) t^\frac{1}{2s}\right] =\kappa^{-1} (x-t^\frac{1}{2s}),
\end{equation}
and thus, we get
 \begin{equation*}
\begin{aligned}
\left[x-X_{\frac{1+\theta_0}{2}}(t)\right]^{2s}\partial_t \underline{u}(t,x)&\le \frac{\kappa^{2s}\theta_\varepsilon \left[x-X_{\frac{1+\theta_0}{2}}(t)\right]^{2s} }{(x-t^\frac{1}{2s})^{2s}}\left(1+\frac{ t^\frac{1}{2s}}{x-t^\frac{1}{2s}}\right)\\
&\le \theta_\varepsilon\left(1+\kappa^{-1} \right) =e^{-\frac{1}{\varepsilon}}+\frac{\Theta+\theta}{2s}\varepsilon \le \left(1+\frac{\Theta+\theta}{2s}\right)\varepsilon .
\end{aligned}
 \end{equation*}
 As a result, for $\varepsilon\le\varepsilon_0$ we achieve 
 \begin{equation*}
\partial_t\underline{u}(t,x)\le \left(1+\frac{\Theta+\theta}{2s}\right)\varepsilon \left[x-X_{\frac{1+\theta_0}{2}}(t)\right]^{-2s}.
 \end{equation*}

\par\textbf{\# Proof of (ii).}
By the definition of $\underline u$, it is easy to get that $\underline u(t,\cdot)$ is convex on $\left[X_{\frac{\Theta+\theta}{2}}(t),\infty\right)$. Some direct calculations show that, for $x\in \left[X_{\frac{1+\theta_0}{2}}(t), X_{\theta_\varepsilon }(t)\right]$,  
 \begin{equation*}
\partial_x\underline u=(m'\circ \phi_m) \partial_x \phi_m\quad \text{and}\quad
\partial_{xx}\underline u=(m''\circ \phi_m) (\partial_x \phi_m)^2+(m'\circ \phi_m) \partial_{xx} \phi_m.
 \end{equation*}
By the convexity of $\phi_m$ and \eqref{smf}, we arrive at
 \begin{equation*}
\partial_{xx}\underline u\ge  -C_\theta (\partial_x  \phi_m)^2.
 \end{equation*}
Thus, for all $x\in \left[X_{\frac{1+\theta_0}{2}}(t), X_{\theta_\varepsilon }(t)\right]$, 
 \begin{equation*}
[\partial_{xx}\underline{u}]^-\le C_\theta (\partial_x  \phi_m)^2\le \frac{(\Theta+\theta+\varepsilon)^2C_\theta \varepsilon^2 }{\left[X_{\frac{1+\theta_0}{2}}(t)-\kappa t^\frac{1}{2s}\right]^2}\lesssim t^{-\frac{1}{s}}\xrightarrow[t\to\infty]{}0.
 \end{equation*}
\end{proof}

By Proposition \ref{prop-D[phi]}, we have the following proposition.
\begin{proposition}\label{prop-D[phi]-s<1/2}
Denote $\varepsilon_1:=\min\left\{\varepsilon_0,\left(\frac{\Theta-\theta}{4\j_0^2(\Theta+\theta)}\right)^{\frac{1}{2s}}\frac{1-2s}{6s},\frac{ (\Theta-\theta)(1-2s)^{2s}(2s)^{1-4s}}{8(\Theta+\theta)\j_0^2}\right\}$.
\begin{itemize}
    \item When $x\in\left(-\infty,X_{\frac{\Theta+\theta}{2}}(t)+1\right]$, for any $\mathcal{C}>0$, there is a time 
$\underline{t}_1$ such that for $t\ge  \underline{t}_1$ we have
     \begin{equation*}
    \mathcal{D}[\underline u](t,x)\ge -\mathcal{C}.
     \end{equation*}
\item When $x\in \left[X_{\frac{\Theta+\theta}{2}}(t)+1,+\infty\right)$, if $\varepsilon\le \varepsilon_1$, then there is a time $\underline{t}_2$ such that for all $t\ge \underline{t}_2$,
     \begin{equation*}
    \mathcal{D}[\underline u](t,x)\ge \frac{\Theta-\theta}{8s\j_0[x-X_{\frac{1+\theta_0}{2}}(t)]^{2s}}.
     \end{equation*}
\end{itemize}
\end{proposition}
\begin{proof}[{\bf Proof of Proposition \ref{prop-D[phi]-s<1/2}}]
Since $X_{\frac{\Theta+\theta}{2}}(t)-X_{\frac{1+\theta_0}{2}}(t)\ge R_0$ and $X_{\theta_\varepsilon}(t)-X_{\frac{\Theta+\theta}{2}}(t)\ge 1$ for $t$ large enough, say $t\ge t_0$, we use \Cref{prop-D[phi]} to prove \Cref{prop-D[phi]-s<1/2} by taking $\xi_1=X_{\frac{1+\theta_0}{2}}(t)$, $\xi_2=X_{\frac{\Theta+\theta}{2}}(t)$ and $\Omega=\{X_{\theta_\varepsilon}(t)\}$.
\par
{\bf \# Take $x\in\left(-\infty,X_{\frac{\Theta+\theta}{2}}(t)+1\right]$.}
 In view of \Cref{prop-D[phi]}, for any $\mathcal{C}>0$, by taking $B_1$ large enough,
  \begin{equation*}
 \int_{B_1}^{\infty}J(z) \dd z \le \frac{\mathcal{C}}{2}.
  \end{equation*}
 In view of \Cref{prop-phi-subcritical}, there is a time 
$\underline{t}_1>t_0$ such that for $t\ge \underline{t}_1$ we have
   \begin{equation*}
\D[\underline u](t,x) \ge -\mathcal{C}.
 \end{equation*}
\par {\bf \# Take $x\in\left[X_{\frac{\Theta+\theta}{2}}(t)+1,\infty\right)$.}
Some direct calculations show that 
 \begin{equation*}
\frac{\underline{u}}{|\partial_x\underline{u}|}=
\begin{cases}
    \dfrac{\phi_m(x-\kappa t^\frac{1}{2s})}{\varepsilon (\Theta+\theta +(\phi_m-\theta_\varepsilon) \varepsilon)}\gtrsim t^\frac{1}{2s}, &x\in\left[X_{\frac{\Theta+\theta}{2}}(t),X_{\theta_\varepsilon }(t)\right],\\[8pt]
    \frac{x-t^\frac{1}{2s}}{2s}\gtrsim  t^\frac{1}{2s}, &x\in[X_{\theta_\varepsilon }(t),+\infty).
\end{cases}
 \end{equation*}
Thus, there is a time $\underline{t}_2$ such that $\frac{1-2s}{2s}\frac{\underline u}{|\partial_x\underline u|}>1$ for all $t\ge \underline{t}_2$.
 By Proposition \ref{prop-D[phi]}, taking $B_2=\frac{1-2s}{2s}\frac{\underline u}{|\partial_x\underline u|}$, we obtain 
 \begin{equation*}
\mathcal{D}[\underline{u}](t,x)\ge -\frac{2\j_0}{(1-2s)^{2s}(2s)^{1-2s}}|\partial_x\underline{u}|^{2s}\underline{u}^{1-2s}+ \frac{\Theta-\theta}{4s\j_0\left[x-X_{\frac{1+\theta_0}{2}}(t)\right]^{2s}}.
 \end{equation*}
For $x\in\left[X_{\frac{\Theta+\theta}{2}}(t),X_{\theta_\varepsilon }(t)\right]$, since $ X_{\frac{1+\theta_0}{2}}(t)\ge \kappa t^\frac{1}{2s}$ and $\phi_m\le \frac{\Theta+\theta}{2}$, we have
 \begin{equation*}
\begin{aligned}
&\frac{2\j_0}{(1-2s)^{2s}(2s)^{1-2s}}\left[x-X_{\frac{1+\theta_0}{2}}(t)\right]^{2s}|\partial_x\underline{u}|^{2s}\underline{u}^{1-2s}\\
&\le\frac{2\j_0}{(1-2s)^{2s}(2s)^{1-2s}}\frac{\varepsilon^{2s}(\Theta + \theta + (\phi_m-\theta_\varepsilon) \varepsilon)^{2s}\phi_m^{1-2s}}{[x-\kappa t^\frac{1}{2s}]^{2s}}\le \frac{3^{2s}\j_0 (\Theta+\theta) }{(1-2s)^{2s}(2s)^{1-2s}}\varepsilon^{2s} \le \frac{\Theta-\theta}{8s\j_0},
\end{aligned}
 \end{equation*}
as long as $\varepsilon\le  \left(\frac{\Theta-\theta}{4\j_0^2(\Theta+\theta)}\right)^{\frac{1}{2s}}\frac{1-2s}{6s}$. 
\par 
For $x\in[X_{\theta_\varepsilon }(t),+\infty)$, it then follows from \eqref{sc-x-t-x-t}  that
 \begin{equation*}
\begin{aligned}    
\frac{2\j_0}{(1-2s)^{2s}(2s)^{1-2s}}\left[x-X_{\frac{1+\theta_0}{2}}(t)\right]^{2s}|\partial_x\underline{u}|^{2s}\underline{u}^{1-2s}
&\le   \frac{2\j_0 \phi_r}{(1-2s)^{2s}(2s)^{1-4s}} \left[\frac{x-X_{\frac{1+\theta_0}{2}}(t)}{x-t^\frac{1}{2s}}\right]^{2s}\\
&\le   \frac{2\j_0 }{(1-2s)^{2s}(2s)^{1-4s}}\frac{\theta_\varepsilon }{\kappa}\\
&\le  \frac{2\j_0 (\Theta+\theta)}{(1-2s)^{2s}(2s)^{2-4s}}\varepsilon \le \frac{\Theta-\theta}{8s\j_0},
\end{aligned}
 \end{equation*}
as long as $\varepsilon \le \frac{ (\Theta-\theta)(1-2s)^{2s}(2s)^{1-4s}}{8(\Theta+\theta)\j_0^2}$. Therefore, if $\varepsilon\le \varepsilon_1$, for all $t\ge \underline{t}_2$,
 \begin{equation*}
\D[\underline u](t,x)\ge \frac{\Theta-\theta}{8s\j_0\left[x-X_{\frac{1+\theta_0}{2}}(t)\right]^{2s}}.
 \end{equation*}

\end{proof}
Now, let us show that $\underline u$ is a subsolution to \eqref{ceq}.
\begin{proposition}\label{prop-subsolution-subcritical}
 Let $\varepsilon^\star: = \min\left\{\frac{\Theta-\theta}{6\j_0(\Theta+\theta+2s)},\varepsilon_1 \right\}$. If $\varepsilon\le \varepsilon^\star$, then there is a time $t^\star>0$ such that for all $t\ge t^\star$, $\underline u(t,\cdot)$ is a subsolution to \eqref{ceq} on $\R$.
\end{proposition}
\begin{proof}[{\bf Proof of Proposition \ref{prop-subsolution-subcritical}}]
 {\bf \# Take $x\in\left(-\infty,X_{\frac{\Theta+\theta}{2}}(t)+1\right]$.}
Denote $\nu:=\inf_{\left[\frac{\Theta+3\theta}{4},\Theta\right]} f>0$. 
Using $X_{\frac{\Theta+3\theta}{4}}(t)-X_{\frac{\Theta+\theta}{2}}(t)\ge 1$ for $t$ large enough, there is $t^\#>0$ such that for all $t\ge t^\#$, 
 \begin{equation*}
\underline{u}(t,x)\ge \underline{u}\left(t,X_{\frac{\Theta+3\theta}{4}}(t)\right)=\frac{\Theta+3\theta}{4},
 \end{equation*}
and then,
 \begin{equation*}
f(\underline{u}(t,x))\ge \nu.
 \end{equation*}
It follows from Proposition \ref{prop-phi-subcritical} that there is a time $\underline{t}_1>0$ such that for all $t\ge \underline{t}_1$,
 \begin{equation*}
\partial_t \underline{u}\le \frac{\nu}{2}.
 \end{equation*}
Therefore, by Proposition \ref{prop-D[phi]-s<1/2}, for all $t\ge \max\{\underline{t}_1,t^\#\}$ we achieve 
 \begin{equation*}
\partial_t \underline u-\D[\underline u]-f(\underline u)\le \frac{\nu}{2}+\frac{\nu}{2}-\nu=0.
 \end{equation*}
\par {\bf \# Take $x\in\left[X_{\frac{\Theta+\theta}{2}}(t)+1,\infty\right)$.} 
By Proposition \ref{prop-phi-subcritical} and Proposition \ref{prop-D[phi]-s<1/2}, we achieve 
 \begin{equation*}
\partial_t \underline u-\mathcal{D}[\underline u]-f(\underline u)
\le \left[\frac{3}{2}\left(\frac{\Theta+\theta}{2s}+1\right)\varepsilon -\frac{\Theta-\theta}{8s\j_0}\right]\left[x-X_{\frac{1+\theta_0}{2}}(t)\right]^{-2s} \le 0,
 \end{equation*}
as long as $\varepsilon \le \varepsilon^\star=\min\left\{\frac{\Theta-\theta}{6\j_0(\Theta+\theta+2s)},\varepsilon_1 \right\}$.

\par Combining the above results on two spatial zones, taking $t^\star:=\max\{\underline{t}_1,\underline{t}_2,t^\#\}$, for $\varepsilon\le \varepsilon^\star$, we therefore prove that $\underline u$ is a subsolution to \eqref{ceq} on $[t^\star,\infty)\times \R$.  

\end{proof}

\subsection{Construction of a suitable supersolution}

For $t\gg 1$, let us construct a piecewise supersolution $\overline{u}$ as follows,
 \begin{equation*}
\overline{u}(t,x)=
\begin{cases}
2, &x\le X_2(t),\\[5pt]
 w_l(t,x), &X_2(t)\le x\le X_{\frac{\theta}{2}}(t), \\[5pt]
w_m(t,x), & X_{\frac{\theta}{2}}(t)\le x\le X_{\frac{\theta}{4}}(t), \\[5pt]
 w_r(t,x), &x\ge X_{\frac{\theta}{4}}(t),
\end{cases}
 \end{equation*}
where $X_\lambda(t)$ is the position of level sets of $\overline{u}$, that is $X_\lambda(t):=X_\lambda(t;\overline{u})$.

 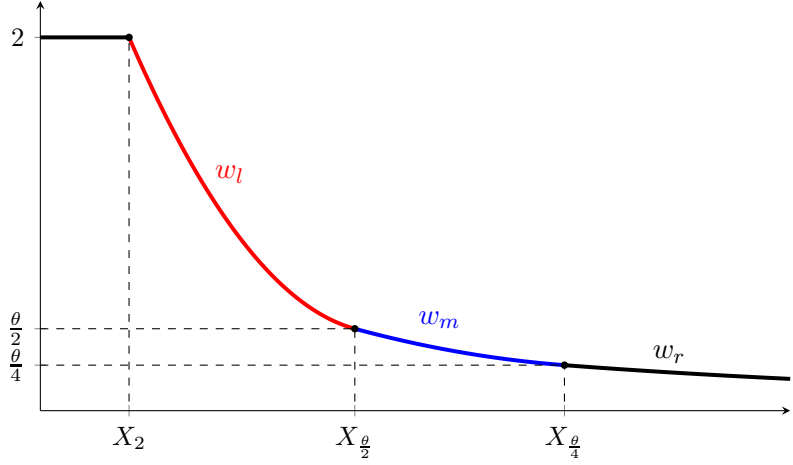
\begin{figure}[H]
\centering
\begin{tikzpicture}

\pgfmathsetmacro{\XA}{1.4}   
\pgfmathsetmacro{\XB}{4.2}   
\pgfmathsetmacro{\XC}{6.8}   
\pgfmathsetmacro{\Xmax}{9.6}

\pgfmathsetmacro{\yA}{2.0}  
\pgfmathsetmacro{\yB}{0.40} 
\pgfmathsetmacro{\yC}{0.20} 

\pgfmathsetmacro{\mB}{-0.12}

\pgfmathsetmacro{\Lone}{\XB-\XA}
\pgfmathsetmacro{\Ltwo}{\XC-\XB}

\pgfmathsetmacro{\aone}{(\yA+\mB*\Lone-\yB)/(\Lone*\Lone)}
\pgfmathsetmacro{\bone}{\mB-2*\aone*\Lone}

\pgfmathsetmacro{\atwo}{(\yC-\yB-\mB*\Ltwo)/(\Ltwo*\Ltwo)}
\pgfmathsetmacro{\btwo}{\mB}

\pgfmathsetmacro{\mC}{2*\atwo*\Ltwo+\btwo}

\pgfmathsetmacro{\alphaR}{-(\mC)/\yC}

\begin{axis}[
axis lines = left,
width = 11.5cm,
height = 7cm,
xmin = 0.3,
xmax = \Xmax,
ymin = -0.05,
ymax = 2.2,
xtick = {\XA,\XB,\XC},
xticklabels = {$X_2$,$X_{\frac{\theta}{2}}$,$X_{\frac{\theta}{4}}$},
ytick = {\yC,\yB,\yA},
yticklabels = {$\frac{\theta}{4}$,$\frac{\theta}{2}$,$2$},
tick label style={font=\small},
label style={font=\small},
samples=200,
clip=false,
declare function={
wl(\x)=\aone*(\x-\XA)^2+\bone*(\x-\XA)+\yA;
wm(\x)=\atwo*(\x-\XB)^2+\btwo*(\x-\XB)+\yB;
wr(\x)=\yC*exp(-\alphaR*(\x-\XC));
}
]

\addplot[
black,
line width=1.5pt,
domain=0.3:\XA
] {\yA};

\addplot[
red,
line width=1.5pt,
domain=\XA:\XB
] {wl(x)};

\addplot[
blue,
line width=1.5pt,
domain=\XB:\XC
] {wm(x)};

\addplot[
black,
line width=1.5pt,
domain=\XC:\Xmax
] {wr(x)};

\draw[dashed] (axis cs:\XA,0) -- (axis cs:\XA,\yA);
\draw[dashed] (axis cs:\XB,0) -- (axis cs:\XB,\yB);
\draw[dashed] (axis cs:\XC,0) -- (axis cs:\XC,\yC);

\draw[dashed] (axis cs:0.3,\yB) -- (axis cs:\XB,\yB);
\draw[dashed] (axis cs:0.3,\yC) -- (axis cs:\XC,\yC);

\fill (axis cs:\XA,\yA) circle (1.4pt);
\fill (axis cs:\XB,\yB) circle (1.4pt);
\fill (axis cs:\XC,\yC) circle (1.4pt);

\node[red] at (axis cs:2.65,1.25) {$w_l$}; 
\node[blue] at (axis cs:5.25,0.45) {$w_m$};
\node at (axis cs:8.10,0.26) {$w_r$};

\end{axis}
\end{tikzpicture}
\caption{Schematic graph of the piecewise supersolution $\overline u(t,\cdot)$.}
\label{fig:piecewise-supersolution}
\end{figure}

We construct an explicit $\mathscr{C}^{1,1}_{\textrm{loc}}(X_2(t),\infty)$ supersolution $\overline{u}$ by defining the three parts: $w_l$, $w_m$ and $w_r$. Let us choose $w_l$, $w_m$ and $w_r$ as follows,
 \begin{equation*}
w_l(t,x):=\frac{\theta^2}{4(1-2s)\mu}\left(\frac{x}{t^{\frac{1}{2s}-1}}-\frac{s}{(1-2s)\mu}\gamma t\right)^{-1},\quad
 w_m(t,x):=\frac{\theta}{2}-\mu \frac{\left[x-X_{\frac{\theta}{2}}(t)+1\right]^{1-2s}-1}{t^{\frac{1}{2s}-1}},
 \end{equation*}
and 
 \begin{equation*}
w_r(t,x):=\frac{\frac{\theta}{4} Z^{2s}(t)}{\left[x-X_{\frac{\theta}{4}}(t)+Z(t)\right]^{2s}},
 \end{equation*}
where $\mu>0$ is to be determined later and $Z(t):=\frac{s\theta}{2(1-2s)\mu}t^{\frac{1}{2s}-1}\left(\frac{\theta}{4\mu}t^{\frac{1}{2s}-1}+1\right)^{\frac{2s}{1-2s}}$.

Notice that 
\begin{equation}\label{sc-ls}
X_\lambda(t)=
\begin{cases}
\dfrac{s}{(1-2s)\mu}\gamma t^\frac{1}{2s}+\dfrac{\theta^2}{4(1-2s)\mu \lambda}t^{\frac{1}{2s}-1}, & \lambda\in \left[\frac{\theta}{2},2\right],\\[5pt]
   X_{\frac{\theta}{2}}(t)+\left(\frac{\theta-2\lambda}{2\mu}t^{\frac{1}{2s}-1}+1\right)^\frac{1}{1-2s}-1, &\lambda\in \left[\frac{\theta}{4},\frac{\theta}{2}\right),\\[5pt]
   X_{\frac{\theta}{4}}(t)+\left[\left(\frac{\theta}{4\lambda}\right)^\frac{1}{2s}-1\right]Z(t), &\lambda\in \left(0,\frac{\theta}{4}\right).
\end{cases}
\end{equation}
We have the following estimates for $Z(t)$ and $Z'(t)$.
\begin{lemma}\label{lemma-Zt2}
  There is a time $t_0>1$ such that for all $t>t_0$, we have
     \begin{equation*}
   \zeta_1 t^\frac{1}{2s} \le Z(t)\le  2\zeta_1 t^\frac{1}{2s},  \quad  Z'(t)\ge \zeta_2 t^{\frac{1}{2s}-1},
 \end{equation*}
  where $\zeta_1:=\frac{2s}{1-2s}\left(\frac{\theta}{4\mu}\right)^\frac{1}{1-2s}$ and $\zeta_2:=\left(\frac{\theta}{4\mu}\right)^\frac{1}{1-2s}$.
\end{lemma}
\begin{proof}[{\bf Proof of \Cref{lemma-Zt2}}]
By $Z(t)=\frac{s\theta}{2(1-2s)\mu}t^{\frac{1}{2s}-1}\left(\frac{\theta}{4\mu}t^{\frac{1}{2s}-1}+1\right)^{\frac{2s}{1-2s}}$, for $t$ large enough, say $t\ge t_0$, we arrive at, 
 \begin{equation*}
 \frac{2s}{1-2s}\left(\frac{\theta}{4\mu}\right)^\frac{1}{1-2s} t^\frac{1}{2s}\le Z(t)\le \frac{4s}{1-2s}\left(\frac{\theta}{4\mu}\right)^\frac{1}{1-2s} t^\frac{1}{2s}.
 \end{equation*}
Some calculations show that 
 \begin{equation*}
Z'(t)=\frac{\theta}{4\mu}t^{\frac{1}{2s}-2}\left(\frac{\theta}{4\mu}t^{\frac{1}{2s}-1}+1\right)^\frac{2s}{1-2s}\left[ 1+\frac{s\theta}{2(1-2s)\mu} t^{\frac{1}{2s}-1} \left(\frac{\theta}{4\mu}t^{\frac{1}{2s}-1}+1\right)^{-1}\right].
 \end{equation*}
It follows that
 \begin{equation*}
Z'(t)\ge\left(\frac{\theta}{4\mu}\right)^\frac{1}{1-2s} t^{\frac{1}{2s}-1}.
 \end{equation*}
   
\end{proof}

We list some useful properties of the function $\overline{u}$.
\begin{proposition}\label{prop:pw2}
The function $\overline{u}(\cdot,x)$ is nondecreasing on $\R^+$, and $\overline{u}(t,\cdot)\in \mathscr{C}^{1,1}_{\textrm{\textrm{loc}}}(X_2(t),\infty)$ is nonincreasing on $\R$ and convex on $(X_2(t),\infty)$. Moreover, if $\gamma\ge \gamma_0:=\frac{\theta (4s-1)_+}{4s}\left(\frac{\theta}{4\mu}+1\right)^{\frac{2s}{1-2s}}$, for $t\ge t_0$, we have 
     \begin{equation*}
    \left\{
    \begin{aligned}
        &\partial_t w_l(t,x)\ge  \frac{2 \gamma}{\theta^2}  w_l^2(t,x),\\
        &\partial_t w_m(t,x)\ge  \frac{\gamma}{2\left[x-X_{\frac{\theta}{2}}(t)+1\right]^{2s}}, \\
        &\partial_t w_r(t,x)\ge \frac{\zeta_0\mu^{-\frac{2s}{1-2s}}}{\left[x-X_{\frac{\theta}{4}}(t)+Z(t)\right]^{2s}},
    \end{aligned}\right.
    \quad \text{and} \quad 
     \left\{
     \begin{aligned}
     &\partial_{xx} w_l(t,x) =\frac{ 32(1-2s)^2\mu^2}{\theta^4 t^{\frac{1}{s}-2}}w_l^3(t,x),\\
     &\partial_{xx} w_m(t,x) =\frac{2s(1-2s)\mu}{t^{\frac{1}{2s}-1}\left[x-X_{\frac{\theta}{2}}(t)+1\right]^{2s+1}},\\
     &\partial_{xx} w_r(t,x) =  \frac{s(2s+1)\theta Z^{2s}(t)}{2\left[x-X_{\frac{\theta}{4}}(t)+Z(t)\right]^{2s+2}},
     \end{aligned}\right.
     \end{equation*}
    where $\zeta_0:=\frac{s\theta}{2}\left(\frac{1-2s}{4s}\right)^{1-2s}\left(\frac{\theta}{4}\right)^\frac{2s}{1-2s}$.
   
\end{proposition}
\begin{proof}[{\bf Proof of \Cref{prop:pw2}}]
A direct calculation gives 
 \begin{equation*}
\partial_t w_l(t,x) =\frac{2\gamma}{\theta^2} w_l^2(t,x)+\frac{(1-2s) w_l(t,x)}{2s t} >  \frac{2\gamma}{\theta^2} w_l^2(t,x)>0.
 \end{equation*}
And we have
 \begin{equation*}
\left\{
\begin{aligned}
&\partial_x w_l(t,x) =-\frac{ 4(1-2s)\mu w_l^2(t,x)}{\theta^2 t^{\frac{1}{2s}-1}}< 0,\\
& \partial_{xx} w_l(t,x) = \frac{ 32(1-2s)^2\mu^2  w_l^3(t,x)}{\theta^4 t^{\frac{1}{s}-2}}> 0.
\end{aligned}\right.
 \end{equation*} 
Using $X_{\frac{\theta}{2}}'(t) \ge \frac{\gamma}{2(1-2s)\mu} t^{\frac{1}{2s}-1}$, we get 
 \begin{equation*}
\begin{aligned}
    &\partial_t w_m(t,x) =\frac{ \mu(1-2s) X_{\frac{\theta}{2}}'(t)}{t^{\frac{1}{2s}-1}\left[x-X_{\frac{\theta}{2}}(t)+1\right]^{2s}}+\frac{(1-2s)\mu \left[(x-X_{\frac{\theta}{2}}(t)+1)^{1-2s}-1\right]}{2st^{\frac{1}{2s}}}\\
&\quad\quad\  \ge\frac{\gamma}{2\left[x-X_{\frac{\theta}{2}}(t)+1\right]^{2s}}.
\end{aligned}
 \end{equation*}
Some direct calculations show that 
\begin{equation*}
\left\{\begin{aligned}
&\partial_x w_m(t,x)=  \frac{-(1-2s)\mu}{t^{\frac{1}{2s}-1}\left[x-X_{\frac{\theta}{2}}(t)+1\right]^{2s}}<0,\\
&\partial_{xx} w_m(t,x) =\frac{2s(1-2s)\mu}{t^{\frac{1}{2s}-1}\left[x-X_{\frac{\theta}{2}}(t)+1\right]^{2s+1}}\ge 0.
\end{aligned}\right.
\end{equation*}
\par
In view of Lemma \ref{lemma-Zt2}, by $2s-1<0$, for $t\ge t_0$, we have
 \begin{equation*}
Z^{2s-1}(t)Z'(t)\ge 2^{2s-1}\zeta_1^{2s-1}\zeta_2=\left(\frac{1-2s}{4s}\right)^{1-2s}\left(\frac{\theta}{4}\right)^\frac{2s}{1-2s}\mu^{-\frac{2s}{1-2s}}.
 \end{equation*}
By the definition of $X_\lambda(t)$ and $Z(t)$, we write
 \begin{equation*}
X_{\frac{\theta}{4}}(t)-Z(t)=X_{\frac{\theta}{2}}(t)-1+\left[\frac{\theta}{4\mu}t^{\frac{1}{2s}-1}+1\right]^{\frac{1}{1-2s}} -\frac{s\theta}{2(1-2s)\mu}t^{\frac{1}{2s}-1}\left[\frac{\theta}{4\mu}t^{\frac{1}{2s}-1}+1\right]^{\frac{2s}{1-2s}}.
 \end{equation*}
Then, for all $t\ge 1$, by take $\gamma\ge \gamma_0=\frac{\theta (4s-1)_+}{4s}\left(\frac{\theta}{4\mu}+1\right)^{\frac{2s}{1-2s}}$,  we get 
 \begin{equation*}
\begin{aligned}
&X_{\frac{\theta}{4}}'(t)-Z'(t)\\
&=\frac{\gamma}{2(1-2s)\mu} t^{\frac{1}{2s}-1}+\frac{\theta}{4s\mu }t^{\frac{1}{2s}-2}+\frac{\theta}{4\mu}t^{\frac{1}{2s}-2}\left(\frac{\theta}{4\mu}t^{\frac{1}{2s}-1}+1\right)^{\frac{2s}{1-2s}}\left[\frac{1-2s}{2s}-\frac{2s}{1-2s}\frac{\frac{\theta}{4\mu}t^{\frac{1}{2s}-1}}{\frac{\theta}{4\mu}t^{\frac{1}{2s}-1}+1}\right]\\
&\ge \frac{\gamma}{2(1-2s)\mu}t^{\frac{1}{2s}-1}+\frac{\theta (1-4s)}{8s(1-2s)\mu}t^{\frac{1}{2s}-2}\left(\frac{\theta}{4\mu}t^{\frac{1}{2s}-1}+1\right)^{\frac{2s}{1-2s}}\\
&\ge  \frac{1}{2(1-2s)\mu} \left[\gamma-\frac{\theta (4s-1)_+}{4s}\left(\frac{\theta}{4\mu}+1\right)^{\frac{2s}{1-2s}} \right]t^{\frac{1}{2s}-1} \ge 0.
\end{aligned}
 \end{equation*}
Thus, for $\gamma\ge \gamma_0$ and $t\ge t_0$, we have
 \begin{equation*}
\begin{aligned}
    &\partial_t w_r(t,x) =\frac{s \theta Z^{2s-1}(t)Z'(t)}{2\left[x-X_{\frac{\theta}{4}}(t)+Z(t)\right]^{2s}}+\frac{s\theta Z^{2s}(t)\left[X_{\frac{\theta}{4}}'(t)-Z'(t)\right]}{2\left[x-X_{\frac{\theta}{4}}(t)+Z(t)\right]^{1+2s}}\\
&\qquad\ge\frac{s \theta Z^{2s-1}(t)Z'(t)}{2\left[x-X_{\frac{\theta}{4}}(t)+Z(t)\right]^{2s}} \ge\frac{\zeta_0\mu^{-\frac{2s}{1-2s}}}{\left[x-X_{\frac{\theta}{4}}(t)+Z(t)\right]^{2s}} >0,
\end{aligned}
 \end{equation*}
where $\zeta_0=\frac{s\theta}{2}\left(\frac{1-2s}{4s}\right)^{1-2s}\left(\frac{\theta}{4}\right)^\frac{2s}{1-2s}$.
Moreover, by direct some calculations, we have
 \begin{equation*}
\left\{\begin{aligned}
&\partial_x w_r(t,x) = -\frac{s\theta Z^{2s}(t)}{2\left[x-X_{\frac{\theta}{4}}(t)+Z(t)\right]^{2s+1}}<0,\\
& \partial_{xx} w_r(t,x) =  \frac{s(2s+1)\theta Z^{2s}(t)}{2\left[x-X_{\frac{\theta}{4}}(t)+Z(t)\right]^{2s+2}}>0.
\end{aligned}\right.
 \end{equation*}
 Therefore, collecting the above calculations and using the definition of $\overline{u}$, the function $\overline{u}$ is nondecreasing on $\R^+$, and nonincreasing on $x\in \R$ and convex on $x\in(X_2(t),\infty)$.
In view of \Cref{lemma-c11}, using 
 \[
 \left\{
 \begin{aligned}
 &w_l\left(t,X_{\frac{\theta}{2}}(t)\right)= w_m\left(t,X_{\frac{\theta}{2}}(t)\right),  
  &\partial_x w_l\left(t,X_{\frac{\theta}{2}}(t)\right)= \partial_x w_m\left(t,X_{\frac{\theta}{2}}(t)\right),\\
  &w_m\left(t,X_{\frac{\theta}{4}}(t)\right)= w_r\left(t,X_{\frac{\theta}{4}}(t)\right),  
  &\partial_x w_m\left(t,X_{\frac{\theta}{4}}(t)\right)= \partial_x w_r\left(t,X_{\frac{\theta}{4}}(t)\right),
  \end{aligned}\right.
 \]
 one has $\overline{u}(t,\cdot)\in \mathscr{C}^{1,1}_{\mathrm{\textrm{loc}}}(X_2(t),\infty)$.
\end{proof}
Let us give the estimates for $\mathcal{D}[\overline{u}]$.
\begin{proposition}\label{prop:dec2}
Let $J$ satisfy \Cref{kernel_hypothesis}. Then the following estimates hold:
\begin{itemize}
    \item 
When $x \in \left[X_1(t),X_{\frac{\theta}{2}}(t)\right]$, there are $\mathcal{C}>0$ and $\widetilde t_0>0$ such that for all $t\ge \widetilde t_0$,
 \begin{equation*}
\D[\overline{u}](t,x)\le \mathcal{C}.
 \end{equation*}
       \item  When $x\in \left[X_{\frac{\theta}{2}}(t),X_{\frac{\theta}{4}}(t)\right]$, there is $\widetilde t_1\ge1$ such that for all $t\ge \widetilde t_1$,
     \begin{equation*}
    \mathcal{D}[\overline{u}](t,x)\le 2\mathcal{C}_0\left[x-X_{\frac{\theta}{2}}(t)+1\right]^{-2s}.
     \end{equation*}
    \item  When $x\in \left[X_{\frac{\theta}{4}}(t),\infty\right)$, there is $\widetilde t_2\ge 1$ such that for all $t\ge \widetilde t_2$,
     \begin{equation*}
    \mathcal{D}[\overline{u}](t,x)\le2c_s^{-2s}\mathcal{C}_0\left[x-X_{\frac{\theta}{4}}(t)+Z(t)\right]^{-2s},
     \end{equation*}
     where $c_s:=\min\left\{1,\frac{1-2s}{4s}\right\}$.
     \end{itemize}

\end{proposition}
\begin{proof}[{\bf Proof of \Cref{prop:dec2}}]

We split space into sub-zones as follows.
\par
 \# {\bf Zone 1: $\left[X_1(t),X_{\frac{\theta}{2}}(t)\right]$.} Since $X_1(t)-X_2(t)\ge 1$ for $t$ large enough, say $t\ge \widetilde t_0$, it follows from \Cref{prop:pw2} that
 $\|\partial_{xx}w_l(t,\cdot)\|_{\mathscr{L}^{\infty}\left(\left[X_1-1,X_{\frac{\theta}{2}}\right]\right)} $ and $\|\partial_{xx}w_m(t,\cdot)\|_{\mathscr{L}^{\infty}\left(\left[X_{\frac{\theta}{2}},X_{\frac{\theta}{2}}+1\right]\right)}$ is uniformly bounded for $t\ge \widetilde t_0$. 
 It then follows from \Cref{lemma-c11} and \Cref{D-bounded} that there is $\mathcal{C}>0$ such that for all $t\ge \widetilde t_0$,
\begin{equation}\label{ub-s-zone1}
\D[\overline{u}]\le \mathcal{C}.
\end{equation}
\par
    \# {\bf Zone 2: $\left[X_{\frac{\theta}{2}}(t),X_{\frac{\theta}{4}}(t)\right]$.}
    In view of Proposition \ref{gen-upper-bound-of-D}, the convexity of $\overline{u}$ on $\left[X_{ \frac{\theta}{2}}(t)-1,\infty\right)$ gives 
 \begin{equation*}
\mathcal{D}[\overline{u}]\le \j_1 \mathcal{Q}[\overline{u}](t,x)+ \mathcal{C}_0\left[x-X_{\frac{\theta}{2}}(t)+1\right]^{-2s}.
 \end{equation*}
In view of Proposition \ref{prop:pw2}, we get
 \begin{equation*}
\|\partial_{xx}w_l(t,\cdot)\|_{\mathscr{L}^{\infty}\left(\left[X_{\frac{\theta}{2}}-2,X_{\frac{\theta}{2}}\right]\right)}\lesssim  t^{-2(\frac{1}{2s}-1)}, \quad\|\partial_{xx}w_m(t,\cdot)\|_{\mathscr{L}^{\infty}\left(\left[X_{\frac{\theta}{2}},X_{\frac{\theta}{2}}+2\right]\right)}\lesssim  t^{-(\frac{1}{2s}-1)}.
 \end{equation*}
Thus, for $X_{\frac{\theta}{2}}(t)\le x\le X_{\frac{\theta}{2}}(t)+1$, we have
\begin{align*}
\mathcal{Q}[\overline{u}](t,x)\le L_{X_{\frac{\theta}{2}}}=\max\left\{\|\partial_{xx}w_l(t,\cdot)\|_{\mathscr{L}^{\infty}\left(\left[X_{\frac{\theta}{2}}-2,X_{\frac{\theta}{2}}\right]\right)},\|\partial_{xx}w_m(t,\cdot)\|_{\mathscr{L}^{\infty}\left(\left[X_{\frac{\theta}{2}},X_{\frac{\theta}{2}}+2\right]\right)}\right\} \lesssim  t^{-(\frac{1}{2s}-1)}.
\end{align*}
It follows that for $X_{\frac{\theta}{2}}(t)\le x\le X_{\frac{\theta}{2}}(t)+1$ there is a time $\widetilde t_1\ge 1$ such that for all $t\ge \widetilde t_1$,
 \begin{equation*}
\mathcal{C}_0\left[x-X_{\frac{\theta}{2}}(t)+1\right]^{-2s}\ge 2^{-2s}\mathcal{C}_0 \ge \j_1\mathcal{Q}[\overline{u}](t,x).
 \end{equation*}
For  $X_{\frac{\theta}{2}}(t)+1\le x\le X_{\frac{\theta}{4}}(t)-1$, for $t\ge \widetilde t_1$, up to enlarging $\widetilde t_1$ if necessary,
 \begin{equation*}
\mathcal{Q}[\overline{u}](t,x)\le \sup_{y\in [x-1,x+1]}\partial_{xx}w_m(t,y)\le\frac{2s(1-2s)\mu}{t^{\frac{1}{2s}-1}\left[x-X_{\frac{\theta}{2}}(t)\right]^{2s+1}}\le \mathcal{C}_0\left[x-X_{\frac{\theta}{2}}(t)+1\right]^{-2s}.
 \end{equation*} 
Using \Cref{prop:pw2} again, we get
 \begin{equation*}
\|\partial_{xx}w_m(t,\cdot)\|_{\mathscr{L}^{\infty}\left(\left[X_{\frac{\theta}{4}}-2,X_{\frac{\theta}{4}}\right]\right)}\lesssim  t^{-\frac{1}{2s}+1}\left\{\left[\frac{\theta}{4\mu}t^{\frac{1}{2s}-1}+1\right]^\frac{1}{1-2s}-2\right\}^{-2s-1}\lesssim t^{-\frac{1}{s}},
 \end{equation*} 
and by Lemma \ref{lemma-Zt2}, for $t\ge \widetilde t_1$, up to enlarging $\widetilde t_1$ if necessary,
 \begin{equation*}
\|\partial_{xx}w_r(t,\cdot)\|_{\mathscr{L}^{\infty}\left(\left[X_{\frac{\theta}{4}},X_{\frac{\theta}{4}}+2\right]\right)}\lesssim  Z^{-2}(t)\lesssim  t^{-\frac{1}{s}}.
 \end{equation*}
For $X_{\frac{\theta}{4}}(t)-1\le x\le X_{\frac{\theta}{4}}(t)+1$, we have
\begin{align}\label{LXt2} 
\mathcal{Q}[\overline{u}](t,x)&\le L_{X_{\frac{\theta}{4}}}=\max\left\{\|\partial_{xx}w_m(t,\cdot)\|_{\mathscr{L}^{\infty}\left(\left[X_{\frac{\theta}{4}}-2,X_{\frac{\theta}{4}}\right]\right)},\|\partial_{xx}w_r(t,\cdot)\|_{\mathscr{L}^{\infty}\left(\left[X_{\frac{\theta}{4}},X_{\frac{\theta}{4}}+2\right]\right)}\right\} \nonumber\\
&\lesssim t^{-\frac{1}{s}}.
\end{align}
Since for $X_{\frac{\theta}{4}}(t)-1\le x\le X_{\frac{\theta}{4}}(t)$,
 \begin{equation*}
\mathcal{C}_0\left[x-X_{\frac{\theta}{2}}(t)+1\right]^{-2s}\ge \mathcal{C}_0\left[X_{\frac{\theta}{4}}(t)-X_{\frac{\theta}{2}}(t)+1\right]^{-2s}\gtrsim t^{-1},
 \end{equation*}
we get for $t\ge \widetilde t_1$, up to enlarging $\widetilde t_1$ if necessary,
 \begin{equation*}
\j_1\mathcal{Q}[\overline{u}](t,x)\le \mathcal{C}_0\left[x-X_{\frac{\theta}{2}}(t)+1\right]^{-2s}.
 \end{equation*}
Therefore, for all $t\ge \widetilde t_1$ and $X_{\frac{\theta}{2}}(t)\le x\le X_{\frac{\theta}{4}}(t)$, we achieve 
     \begin{equation*}
    \mathcal{D}[\overline{u}](t,x)\le 2\mathcal{C}_0\left[x-X_{\frac{\theta}{2}}(t)+1\right]^{-2s}.
     \end{equation*}
\medskip
\par
    \# {\bf Zone 3: $\left[X_{\frac{\theta}{4}}(t),\infty\right)$.}
    Since $\overline u(t,\cdot)$ is convex on $[X_2(t),\infty)$, Proposition
\ref{gen-upper-bound-of-D}, applied with
$A=X_2(t)$, yields, for $t$ sufficiently large and
$x\ge X_{\frac{\theta}{4}}(t)$,
\[
\mathcal D[\overline u](t,x)
\le\j_1\mathcal Q[\overline u](t,x)+\mathcal C_0[x-X_2(t)]^{-2s}.
\]
By the definitions of
$X_{\theta/4}(t)$ and $Z(t)$,
\[
\lim_{t\to\infty}\frac{X_{\frac{\theta}{4}}(t)-X_2(t)}{Z(t)}
=
\frac{1-2s}{2s}.
\]
Since $c_s=\min\left\{1,\frac{1-2s}{4s}\right\},$ there exists $\widetilde t_2>0$ such that
\[
X_{\frac{\theta}{4}}(t)-X_2(t)\ge c_s Z(t)
\]
for all $t\ge \widetilde t_2$. Consequently, for every
$x\ge X_{\frac{\theta}{4}}(t)$,
\[
x-X_2(t)=x-X_{\frac{\theta}{4}}(t)+X_{\frac{\theta}{4}}(t)-X_2(t)
\ge c_s\left[x-X_{\frac{\theta}{4}}(t)+Z(t)\right].
\]
It follows that
\[
\mathcal D[\overline u](t,x) \le\j_1\mathcal Q[\overline u](t,x) +c_s^{-2s}\mathcal{C}_0\left[x-X_{\frac{\theta}{4}}(t)+Z(t)\right]^{-2s}.
\]
\par
Let us estimate $\mathcal{Q}[\overline{u}]$ precisely. For $X_{\frac{\theta}{4}}(t)\le x\le X_{\frac{\theta}{4}}(t)+1$, 
 \begin{equation*}
\left[x-X_{\frac{\theta}{4}}(t)+Z(t)\right]^{-2s}\ge [Z(t)+1]^{-2s}\gtrsim t^{-1},
 \end{equation*}
 and thus, by \eqref{LXt2}, for $t\ge \widetilde t_2$, up to enlarging $\widetilde t_2$ if necessary, we have 
 \begin{equation*}
\j_1\mathcal{Q}[\overline{u}](t,x)\le c_s^{-2s}\mathcal{C}_0\left[x-X_{\frac{\theta}{4}}(t)+Z(t)\right]^{-2s}
 \end{equation*}
For $ x\ge X_{\frac{\theta}{4}}(t)+1$, one has for $t\ge \widetilde t_2$, up to enlarging $\widetilde t_2$ if necessary,
 \begin{equation*}
 \mathcal{J}_1\mathcal{Q}[\overline{u}](t,x)\le \mathcal{J}_1\sup_{y\in[x-1,x+1]}\partial_{xx} w_r(t,y) \le  \frac{s(2s+1)\mathcal{J}_1\theta Z^{2s}(t)}{2\left[x-X_{\frac{\theta}{4}}(t)+Z(t)-1\right]^{2s+2}}\le\frac{c_s^{-2s}\mathcal{C}_0}{\left[x-X_{\frac{\theta}{4}}(t)+Z(t)\right]^{2s}}.
 \end{equation*}
Therefore, for $t\ge \widetilde t_2$ and $x\ge X_{\frac{\theta}{4}}(t)$, we arrive at
$$\mathcal{D}[\overline{u}](t,x)\le 2c_s^{-2s}\mathcal{C}_0\left[x-X_{\frac{\theta}{4}}(t)+Z(t)\right]^{-2s}.$$ 
\end{proof}
Now, we are ready to prove that $\overline{u}$ is a supersolution to \eqref{ceq}.
Since the solution $u$ to \eqref{ceq} satisfies $0\le u\le 1$ on $\R^+\times \R$, the parabolic comparison principle in \cite{Brasseur2021,zhang2023optimal} implies that it is enough to verify that $\overline{u}(t,\cdot)$ is a supersolution to \eqref{ceq} on $[X_1(t),\infty)$. 
\begin{proposition}\label{prop:supersolsubcrit}
    For $t\ge \widetilde t:=\max\{t_0,\widetilde t_0,\widetilde t_1,\widetilde t_2\}$, if $$\mu \le \left(\frac{c_s^{2s}\zeta_0}{2\mathcal{C}_0}\right)^{\frac{1-2s}{2s}} \quad \text{and} \quad \gamma\ge \widetilde \gamma:=\max\left\{ \gamma_0,\frac{\theta^2 \rho}{2} + 2 \mathcal{C},4\mathcal{C}_0\right\},$$ 
    then the function $\overline{u}(t,\cdot)$ is a supersolution to \eqref{ceq} on $[X_1(t),+\infty)$.
\end{proposition}

\begin{proof}[{\bf Proof of \Cref{prop:supersolsubcrit}}]
Let us show that $\partial_t \overline{u}-\D[\overline{u}]-f(\overline{u})\ge 0$ for all $t \geq \widetilde t$ and $x\ge X_1(t)$.
\par {\bf \# Zone 1: $\left[X_1(t), X_{ \frac{\theta}{2}}(t)\right]$.} 
By Hypothesis \ref{hypo_f}, there is $\rho >0$ such that $f(z)\le \rho z^2$ for all $z\in(0,1)$. 
Therefore, using \Cref{prop:dec2}, we get
 \begin{equation*}
\partial_t \overline{u}-\D[\overline{u}]-f(\overline{u})\ge  \frac{2\gamma }{\theta^2 } \overline{u}^2 -\mathcal{C}-\rho\overline{u}^2\ge \left(\frac{2\gamma}{\theta^2 } -\rho\right)\frac{\theta^2}{4} -\mathcal{C}\ge 0,
 \end{equation*}
as long as $\gamma \ge  \frac{\theta^2 \rho}{2} +2 \mathcal{C}$.
\par {\bf \# Zone 2: $\left[X_{ \frac{\theta}{2}}(t), X_{ \frac{\theta}{4}}(t)\right]$.} 
In view of Proposition \ref{prop:pw2} and Proposition \ref{prop:dec2}, for $t\ge \widetilde t_1$,
 \begin{equation*}
\partial_t \overline{u}-\D[\overline{u}]-f(\overline{u})\ge \frac{\frac{\gamma}{2} -2\mathcal{C}_0}{\left[x-X_{\frac{\theta}{2}}(t)+1\right]^{2s} }\ge 0,
 \end{equation*}
by taking $\gamma\ge 4\mathcal{C}_0$.

\par {\bf \# Zone 3: $\left[X_{ \frac{\theta}{4}}(t),\infty\right)$.} 
In view of Proposition \ref{prop:pw2} and Proposition \ref{prop:dec2}, for $t\ge \widetilde t_2$ we achieve
 \begin{equation*}
\begin{aligned}
\partial_t \overline{u}-\D[\overline{u}]-f(\overline{u})
& \ge  \frac{\zeta_0\mu^{-\frac{2s}{1-2s}}-2c_s^{-2s}\mathcal{C}_0}{\left[x-X_{\frac{\theta}{4}}(t)+Z(t)\right]^{2s}} \ge 0,
\end{aligned}
 \end{equation*}
as long as $\mu \le \left(\frac{\zeta_0 c_s^{2s}}{2\mathcal{C}_0}\right)^{\frac{1-2s}{2s}}$.
\end{proof}

\section{The case  \texorpdfstring{$s = \frac{1}{2}$}{s=1/2}}

\subsection{Construction of a subsolution}
Let $\Theta$ be any constant in $(\theta,\theta_0)$ and $\varepsilon\in \left(0,\min\{\Theta-\theta,\theta,\frac{1-\theta}{2}\}\right)$ be determined later. 
In view of \Cref{lem:smooth-truncation}, one may find $m\in \mathscr{C}^2([0,1];[0,\Theta])$ satisfying 
\par
 \begin{equation*}
m(y)=\begin{cases}
    y, &y\in \left[0,\theta+\varepsilon\right],\\[5pt]
    \Theta, &y\in\left[\theta+2\varepsilon,1\right],
    \end{cases}
 \end{equation*}
and 
\begin{equation}
    \label{smf2}
0\le m'\le M_\varepsilon \quad\text{and}\quad m''\ge -C_\varepsilon\quad \text{on }[0,1],
\end{equation}
for some $M_\varepsilon$ and $C_\varepsilon>0$ depending only on $\varepsilon$, $\theta$ and $\Theta$. For $\lambda\in(0,1)$, denote
 \begin{equation*}
X_\lambda(t)=
\begin{cases}
   \frac{rt\ln t}{1-t^{-\frac{3\lambda}{4(\theta-\varepsilon)}}} ,& \lambda\in\left(0,\theta-\varepsilon\right],\\
  \frac{rt\ln t}{1-t^{-\frac{3}{4}}}-t^\frac{1}{4}+t^{\sigma_\lambda}, & \lambda\in\left(\theta-\varepsilon,\theta+2\varepsilon\right],
\end{cases}
 \end{equation*}
where $r\in(0,1)$ is to be determined later, $\sigma_\lambda:=\ln \frac{\gamma e^\frac{1}{4}}{\lambda-\theta+\varepsilon+\gamma }$ and  $\gamma:=\frac{6\varepsilon}{e^\frac{1}{4}-1}<24\varepsilon$. Since $\lambda\mapsto \sigma_\lambda$ is strictly decreasing, using $\sigma_{\theta+2\varepsilon}=\ln \frac{2}{e^{-\frac{1}{4}}+1}>0$, we get 
 \begin{equation*}
0<\sigma_{\theta+2\varepsilon}<\sigma_{\theta+\varepsilon}<\sigma_{\theta-\varepsilon}=\frac{1}{4}.
 \end{equation*}
  Notice that $X_{\Theta}(t;\underline{u})=X_{\theta+2\varepsilon}(t)$ and $\lim_{t\to\infty}X_{\lambda_1}(t)-X_{\lambda_2}(t)=\infty$ for any $0<\lambda_1<\lambda_2\le\theta+2\varepsilon$. 
For $t\gg e$, we define
 \begin{equation*}
\underline u(t,x):=\begin{cases}
\Theta, & x\le X_{\theta+2\varepsilon}(t),\\
   m\circ  \phi_m(t,x), &X_{\theta+2\varepsilon}(t)\le x\le X_{\theta-\varepsilon }(t),\\
  \phi_r(t,x) & x\ge X_{\theta-\varepsilon }(t),
\end{cases} 
 \end{equation*}
where
 \begin{equation*} 
\phi_m(t,x):=\frac{\gamma e^\frac{1}{4}}{[x-X_{\theta-\varepsilon}(t)+t^\frac{1}{4}]^{\frac{1}{\ln t}}}-\gamma+\theta-\varepsilon,
 \end{equation*}
and
 \begin{equation*}
\phi_r(t,x)= -\frac{4(\theta-\varepsilon)}{3\ln t}\ln \left(1-\frac{rt\ln t}{x}\right).
 \end{equation*}
Observe that for $t$ large enough, say $t\ge t_0$ where $t_0\ge e$ depending on $r$,
\begin{equation}
    \label{x-t2e}
X_{\theta+2\varepsilon}-rt\ln t=\frac{rt^\frac{1}{4}\ln t}{1-t^{-\frac{3}{4}}}-t^{\frac{1}{4}}+t^{\sigma_{\theta+2\varepsilon}}\ge 0.
\end{equation}
\par
Before proving $\underline u$ is a subsolution to \eqref{ceq}, we list some useful properties of $\underline u$.

\begin{lemma}\label{lb-le2}
 There is $\widetilde t>t_0$ such that for all  $t>\widetilde t$, the function $\underline{u}$ is increasing in $t\in [\widetilde t,\infty)$, and decreasing in $x\in \R$ and convex in $x\in [X_{\theta+\varepsilon}(t),\infty)$. Moreover,  for all $t>0$, $\underline{u}(t,\cdot)\in  \mathscr{C}(\R)\cap\mathscr{C}^2(\R\setminus\{X_{\theta-\varepsilon}(t)\})$.
\end{lemma}
\begin{proof}[{\bf Proof of Lemma \ref{lb-le2}}]

Some direct calculations give 
\begin{equation}\label{dpm2}
\left\{\begin{aligned}
&\partial_t \phi_m(t,x)  = \frac{\phi_m(t,x) +\gamma-\theta+\varepsilon}{\ln t}\left[\frac{X_{\theta-\varepsilon}'(t)-\frac{t^{-\frac{3}{4}}}{4}}{x-X_{\theta-\varepsilon}(t)+t^\frac{1}{4}}+\frac{\ln[x-X_{\theta-\varepsilon}(t)+t^\frac{1}{4}]}{t\ln t}\right],
\\
&\partial_x \phi_m(t,x)  =-\frac{1}{\ln t}\frac{\phi_m(t,x) +\gamma -\theta+\varepsilon}{x-X_{\theta-\varepsilon}(t)+t^\frac{1}{4}}<0,\\
&\partial_{xx} \phi_m(t,x)  =\frac{ \frac{\gamma e^\frac{1}{4}}{\ln t}(\frac{1}{\ln t}+1)}{[x-X_{\theta-\varepsilon}(t)+t^\frac{1}{4}]^{\frac{1}{\ln t}+2}}\ge 0.
\end{aligned}\right.
\end{equation}
and
\begin{equation}
    \label{dwf}
\left\{ \begin{aligned} 
&\partial_t \phi_r(t,x) =\frac{4}{3}(\theta-\varepsilon) \left[\frac{\ln \left(1-\frac{rt\ln t}{x}\right)}{t\ln^2 t}+\frac{r\left(1+\frac{1}{\ln t}\right)}{x-rt\ln t}\right],\\
&\partial_x \phi_r(t,x) =-\frac{\frac{4}{3}(\theta-\varepsilon) rt}{x(x-rt\ln t)}<0,\\
&\partial_{xx}\phi_r(t,x) =\frac{\frac{4}{3}(\theta-\varepsilon)rt}{x(x-rt\ln t)}\left(\frac{1}{x}+\frac{1}{x-rt\ln t}\right)\ge 0.
\end{aligned}\right.
\end{equation}
Among the above derivatives, only the signs of $\partial_t \phi_m$ and $\partial_t \phi_r$ are nontrivial; the others are straightforward.
Since $X_{\theta-\varepsilon}(t)=\frac{rt\ln t}{1-t^{-\frac{3}{4}}}$, 
 \begin{equation*}
X_{\theta-\varepsilon}'(t)=\frac{r\ln t}{1-t^{-\frac{3}{4}}}+r\frac{1-t^{-\frac{3}{4}}\left(1+\frac{3}{4}\ln t\right)}{(1-t^{-\frac{3}{4}})^2}.
 \end{equation*}
Using $e^{\frac{3}{4}\ln t}>1+\frac{3}{4}\ln t$ for $t\ge e$, we get 
 \begin{equation*}
1-t^{-\frac{3}{4}}\left(1+\frac{3}{4}\ln t\right)>0.
 \end{equation*}
Thus, for $t$ large enough, say $t\ge \widetilde t>\max\left\{e,(4r)^{-\frac{4}{3}}\right\}$,
 \begin{equation*}
X_{\theta-\varepsilon}'(t)>r\ln t>\frac{t^{-\frac{3}{4}}}{4}.
 \end{equation*}
For all $t\ge\widetilde t$, it then follows from $X_{\theta+2\varepsilon}(t)-X_{\theta-\varepsilon}+t^\frac{1}{4}=t^{\sigma_{\theta+2\varepsilon}}>1$ that
 \begin{equation*}
\partial_t\phi_m(t,\cdot)> 0 \quad\text{on }\left[X_{\theta+2\varepsilon}(t),\infty\right).
 \end{equation*}
\par On the other hand, by the inequality  $\ln(1-y)\ge \frac{-y}{1-y}$ for $y\in(0,1)$, we get, for $x>rt\ln t$,
 \begin{equation*}
\ln \left(1-\frac{rt\ln t}{x}\right)\ge -\frac{rt\ln t}{x-rt\ln t}.
 \end{equation*}
and thus, for $t\ge e$,
 \begin{equation*}
\partial_t \phi_r(t,x) \ge\frac{\frac{4 }{3}(\theta-\varepsilon)r}{x-rt\ln t}>0.
 \end{equation*}
It follows that $\partial_t \phi_r(t,\cdot)>0$ on $[rt\ln t,\infty)$ for $t\ge e$. Therefore, for $t\ge \widetilde t$, the functions $\phi_m$ and $\phi_r$ are increasing in $t$, and decreasing and convex in $x$.
\par 
   Furthermore, 
    \begin{equation*}
\partial_x \phi_m(t,X_{\theta-\varepsilon}(t))=-\frac{\gamma }{t^\frac{1}{4}\ln t}.
 \end{equation*}
On the other hand,  
 \begin{equation*}
\partial_x \phi_r(t,X_{\theta-\varepsilon}(t))=-\frac{4(\theta-\varepsilon)(1-t^{-\frac{3}{4}})^2}{3rt^\frac{1}{4}\ln^2 t}.
 \end{equation*}
Then, for all $t\ge \widetilde t$, up to enlarging $\widetilde t$ if necessary, we have
 \begin{equation*}
\partial_x \phi_m(t,X_{\theta-\varepsilon}(t))\le \partial_x \phi_r(t,X_{\theta-\varepsilon}(t)). 
 \end{equation*}
Therefore, it follows from the definition of $\underline{u}(t,\cdot)$ and the convexities of $\phi_m(t,\cdot)$ and $\phi_r(t,\cdot)$ that for $t\ge \widetilde t$, the function $\underline{u}(t,\cdot)$ is convex on $[X_{\theta+\varepsilon}(t),\infty)$. Moreover, in view of the definition of $\underline{u}$ and using \eqref{dpm2} and \eqref{dwf}, for all $t>0$, $\underline{u}(t,\cdot)\in \mathscr{C}(\R)\cap\mathscr{C}^2(\R\setminus\{X_{\theta-\varepsilon}(t)\})$.
It completes the proof.
\end{proof}

\begin{proposition}
 \label{prop-phi}
 We have the following derivative estimates:
\begin{enumerate}[label=(\roman*)]
    \item We have
 \begin{equation*}
    \lim_{t\to\infty}\sup_{x\in [X_{\theta+2\varepsilon}(t),X_{\theta+\varepsilon}(t)+1]}|\partial_t \underline{u}(t,x)|= 0.
 \end{equation*}
\item For all $t\ge t_0$, 
 \begin{equation*}
\partial_t \underline{u}(t,x)\le
\begin{cases}
\frac{115\varepsilon}{x-X_{\theta+2\varepsilon}(t)}, &x\in [X_{\theta+\varepsilon}(t), X_{\theta-\varepsilon}(t)),\\[5pt]
\frac{\frac{8}{3}\theta r}{x-X_{\theta+2\varepsilon}(t)}, &x\in (X_{\theta-\varepsilon}(t),\infty).
\end{cases}
 \end{equation*}
\item We have
 \begin{equation*}
\lim_{t\to\infty}\sup_{x\in[X_{\theta+2\varepsilon}(t), X_{\theta-\varepsilon}(t))}[\partial_{xx}\underline{u}(t,x)]^-=0.
 \end{equation*}
\end{enumerate}
   
\end{proposition}
\begin{proof}[{\bf Proof of Proposition \ref{prop-phi}}]
\textbf{\# Proof of (i).}
 By \eqref{smf2}, for $x\in [X_{\theta+2\varepsilon}(t),X_{\theta-\varepsilon}(t)]$. a direct calculation shows that
\begin{equation*}
\partial_t \underline{u}=(m'\circ  \phi_m) \partial_t  \phi_m
\le (m'\circ  \phi_m) \frac{\phi_m+\gamma-\theta+\varepsilon}{\ln t}\left[\frac{X_{\theta-\varepsilon}'(t)-\frac{t^{-\frac{3}{4}}}{4}}{x-X_{\theta-\varepsilon}(t)+t^\frac{1}{4}}+\frac{\ln[x-X_{\theta-\varepsilon}(t)+t^\frac{1}{4}]}{t\ln t}\right].
\end{equation*}
Since $\phi_m\le \theta+2\varepsilon$ for $x\ge X_{\theta+2\varepsilon}(t)$, by $\gamma< 24\varepsilon$, we have 
 \begin{equation*}
\phi_m+\gamma-\theta+\varepsilon\le 27\varepsilon.
 \end{equation*}
 Using $X_{\theta-\varepsilon}'(t)\le \frac{r(\ln t+1)}{1-t^{-\frac{3}{4}}}\le \frac{2}{1-e^{-\frac{3}{4}}}r\ln t<4r\ln t$, 
 \begin{equation}
    \label{phit} 
 \partial_t \underline{u}(t,x)\le \frac{108r\varepsilon (m'\circ  \phi_m)}{x-X_{\theta-\varepsilon}(t)+t^\frac{1}{4}}+\frac{7\varepsilon (m'\circ  \phi_m)}{t\ln t}.
 \end{equation}
 Since $X_{\theta+2\varepsilon}(t)-X_{\theta-\varepsilon}(t)+t^\frac{1}{4}=t^{\sigma_{\theta+2\varepsilon}}\to \infty$ as $t\to \infty$, using $m'\le M_\varepsilon$ on $[0,1]$,
  \begin{equation*}
  \limsup_{t\to\infty}\sup_{x\in[X_{\theta+2\varepsilon}(t),X_{\theta+\varepsilon}(t)+1]}\partial_t \underline{u}(t,x)\le \limsup_{t\to\infty} \left( \frac{108r\varepsilon M_\varepsilon}{X_{\theta+2\varepsilon}(t)-X_{\theta-\varepsilon}(t)+t^\frac{1}{4}}+\frac{7\varepsilon M_\varepsilon}{t\ln t}\right)=0.
  \end{equation*}
 Thus, using $\partial_t\underline{u}\ge 0$ for $t\ge \widetilde t$, we get $\lim_{t\to\infty}\sup_{x\in [X_{\theta+2\varepsilon}(t),X_{\theta+\varepsilon}(t)+1]} |\partial_t \underline{u}(t,x)|=0$.
 \par \textbf{\# Proof of (ii).}
\textbf{\#\# Take $x\in [X_{\theta+\varepsilon}(t),X_{\theta-\varepsilon}(t))$.}
For $t\ge e$, we have
 \begin{equation*}
x-X_{\theta+2\varepsilon}(t)\le X_{\theta-\varepsilon}(t)-X_{\theta+2\varepsilon}(t)\le t^\frac{1}{4}\le t\ln t.
 \end{equation*}
Since $X_{\theta-\varepsilon}(t)-t^\frac{1}{4}-X_{\theta+2\varepsilon}(t)=-t^{\sigma_{\theta+2\varepsilon}}<0$,
 \begin{equation*}
x-X_{\theta-\varepsilon}(t)+t^\frac{1}{4}\ge x-X_{\theta+2\varepsilon}(t).
 \end{equation*}
Therefore, since $r<1$ and $m'\equiv1$ on $(0,\theta+\varepsilon)$, using \eqref{phit},
 \begin{equation*}
\partial_t \underline{u}(t,x)\le \frac{ 108 r \varepsilon }{x-X_{\theta+2\varepsilon}(t)}+\frac{7\varepsilon}{x-X_{\theta+2\varepsilon}(t)}\le \frac{115 \varepsilon }{x-X_{\theta+2\varepsilon}(t)}.
 \end{equation*}
\par \textbf{\#\# Take $x\in (X_{\theta-\varepsilon}(t),\infty)$.}
Since $x\ge X_{\theta-\varepsilon}(t)=\frac{rt\ln t}{1-t^{-\frac{3}{4}}}> rt\ln t$, it follows from \eqref{dwf} that 
 \begin{equation*}
\partial_t \phi_r(t,x) \le  \frac{\frac{8}{3}\theta r}{x-rt\ln t}.
 \end{equation*}
By \eqref{x-t2e}, for all $t\ge t_0$, using the definition of $\underline{u}$, we get
 \begin{equation*}
\partial_t \underline{u}(t,x) \le  \frac{\frac{8}{3}\theta r}{x-X_{\theta+2\varepsilon}(t)}.
 \end{equation*}
 Combining the two cases above, we obtain (ii).
\par \textbf{\# Proof of (iii).}
When $x$ lies in $[X_{\theta+2\varepsilon}(t), X_{\theta-\varepsilon}(t))$, since
 \begin{equation*}
\partial_x\underline u=(m'\circ \phi_m) \partial_x \phi_m\quad\text{and}\quad
\partial_{xx}\underline u=(m''\circ  \phi_m) (\partial_x \phi_m)^2+(m'\circ  \phi_m) \partial_{xx}  \phi_m,
 \end{equation*}
by the convexity of $\phi_m$ and \eqref{smf2}, we arrive at
 \begin{equation*}
\partial_{xx}\underline{u}\ge  -C_\varepsilon (\partial_x  \phi_m)^2.
 \end{equation*}
It follows from
 \begin{equation*}
|\partial_x \phi_m(t,x) | \le \frac{1}{\ln t}\frac{ \gamma+3\varepsilon}{X_{\theta+2\varepsilon}(t)-X_{\theta-\varepsilon}(t)+t^\frac{1}{4}}=\frac{ \gamma+3\varepsilon}{t^{\sigma_{\theta+2\varepsilon}}\ln t}
 \end{equation*}
that 
 \begin{equation*}
\lim_{t\to\infty}\sup_{x\in[X_{\theta+2\varepsilon}(t), X_{\theta-\varepsilon}(t))}|\partial_x \phi_m(t,x)|=0.
 \end{equation*}
Thus, using \Cref{lb-le2}, we get 
 \begin{equation*}
\lim_{t\to\infty}\sup_{x\in[X_{\theta+2\varepsilon}(t), X_{\theta-\varepsilon}(t))}[\partial_{xx}\underline{u}(t,x)]^-=0.
 \end{equation*}
\end{proof}

Equipped with Proposition \ref{prop-D[phi]}, we have the following proposition.
\begin{proposition}\label{prop-D[phi]-s=1/2}
Let $J$ satisfy \Cref{kernel_hypothesis}. Then the following estimates hold:
 \begin{itemize}
\item When $x\in(-\infty,X_{\theta+\varepsilon}(t)+1]$, for any $C>0$, there is a time $\underline t_1$ such that, for $t\ge \underline t_1$, we have
     \begin{equation*}
    \mathcal{D}[\underline u](t,x)\ge -C.
     \end{equation*}
\item When $x\in [X_{\theta+\varepsilon}(t)+1,+\infty)$, if $\varepsilon\le \varepsilon_0:=\frac{\Theta-\theta}{104\j_0^2+2}$, then there exists a time $\underline t_2$ such that, for all $t\ge \underline t_2$,  we have
     \begin{equation*}
    \mathcal{D}[\underline{u}](t,x)\ge \frac{\Theta-\theta}{2\j_0(x-X_{\theta+2\varepsilon}(t))}.
     \end{equation*}
\end{itemize}
\end{proposition}
\begin{proof}[{\bf Proof of Proposition \ref{prop-D[phi]-s=1/2}}]
Since $\sigma_{\theta-\varepsilon}>\sigma_{\theta+\varepsilon}>\sigma_{\theta+2\varepsilon}$, we have
$X_{\theta+\varepsilon}(t)-X_{\theta+2\varepsilon}(t)=t^{\sigma_{\theta+\varepsilon}}-t^{\sigma_{\theta+2\varepsilon}}\ge R_0$ and $X_{\theta-\varepsilon}(t)-X_{\theta+\varepsilon}(t)=t^{\sigma_{\theta-\varepsilon}}-t^{\sigma_{\theta+\varepsilon}}\ge 1$ for $t$ large enough, say $t\ge \underline{t}_0$. In view of the definition of $\underline u$, by taking $\xi_1=X_{\theta+2\varepsilon}(t)$, $\xi_2=X_{\theta+\varepsilon}(t)$ and $\Omega=\{X_{\theta-\varepsilon}(t)\}$, we can apply Proposition \ref{prop-D[phi]} to the function $\underline u(t,\cdot)$ for $s\in\left(0,\frac{1}{2}\right]$. 

{\bf \# Take $x\in(-\infty,X_{\theta+\varepsilon}(t)+1]$.}
Since $J\in L^1([1,+\infty))$ and $\underline u\le \Theta<1$, by choosing $B_1$ large enough, one may get 
 \begin{equation*}
\underline u(t,x)\int_{B_1}^{+\infty}J(z) \dd z \le \frac{C}{2}.
 \end{equation*}
 There is a time $\underline{t}_1>\underline{t}_0$ such that for $t\ge \underline t_1$ we have
   \begin{equation*}
\D[\underline u](t,x)\ge  -C.
 \end{equation*}
\par {\bf \# Take $x\in[X_{\theta+\varepsilon}(t)+1,X_{\theta-\varepsilon}(t)]$.}
In view of the definition of $\underline u$, we write
\begin{equation}
    \label{u/pu}
-\frac{\underline{u}}{\partial_x \underline{u}}=-\frac{\phi_m}{\partial_x \phi_m}=\frac{\phi_m}{\phi_m+\gamma-\theta+\varepsilon}[x-X_{\theta-\varepsilon}(t)+t^\frac{1}{4}]\ln t.
\end{equation}
Since $\theta-\varepsilon\le \phi_m\le \theta+2\varepsilon$ and $X_{\theta+2\varepsilon}(t)-X_{\theta-\varepsilon}(t)+t^\frac{1}{4}=t^{\sigma_{\theta+2\varepsilon}}$,  there is a time $\underline t_2'>\underline{t}_0$ such that for all $t\ge \underline t_2'$, we have
 \begin{equation*}
-\frac{\underline{u}}{\partial_x \underline{u}}\ge \frac{\theta-\varepsilon}{\gamma+3\varepsilon} t^{\sigma_{\theta+2\varepsilon}}\ln t\ge e.
 \end{equation*}
Since $\frac{y}{\ln y}\ge e$ for all $y\in(1,+\infty)$,
we have for all $t\ge \underline t_2'$,
 \begin{equation*}
\frac{-\frac{\underline u}{\partial_x\underline u}}{\ln\left(-\frac{\underline u}{\partial_x \underline u}\right)}\ge e>1.
 \end{equation*}
Since $\underline u(t,\cdot)$ is $\mathscr{C}^2$ on $[X_{\theta+\varepsilon}(t)+1,X_{\theta-\varepsilon}(t))$, the subdifferential $\partial \underline u(t,\cdot)$ reduces to the singleton $\{\partial_x \underline u(t,\cdot)\}$. 
Thus, for all $t\ge \underline t_2'$, by Proposition \ref{prop-D[phi]} and taking $B_2=\frac{-\frac{\underline u}{\partial_x\underline u}}{\ln\left(-\frac{\underline u}{\partial_x \underline u}\right)}>1$,  we obtain
\begin{equation}
   \label{pephi} 
\begin{aligned}  
\mathcal{D}[\underline u](t,x)&\ge \j_0\partial_x\underline u\ln B_2-\frac{\j_0\underline u}{B_2}+ \frac{\j_0^{-1}[\Theta-\underline u]}{x-X_{\theta+2\varepsilon}(t)}\\
&\ge 2\j_0 \partial_x\underline u\ln\left(-\frac{\underline u}{\partial_x\underline u}\right)+ \frac{\j_0^{-1}[\Theta-\underline u]}{x-X_{\theta+2\varepsilon}(t)}.
\end{aligned}
\end{equation}
By \eqref{u/pu}, it follows from $\theta-\varepsilon\le\phi_m\le \theta+\varepsilon$ and $\ln t<t$ for all $t\ge e$ that for all $t\ge \underline t_2'$, up to enlarging $\underline t_2'$ if necessary,
 \begin{equation*}
-\frac{\underline{u}}{\partial_x \underline{u}}\le\frac{\phi_m [x-X_{\theta-\varepsilon}(t)+t^\frac{1}{4}]\ln t}{\phi_m+\gamma-\theta+\varepsilon}\le \frac{(e^{\frac{1}{4}}-1)(\theta+\varepsilon)}{6\varepsilon}t^\frac{1}{4}\ln t < \frac{\theta+\varepsilon}{6\varepsilon}t^\frac{3}{4}<t.
 \end{equation*}
Thus,
 \begin{equation*}
\partial_x\underline{u}\ln\left(-\frac{\underline{u}}{\partial_x\underline{u}}\right)\ge -\frac{\gamma+2\varepsilon}{x-X_{\theta-\varepsilon}(t)+t^\frac{1}{4}}\ge -\frac{26\varepsilon}{x-X_{\theta+2\varepsilon}(t)}.
 \end{equation*}
It then follows from Proposition \ref{prop-D[phi]} that
 \begin{equation*} 
(x-X_{\theta+2\varepsilon}(t))\D[\underline{u}](t,x)
\ge -52\j_0\varepsilon+\j_0^{-1}[\Theta-\theta-\varepsilon]\ge   \frac{\Theta-\theta}{2\j_0},
 \end{equation*}
as long as $\varepsilon\le \frac{\Theta-\theta}{104\j_0^2+2}$. As a result, for all $(t,x)\in [\underline{t}_2',\infty)\times[X_{\theta+\varepsilon}(t)+1,X_{\theta-\varepsilon}(t)]$, we achieve 
 \begin{equation*}
\D[\underline u](t,x)\ge \frac{\Theta-\theta}{2\j_0(x-X_{\theta+2\varepsilon}(t))}.
 \end{equation*}

\par {\bf \# Take $x\in[X_{\theta-\varepsilon}(t),\infty)$.} 
The symmetry of $J$ gives  
 \begin{equation*}
\mathcal{D}[\underline{u}](t,x)=\int_0^{\infty}[\underline{u}(t,x-z)+\underline{u}(t,x+z)-2\underline{u}(t,x)]J(z) \dd z .
 \end{equation*}
The convexity $\underline{u}(t,\cdot)$ on $[X_{\theta+\varepsilon}(t),\infty)$ implies $\underline{u}(t,x-z)+\underline{u}(t,x+z)-2\underline{u}(t,x)\ge 0$ for $z\in[0,x-X_{\theta+\varepsilon}(t)]$, and thus, 
 \begin{equation*}
\begin{aligned} 
\mathcal{D}[\underline{u}](t,x)&\ge \int_{x-X_{\theta+\varepsilon}(t)}^{\infty}[\underline{u}(t,x-z)+\underline{u}(t,x+z)-2\underline{u}(t,x)]J(z) \dd z \\
&=\left(\int_{x-X_{\theta+\varepsilon}(t)}^{x-X_{\theta+2\varepsilon}(t)}+\int_{x-X_{\theta+2\varepsilon}(t)}^{K(x-X_{\theta+2\varepsilon}(t))}+\int_{K(x-X_{\theta+2\varepsilon}(t))}^{\infty}\right)[\underline{u}(t,x-z)+\underline{u}(t,x+z)-2\underline{u}(t,x)]J(z) \dd z \\
&:=I_1+I_2+I_3.
\end{aligned}
 \end{equation*}
where $K>1$ is to be determined later.
For $I_1$, since $0<\underline{u}(t,\cdot)< 1$ on $\R$, 
 \begin{equation*}
I_1\ge -2\j_0\int_{x-X_{\theta+\varepsilon}(t)}^{x-X_{\theta+2\varepsilon}(t)}z^{-2} \dd z =-2\j_0\frac{X_{\theta+\varepsilon}(t)-X_{\theta+2\varepsilon}(t)}{(x-X_{\theta+2\varepsilon}(t))(x-X_{\theta+\varepsilon}(t))}.
 \end{equation*}
By $0<\sigma_{\theta+2\varepsilon}<\sigma_{\theta+\varepsilon}<\frac{1}{4}$, there is a time $\underline{t}_2''> \underline{t}_0$ such that for all $t\ge \underline{t}_2''$,
 \begin{equation*}
\frac{X_{\theta+\varepsilon}(t)-X_{\theta+2\varepsilon}(t)}{x-X_{\theta+\varepsilon}(t)}\le \frac{X_{\theta+\varepsilon}(t)-X_{\theta+2\varepsilon}(t)}{X_{\theta-\varepsilon}(t)-X_{\theta+\varepsilon}(t)}=\frac{t^{\sigma_{\theta+\varepsilon}}-t^{\sigma_{\theta+2\varepsilon}}}{t^\frac{1}{4}-t^{\sigma_{\theta+\varepsilon}}}\le \frac{\Theta-\theta}{16\j_0^2}.
 \end{equation*}
As a result, for all $t\ge \underline{t}_2''$, 
 \begin{equation*}
I_1\ge -\frac{\Theta-\theta}{8\j_0(x-X_{\theta+2\varepsilon}(t))} .
 \end{equation*}
\par
For $I_2$, by \eqref{x-t2e}, we get
 \begin{equation*}
\underline{u}(t,x)-\underline{u}(t,x+z)\le\frac{\frac{4}{3}(\theta-\varepsilon)}{\ln t}\ln \left(1+\frac{z}{x-r t\ln t}\right)\le \frac{\frac{4}{3}(\theta-\varepsilon)\ln (1+K)}{\ln t}\to 0\quad\text{as }t\to \infty,
 \end{equation*}
and thus, for $t\ge \underline{t}_2''$, up to enlarging $\underline{t}_2''$ if necessary,
 \begin{equation*}
\underline{u}(t,x-z)+\underline{u}(t,x+z)-2\underline{u}(t,x) \ge\Theta -\theta+\underline{u}(t,x+z)-\underline{u}(t,x)\ge \frac{7(\Theta-\theta)}{8}. 
 \end{equation*}
It follows from $x-X_{\theta+2\varepsilon}(t)\ge X_{\theta-\varepsilon}(t)-X_{\theta+2\varepsilon}(t)\ge R_0$ for $t\ge \underline{t}_2''$, up to enlarging $\underline{t}_2''$ if necessary, that by taking $K\ge 7$,
 \begin{equation*}
I_2\ge \frac{7(\Theta-\theta)}{8\j_0}\int_{x-X_{\theta+2\varepsilon}(t)}^{K(x-X_{\theta+2\varepsilon}(t))} z^{-2} \dd z =\frac{7(\Theta-\theta)}{8\j_0(x-X_{\theta+2\varepsilon}(t))}\left(1-\frac{1}{K}\right)\ge\frac{3(\Theta-\theta)}{4\j_0(x-X_{\theta+2\varepsilon}(t))}.
 \end{equation*}
\par 
For $I_3$, using $0<\underline{u}(t,\cdot)< 1$ on $\R$ and taking $K\ge \max\left\{\frac{16\j_0^2}{\Theta-\theta},7\right\}$, we get
 \begin{equation*}
I_3\ge -2\j_0\int_{K(x-X_{\theta+2\varepsilon}(t))}^{\infty} z^{-2} \dd z =\frac{-2\j_0}{K(x-X_{\theta+2\varepsilon}(t))}\ge-\frac{\Theta-\theta}{8\j_0(x-X_{\theta+2\varepsilon}(t))} .
 \end{equation*}
As a result, for all $(t,x)\in [\underline t_2'',\infty)\times[X_{\theta-\varepsilon}(t),\infty)$ we achieve 
 \begin{equation*}
\D[\underline{u}](t,x)\ge \frac{\Theta-\theta}{2\j_0(x-X_{\theta+2\varepsilon}(t))}.
 \end{equation*}
\par
Therefore, for any $\varepsilon\le \frac{\Theta-\theta}{104\j_0^2+2}$,  by taking $\underline t_2=\max\{\underline t_2',\underline t_2''\}$, for all $(t,x)\in [\underline t_2,\infty)\times[X_{\theta+\varepsilon}(t)+1,+\infty)$  we have
     \begin{equation*}
    \mathcal{D}[\underline u](t,x)\ge \frac{\Theta-\theta}{2\j_0(x-X_{\theta+2\varepsilon}(t))}.
     \end{equation*}
\end{proof}
Now, let us show that $\underline u$ is a subsolution to \eqref{ceq}.
\begin{proposition}\label{prop-subsolution}
 If $\varepsilon\le \underline{\varepsilon}:= \min\left\{\theta,\Theta-\theta,\frac{1-\theta}{2},\frac{\Theta-\theta}{104\j_0^2+2},\frac{\Theta-\theta}{230\j_0}\right\}$ and $r \le  \underline{r}:=\min\left\{\frac{3(\Theta-\theta)}{16\theta\j_0},1\right\}$, then there is a time $t^\star$ such that, for all $t\ge t^\star$, $\underline u(t,\cdot)$ is a subsolution to \eqref{ceq} on $\R$.
\end{proposition}
\begin{proof}[{\bf Proof of Proposition \ref{prop-subsolution}}]
 {\bf \# Take $x\in(-\infty,X_{\theta+\varepsilon}(t)+1]$.} 
  Denote $\nu:=\inf_{[\theta+\frac{\varepsilon}{2},\Theta]} f>0$. It follows from \Cref{prop-phi} and \Cref{prop-D[phi]} that there is a time $t^\#$ such that for all $t\ge t^\#$,
 \begin{equation*}
\partial_t \underline u\le \frac{\nu}{2}\quad \text{and}\quad \mathcal{D}[\underline{u}]\ge -\frac{\nu}{2}.
 \end{equation*}
 Since $X_{\theta+\frac{\varepsilon}{2}}(t)-X_{\theta+\varepsilon}(t)\ge 1$ for $t$ large enough, for all $(t,x)\in [t^\#,\infty)\times(-\infty,X_{\theta+\varepsilon}(t)+1]$, up to enlarging $t^\#$ if necessary, we have 
  \begin{equation*}
 \underline{u}(t,x)\ge \underline{u}(t,X_{\theta+\frac{\varepsilon}{2}}(t))=\theta+\frac{\varepsilon}{2},
  \end{equation*}
 and thus,
  \begin{equation*}
 f(\underline{u}(t,x))\ge \nu.
  \end{equation*}
Therefore,  for all $t\ge \max\{\widetilde t, t^\#\}$, we achieve 
 \begin{equation*}
\partial_t \underline u-\D[\underline u]-f(\underline u)\le \frac{\nu}{2}+\frac{\nu}{2}-\nu=0.
 \end{equation*}
\par {\bf \# Take $x\in [X_{\theta+\varepsilon}(t)+1,X_{\theta-\varepsilon}(t)]$.}
Since $f\ge 0$ on $[0,1]$, for all $t\ge \max\{\underline t_2,\widetilde t\}$, by Proposition \ref{prop-phi} and Proposition \ref{prop-D[phi]-s=1/2}, we achieve 
 \begin{equation*}
\partial_t \underline{u}-\mathcal{D}[\underline{u}]-f(\underline{u})
\le \left(115\varepsilon-\frac{\Theta-\theta}{2\j_0}\right)\frac{1 }{x-X_{\theta+2\varepsilon}(t)} \le 0,
 \end{equation*}
as long as $\varepsilon\le \frac{\Theta-\theta}{230\j_0}$.

\par {\bf \# Take $x\in[X_{\theta-\varepsilon}(t),\infty)$.} Since $f\ge 0$ on $[0,1]$, by Proposition \ref{prop-phi} and Proposition \ref{prop-D[phi]-s=1/2}, for $t\ge\max\{\widetilde t,\underline t_2\}$ and $x\ge X_{\theta-\varepsilon}(t)$, we have
 \begin{equation*}
\partial_t \underline u-\mathcal{D}[\underline u]-f(\underline u)
\le \left(\frac{8}{3}\theta r-\frac{\Theta-\theta}{2\j_0}\right)\frac{1 }{x-X_{\theta+2\varepsilon}(t)} \le 0,
 \end{equation*}
as long as $r \le  \frac{3(\Theta-\theta)}{16\theta\j_0}$.

\par Combining the above results on three spatial zones, taking $t^\star:=\max\{\widetilde t,t^\#,\underline t_1, \underline t_2\}$, for $\varepsilon\le \underline{\varepsilon}$ and $r\le \underline{r}$, we therefore prove that $\underline u$ is a subsolution to \eqref{ceq} on $[t^\star,\infty)\times \R$.  

\end{proof}

\subsection{Construction of a supersolution}

We construct the supersolution by combining 
three components: $w_l$, $w_m$, and $w_r$, defined below. For $t\gg e$, define
 \begin{equation*}
\overline{u}(t,x)=
\begin{cases}
2, &x\le X_2(t),\\[5pt]
 w_l(t,x), &X_2(t)\le x\le X_{\frac{\theta}{2}}(t), \\[5pt]
w_m(t,x), & X_{\frac{\theta}{2}}(t)\le x\le X_{\frac{\theta}{8}}(t), \\[5pt]
 w_r(t,x), &x\ge X_{\frac{\theta}{8}}(t),
\end{cases}
 \end{equation*}
where
 \begin{equation*}
\left\{\begin{aligned}
&w_l(t,x):=\frac{\alpha(t) \ln t}{x-rt\ln t},\\
&w_m(t,x):=\frac{\theta}{2}- \frac{\theta}{8}\frac{\ln \left[x-X_{\frac{\theta}{2}}(t)+1\right]}{\ln t}- \frac{x-X_{\frac{\theta}{2}}(t)}{rt\ln t},\\
&w_r(t,x):=\frac{\frac{\theta}{8} Z(t)}{x-X_{\frac{\theta}{8}}(t)+Z(t)},
\end{aligned}\right.
 \end{equation*}
 where $\alpha(t):=\frac{2rt\theta^2}{rt\theta +8}$ and $Z(t):=\frac{\frac{\theta}{8} rt\ln t \left[X_{\frac{\theta}{8}}(t)-X_{\frac{\theta}{2}}(t)+1\right]}{\frac{\theta}{8} rt+  X_{\frac{\theta}{8}}(t)-X_{\frac{\theta}{2}}(t)+1}$.
 \par
Notice that for $\lambda\in \left(0,\frac{\theta}{8}\right)\cup\left[\frac{\theta}{2},2\right]$,
\begin{equation}\label{sc-lc}
X_\lambda(t):=
\begin{cases}
rt\ln t+\frac{\alpha}{ \lambda}\ln t, & \lambda\in \left[\frac{\theta}{2},2\right],\\[5pt]
   X_{\frac{\theta}{8}}(t)+\left(\frac{\theta}{8\lambda}-1\right)Z(t), &\lambda\in \left(0,\frac{\theta}{8}\right).
\end{cases}
\end{equation}
The position $X_\lambda(t)$ for $\lambda\in \left[\frac{\theta}{8},\frac{\theta}{2}\right)$ is not explicit, and we have the following estimates.
\begin{proposition}\label{prop:Xt21}
There is a time $t_1(r)\ge e$ such that for all $t\ge t_1$, we have
    \begin{equation*}
 \frac{3 \theta}{16}\le \frac{X_{\frac{\theta}{8}}(t) -X_{\frac{\theta}{2}}(t)}{rt\ln t} \le \frac{3 \theta}{8}.
\end{equation*}

\end{proposition}

\begin{proof}[{\bf Proof of \Cref{prop:Xt21}}]
    It follows from $w_m\left(t,X_{\frac{\theta}{8}}(t)\right)=\frac{\theta}{8}$ that 
 \begin{equation*}
\frac{3 \theta}{8}= \frac{\theta}{8}\frac{\ln \left[X_{\frac{\theta}{8}}(t)-X_{\frac{\theta}{2}}(t)+1\right]}{\ln t}+ \frac{X_{\frac{\theta}{8}}(t)-X_{\frac{\theta}{2}}(t)}{rt\ln t}\ge \frac{X_{\frac{\theta}{8}}(t)-X_{\frac{\theta}{2}}(t)}{rt\ln t}.
 \end{equation*}
Thus,
 \begin{equation*}
X_{\frac{\theta}{8}}(t)\le X_{\frac{\theta}{2}}(t)+\frac{3 \theta}{8} rt\ln t.
 \end{equation*}
Then, by $\frac{3 \theta}{8}\le 1$, we have
 \begin{equation*}
\frac{3 \theta}{8}= \frac{\theta}{8}\frac{ \ln \left[X_{\frac{\theta}{8}}(t)-X_{\frac{\theta}{2}}(t)+1\right]}{\ln t}+ \frac{X_{\frac{\theta}{8}}(t)-X_{\frac{\theta}{2}}(t)}{rt\ln t}\le  \frac{\theta}{8}\frac{ \ln [rt\ln t+1]}{\ln t}+\frac{X_{\frac{\theta}{8}}(t)-X_{\frac{\theta}{2}}(t)}{rt\ln t},
 \end{equation*}
and thus, 
 \begin{equation*}
X_{\frac{\theta}{8}}(t)- X_{\frac{\theta}{2}}(t)\ge \frac{\theta}{8}\left[2-\frac{\ln(r\ln t+t^{-1})}{\ln t}\right] rt\ln t.
 \end{equation*}
As a result, there is a time $t_1\ge e$ such that for all $t\ge t_1$, we have
 \begin{equation*}
 X_{\frac{\theta}{2}}(t)+\frac{3\theta}{16} rt\ln t\le X_{\frac{\theta}{8}}(t)\le X_{\frac{\theta}{2}}(t)+\frac{3 \theta}{8}rt\ln t.
 \end{equation*}
\end{proof}

We have the following estimates for $Z(t)$ and $Z'(t)$.
\begin{lemma}\label{prop:Zte} 
For all $t\ge t_1$,
 \begin{equation*}
\frac{\theta}{16}rt\ln t\le Z(t)\le \frac{\theta}{8} rt\ln t\quad\text{and}\quad
    Z'(t)\ge\frac{\theta r\ln t}{32}.
     \end{equation*}
\end{lemma}
\begin{proof}[{\bf Proof of \Cref{prop:Zte}}]
    Denote $W(t):=X_{\frac{\theta}{8}}(t)-X_{\frac{\theta}{2}}(t)+1$. Since $Z(t)=\frac{\frac{\theta}{8} rt\ln t W(t)}{ \frac{\theta}{8} rt+  W(t)}$, it follows from \Cref{prop:Xt21} that, we have  for all $t\ge t_1$,
    \begin{equation}\label{ub-Wrt}
        W(t)\ge \frac{3\theta }{16}rt\ln t> \frac{\theta}{8} r t,
    \end{equation}
    and then 
 \begin{equation*}
\frac{\theta rt\ln t}{16}\le Z(t)\le \frac{\theta rt\ln t}{8} .
 \end{equation*}
Using the definition of $X_{\frac{\theta}{8}}$, we get 
\begin{equation*}
\frac{3 \theta}{8}= \frac{\theta}{8}\frac{\ln W(t)}{\ln t}+ \frac{W(t)-1}{rt\ln t},
 \end{equation*}
and thus, multiplying both sides by $\ln t$ and differentiating with respect to $t$, by $W\ge 1$,
we obtain
 \begin{equation*}
W'(t)=\frac{\frac{3 \theta r}{8}+\frac{W(t)-1}{t}}{\frac{ \frac{\theta}{8} rt}{W(t)}+1}\ge 0.
 \end{equation*}
 Thus, by \eqref{ub-Wrt}, we arrive at
 \begin{equation*}
\begin{aligned}
Z'(t)=\frac{\theta r}{8}\frac{ \frac{\theta}{8} r tW(t) + \frac{\theta}{8} r t^2 \ln t W'(t)+W^2(t)(\ln t+1)}{\left[ \frac{\theta}{8} rt+W(t)\right]^2} \ge \frac{\theta r\ln t}{8\left(1+\frac{\theta rt}{8W(t)}\right)^2}\ge\frac{\theta r\ln t}{32}.
\end{aligned}
 \end{equation*}

\end{proof}

We list some useful properties of $\overline{u}$.
\begin{proposition}\label{pw1}
  The function $\overline{u}(\cdot,x)$ is nondecreasing on $[t_1,\infty)$, and for all $t\ge t_1$, $\overline{u}(t,\cdot)\in \mathscr{C}^{1,1}_{\mathrm{\textrm{loc}}}(X_2(t),\infty)$ is nonincreasing on $\R$ and convex on $(X_2(t),\infty)$. Moreover, we have 
     \begin{equation*}
    \left\{
    \begin{aligned} 
    &\partial_t w_l\ge   \frac{r}{2\theta}w_l^2,
    \\
    &\partial_t w_m\ge  \frac{ \frac{\theta}{8} r}{x-X_{\frac{\theta}{2}}(t)+1}+ \frac{1 }{t},\\
    &\partial_t w_r\ge  \frac{\frac{\theta^2r }{256} \ln t}{x-X_{\frac{\theta}{8}}(t)+Z(t)},
     \end{aligned}\right. \quad\text{and}\quad  \left\{
    \begin{aligned} 
  &\partial_{xx} w_l \le \frac{4\theta \ln t}{(x-rt\ln t)^3},\\
  &\partial_{xx} w_m =\frac{\theta}{8\ln t\left[x-X_{\frac{\theta}{2}}(t)+1\right]^2},\\
  & \partial_{xx} w_r =  \frac{\theta Z(t)}{4\left[x-X_{\frac{\theta}{8}}(t)+Z(t)\right]^{3}}.
     \end{aligned}\right. 
     \end{equation*}

\end{proposition}

\begin{proof}[{\bf Proof of \Cref{pw1}}]
Since $\alpha(t)\le 2\theta$ and $\alpha' \ge 0$, some direct calculations show that
 \begin{equation*}
\partial_t w_l =\frac{r(\ln t+1)}{\alpha(t)\ln t}w_l^2+\frac{[\alpha(t)\ln t]'}{x-rt\ln t}\ge  \frac{r}{2\theta}w_l^2>0.
 \end{equation*}
Furthermore,
 \begin{equation*}
\left\{
\begin{aligned}
&\partial_x w_l =-\frac{\alpha(t) \ln t}{(x-rt\ln t)^2}< 0,\\
& \partial_{xx} w_l =\frac{2\alpha(t) \ln t}{(x-rt\ln t)^3}\le  \frac{4\theta \ln t}{(x-rt\ln t)^3}.
\end{aligned}\right.
 \end{equation*} 
 \par
Using $X_{\frac{\theta}{2}}'(t)> r\ln t$, we get
 \begin{equation*}
\begin{aligned}
&\partial_t w_m =\frac{\theta\ln[x-X_{\frac{\theta}{2}}(t)+1]}{8t\ln^2 t}+ \frac{\theta X_{\frac{\theta}{2}}'(t)}{8\ln t[x-X_{\frac{\theta}{2}}(t)+1]}+ \frac{(\ln t+1)[x-X_{\frac{\theta}{2}}(t)]}{rt^2\ln^2 t}+\frac{X_{\frac{\theta}{2}}'(t)}{rt\ln t}\\
&\qquad\quad \ge  \frac{\theta X_{\frac{\theta}{2}}'(t)}{8\ln t[x-X_{\frac{\theta}{2}}(t)+1]}+ \frac{X_{\frac{\theta}{2}}'(t)}{rt\ln t}
\ge  \frac{ \frac{\theta}{8} r}{x-X_{\frac{\theta}{2}}(t)+1}+ \frac{1 }{t}.
\end{aligned}
 \end{equation*}
Moreover, 
\begin{equation*}
\left\{\begin{aligned}
&\partial_x w_m=-\frac{\theta}{8\ln t[x-X_{\frac{\theta}{2}}(t)+1]}-\frac{1}{rt\ln t}<0\\
&\partial_{xx} w_m =\frac{\theta}{8\ln t[x-X_{\frac{\theta}{2}}(t)+1]^2}\ge 0,
\end{aligned}\right.
\end{equation*}
\par 
By some direct calculations, we obtain
 \begin{equation*}
\begin{aligned}
X'_{\frac{\theta}{8}}(t)-Z'(t)&= \frac{2\alpha}{\theta t}
+ \frac{2 \alpha'}{\theta}\ln t+ r(\ln t+1)\left[1-\frac{\theta}{8}\frac{W^2(t)}{( \frac{\theta}{8} rt+W(t))^2}\right] \\
&\quad -\frac{\theta^2}{64}\frac{ r^2 t W(t)}{( \frac{\theta}{8} rt+W(t))^2}+ W'(t)\left[1-\frac{\theta^2}{64}\frac{ r^2 t^2\ln t}{( \frac{\theta}{8} rt+W(t))^2}\right]\\
&>2r\left(1-\frac{\theta}{8}\right)-\frac{\theta^2}{64}\frac{ r^2 t W(t)}{( \frac{\theta}{8} rt+W(t))^2}+ W'(t)\left[1-\frac{\theta^2}{64}\frac{ r^2 t^2\ln t}{( \frac{\theta}{8} rt+W(t))^2}\right].
\end{aligned}
 \end{equation*}
It follows from \Cref{prop:Xt21} that for $t\ge t_1\ge e$, 
 \begin{equation*}
\frac{\frac{\theta^2}{64} r^2 t^2\ln t}{( \frac{\theta}{8} rt+W(t))^2}\le \frac{4}{9 \ln t }\le \frac{4}{9}<1.
 \end{equation*}
Combining with $\frac{ \frac{\theta}{8} r t W(t)}{( \frac{\theta}{8} rt+W(t))^2}\le \frac{1}{4}$ and $\theta\le 1$, we get 
 \begin{equation*}
X'_{\frac{\theta}{8}}(t)-Z'(t)>2r-\frac{9 \theta r }{32}>0.
 \end{equation*}
Therefore, by \Cref{prop:Zte}, for $t\ge t_1$,
 \begin{equation*}
    \partial_t w_r(t,x) =\frac{ \frac{\theta}{8}Z'(t)}{x-X_{\frac{\theta}{8}}(t)+Z(t)}+\frac{\frac{\theta}{8} Z(t)\left[X_{\frac{\theta}{8}}'(t)-Z'(t)\right]}{\left[x-X_{\frac{\theta}{8}}(t)+Z(t)\right]^{2}} \ge  \frac{\frac{\theta^2r}{256}\ln t}{x-X_{\frac{\theta}{8}}(t)+Z(t)} >0 .
 \end{equation*}
Finally, we get
 \begin{equation*}
\left\{
\begin{aligned}
&\partial_x w_r(t,x) = -\frac{\theta Z(t)}{8\left[x-X_{\frac{\theta}{8}}(t)+Z(t)\right]^{2}}<0,\\
& \partial_{xx} w_r(t,x) =  \frac{\theta Z(t)}{4\left[x-X_{\frac{\theta}{8}}(t)+Z(t)\right]^{3}}>0.
\end{aligned}\right.
 \end{equation*}
 Therefore, it follows from the above calculations and the definition of $\overline{u}$ that for all $t\ge t_1$ the function $\overline{u}$ is nondecreasing on $t\in[t_1,\infty)$, and nonincreasing on $x\in \R$ and convex on $x\in[X_2(t),\infty)$.
 Since 
 \[
 \left\{
 \begin{aligned}
 &w_l\left(t,X_{\frac{\theta}{2}}(t)\right)= w_m\left(t,X_{\frac{\theta}{2}}(t)\right),  
  &\partial_x w_l\left(t,X_{\frac{\theta}{2}}(t)\right)= \partial_x w_m\left(t,X_{\frac{\theta}{2}}(t)\right)\\
  &w_m\left(t,X_{\frac{\theta}{8}}(t)\right)= w_r\left(t,X_{\frac{\theta}{8}}(t)\right),  
  &\partial_x w_m\left(t,X_{\frac{\theta}{8}}(t)\right)= \partial_x w_r\left(t,X_{\frac{\theta}{8}}(t)\right),
  \end{aligned}\right.
 \]
 Using \Cref{lemma-c11}, one has $\overline{u}(t,\cdot)\in \mathscr{C}^{1,1}_{\mathrm{\textrm{loc}}}(X_2(t),\infty)$.
\end{proof}
 Let us give some estimates for $\mathcal{D}[\overline{u}]$.
\begin{proposition}\label{dec}
Let $J$ satisfy \Cref{kernel_hypothesis}. Then the following estimates hold:
\begin{itemize}
    \item 
When $x \in \left[X_1(t),X_{\frac{\theta}{2}}(t)\right]$, there are $\mathcal{C}>0$ and $\bar t_0\ge e$, such that for all $t\ge \bar t_0$,
 \begin{equation*}
\mathcal{D}[\overline{u}](t,x)\le \mathcal{C}.
 \end{equation*}
   \item  When $x\in \left[X_{\frac{\theta}{2}}(t),X_{\frac{\theta}{8}}(t)\right]$, there is $\bar t_1>0$ such that for all $t\ge \bar t_1$,
     \begin{equation*}
    \mathcal{D}[\overline{u}](t,x)\le\frac{12\j_0}{x-X_{\frac{\theta}{2}}(t)+1}+\frac{2\j_0}{rt}.
     \end{equation*}
\item When $x\in \left[X_{\frac{\theta}{8}}(t),\infty \right)$, there is $\bar t_2>0$ such that for all $t\ge \bar t_2$,
     \begin{equation*}
    \mathcal{D}[\overline{u}](t,x)\le \frac{\mathcal{C}_1\ln t}{x-X_{\frac{\theta}{8}}(t)+Z(t)},
     \end{equation*}
    where $\mathcal{C}_1:=2\j_0\theta+4\j_0+1$.
    \end{itemize}

\end{proposition}

\begin{proof}[{\bf Proof of \Cref{dec}}]
We split space into sub-zones as follows.
\par
{\bf \# Zone 1: $\left[X_1(t),X_{\frac{\theta}{2}}(t)\right]$.} Similar to the estimation of \eqref{ub-s-zone1}, there are $\mathcal{C}>0$ and $\bar t_0\ge e$, such that for all $t\ge \bar t_0$,
 \begin{equation*}
\mathcal{D}[\overline{u}](t,x)\le \mathcal{C}.
 \end{equation*}
\par 
\medskip
     {\bf \# Zone 2: $\left[X_{\frac{\theta}{2}}(t),X_{\frac{\theta}{8}}(t)\right]$.} Since $\overline{u}(t,\cdot)$ is nonincreasing for $t\ge t_1$ by \Cref{pw1}, we write
     \begin{equation*}
    \begin{aligned}
    \mathcal{D}[\overline{u}]&\le P.V.\int_{-1}^1[\overline{u}(t,x-y)-\overline{u}(t,x)]J(y)dy+\int_1^{\max\left\{1,\frac{1}{2}\left(x-X_{\frac{\theta}{2}}(t)+1\right)\right\}}[\overline{u}(t,x-y)-\overline{u}(t,x)]J(y)dy\\
    &\quad +2\int_{\max\left\{1,\frac{1}{2}\left(x-X_{\frac{\theta}{2}}(t)+1\right)\right\}}^{\infty}J(y)dy:=I_1+I_2+I_3.
        \end{aligned}
     \end{equation*}
     Since when $x-X_{\frac{\theta}{2}}(t)+1\le 2$ one has $I_2=0$, it remains to consider the case $x-X_{\frac{\theta}{2}}(t)+1>2$.  
     \par
    For $I_1$, it follows from the symmetry of $J$ that
    \begin{equation}\label{ub-s-i1}
    I_1= \frac{1}{2}\int_{-1}^1[\overline{u}(t,x-y)+\overline{u}(t,x+y)-2\overline{u}(t,x)]J(y)dy\le \j_1\mathcal{Q}[\overline{u}](t,x).
       \end{equation}
   In view of Lemma \ref{lemma-c11}, since 
    \begin{equation*}
   \|\partial_{xx}w_l(t,\cdot)\|_{\mathscr{L}^{\infty}\left(\left[X_{\frac{\theta}{2}}-2,X_{\frac{\theta}{2}}\right]\right)}
   \lesssim \frac{1}{\ln^2 t}  \quad \text{ and }\quad 
   \|\partial_{xx}w_m(t,\cdot)\|_{\mathscr{L}^{\infty}\left(\left[X_{\frac{\theta}{2}},X_{\frac{\theta}{2}}+2\right]\right)}\lesssim \frac{1}{\ln t},
    \end{equation*}
    for $x\in\left[X_{\frac{\theta}{2}}(t),X_{\frac{\theta}{2}}(t)+1\right]$, we have
     \begin{equation*}
    \begin{aligned}  
    \mathcal{Q}[\overline{u}](t,x)&\le L_{X_{\frac{\theta}{2}}}=\max\left\{\|\partial_{xx}w_l(t,\cdot)\|_{\mathscr{L}^{\infty}\left(\left[X_{\frac{\theta}{2}}-2,X_{\frac{\theta}{2}}\right]\right)},\|\partial_{xx}w_m(t,\cdot)\|_{\mathscr{L}^{\infty}\left(\left[X_{\frac{\theta}{2}},X_{\frac{\theta}{2}}+2\right]\right)}\right\}\\
    &\lesssim \frac{1}{\ln t}\le\frac{o_{t\to\infty}(1)}{x-X_{\frac{\theta}{2}}(t)+1},
     \end{aligned}
     \end{equation*}
where $o_{t\to\infty}(1)$ is uniform in $x$ in the region.

    For $x\in \left[X_{\frac{\theta}{2}}(t)+1,X_{\frac{\theta}{8}}(t)-1\right]$, we have 
     \begin{equation*}
    \mathcal{Q}[\overline{u}](t,x)\le \sup_{y\in [x-1,x+1]}\partial_{xx}w_m(t,y)\le \frac{\theta}{8\ln t \left[x-X_{\frac{\theta}{2}}(t)\right]^2}\le\frac{o_{t\to\infty}(1)}{x-X_{\frac{\theta}{2}}(t)+1},
     \end{equation*}
    where $o_{t\to\infty}(1)$ is uniform in $x$ in the region. 
    Since $\|\partial_{xx}w_m(t,\cdot)\|_{\mathscr{L}^{\infty}\left(\left[X_{\frac{\theta}{8}}-2,X_{\frac{\theta}{8}}\right]\right)}\le\frac{\theta}{8\ln t \left[X_{\frac{\theta}{8}}(t)-X_{\frac{\theta}{2}}(t)-1\right]^2} $ and $\|\partial_{xx}w_r(t,\cdot)\|_{\mathscr{L}^{\infty}\left(\left[X_{\frac{\theta}{8}},X_{\frac{\theta}{8}}+2\right]\right)}\le \frac{\theta}{4 Z^2(t)}$,
    for $x\in \left[X_{\frac{\theta}{8}}(t)-1,X_{\frac{\theta}{8}}(t)+1\right]$, we have
    \begin{equation}\label{qpt2}
    \begin{aligned}
    \mathcal{Q}[\overline{u}]\le L_{X_{\frac{\theta}{8}}}=\max\left\{\|\partial_{xx}w_m(t,\cdot)\|_{\mathscr{L}^{\infty}\left(\left[X_{\frac{\theta}{8}}-2,X_{\frac{\theta}{8}}\right]\right)},\|\partial_{xx}w_r(t,\cdot)\|_{\mathscr{L}^{\infty}\left(\left[X_{\frac{\theta}{8}},X_{\frac{\theta}{8}}+2\right]\right)}\right\}\lesssim Z^{-2}(t).
        \end{aligned}
       \end{equation}
Thus, for $x\in\left[X_{\frac{\theta}{8}}(t)-1,X_{\frac{\theta}{8}}(t)\right]$, we get
     \begin{equation*}
    \mathcal{Q}[\overline{u}](t,x)\le\frac{o_{t\to\infty}(1)}{x-X_{\frac{\theta}{2}}(t)+1}.
     \end{equation*}
    As a result, we have 
     \begin{equation*}
   I_1\le\frac{o_{t\to\infty}(1)}{x-X_{\frac{\theta}{2}}(t)+1}.
     \end{equation*}
     \par
If $1\le x-X_{\frac{\theta}{2}}(t)+1\le2$, then $I_2=0$, while
\[
I_3=2\int_1^\infty J(y)\,dy\le2\j_0\le\frac{4\j_0}{x-X_{\frac{\theta}{2}}(t)+1}.
\]
Hence, there is $\bar t_1>t_1$ such that for all $t\ge\bar t_1$,
\[
\mathcal D[\overline u](t,x)
\le\frac{12\j_0}{x-X_{\frac{\theta}{2}}(t)+1}
\le\frac{12\j_0}{x-X_{\frac{\theta}{2}}(t)+1}+\frac{2\j_0}{rt}.
\]
\par 
     On the other hand, if $x-X_{\frac{\theta}{2}}(t)+1\ge2$, for $I_2$, by $y\le \frac{1}{2}\left(x-X_{\frac{\theta}{2}}(t)+1\right)$ and the inequality $\ln (1-z)\ge -2z$ for $z\in \left[0,\frac{1}{2}\right]$,
     \begin{equation*}
    \begin{aligned} 
    \overline{u}(t,x-y)-\overline{u}(t,x)&=-\frac{\theta}{8\ln t}\ln\left(1-\frac{y}{x-X_{\frac{\theta}{2}}(t)+1}\right)+\frac{ y}{rt\ln t}\\
    &\le \frac{\theta y}{4\ln t\left[x-X_{\frac{\theta}{2}}(t)+1\right]}+\frac{ y}{rt\ln t}.
        \end{aligned}
     \end{equation*}
    Thus, 
     \begin{equation*}
    I_2\le \j_0 \left(\frac{\theta }{4\ln t\left[x-X_{\frac{\theta}{2}}(t)+1\right]}+\frac{1}{rt\ln t}\right)\ln \left(x-X_{\frac{\theta}{2}}(t)+1\right).
     \end{equation*}
    For $I_3$, using $J(y)\le \j_0 y^{-2}$ for $y\ge 1$, we have
     \begin{equation*}
    I_3\le \frac{4\j_0}{x-X_{\frac{\theta}{2}}(t)+1}.
     \end{equation*}
   Notice that for all $t\ge\bar t_1$ and $x\in \left[X_{\frac{\theta}{2}}(t),X_{\frac{\theta}{8}}(t)\right]$, up to enlarging $\bar t_1$ if necessary,
    \begin{equation*}
   \ln \left(x-X_{\frac{\theta}{2}}(t)+1\right)\le  \ln ( rt\ln t)\le 2\ln t. 
    \end{equation*}
   Therefore, for $t\ge \bar t_1$, up to enlarging $\bar t_1$ if necessary, we have 
    \begin{equation*}
 \mathcal{D}[\overline{u}](t,x)\le \frac{12\j_0}{x-X_{\frac{\theta}{2}}(t)+1}+\frac{2\j_0}{rt}.
    \end{equation*}
   
   \medskip
    \par 
    {\bf \# Zone 3: $\left[X_{\frac{\theta}{8}}(t),\infty\right)$.}  In view of the definition of $Z(t)$, there is a time $\bar t_2\ge t_1$ such that $Z(t)\ge e>2$ for all $t\ge \bar t_2$. Since $\overline{u}(t,\cdot)$ is nonincreasing, for all $t\ge \bar t_2$, we write
       \begin{equation*}
    \begin{aligned}
    \mathcal{D}[\overline{u}](t,x)&\le P.V.\int_{-1}^1[\overline{u}(t,x-y)-\overline{u}(t,x)]J(y)dy+\int_1^{\frac{1}{2}\left(x-X_{\frac{\theta}{8}}(t)+Z(t)\right)}[\overline{u}(t,x-y)-\overline{u}(t,x)]J(y)dy\\
    &\quad +2\int_{\frac{1}{2}\left(x-X_{\frac{\theta}{8}}(t)+Z(t)\right)}^{\infty}J(y)dy:=II_1+II_2+II_3.
        \end{aligned}
     \end{equation*}
    Similar to \eqref{ub-s-i1}, we have $II_1\le \j_1\mathcal{Q}[\overline{u}](t,x)$.
    By \Cref{prop:Zte} and \eqref{qpt2}, for $x\in\left[X_{\frac{\theta}{8}}(t),X_{\frac{\theta}{8}}(t)+1\right]$, we have
     \begin{equation*}
    \mathcal{Q}[\overline{u}](t,x)\lesssim  Z^{-2}(t)\le\frac{o_{t\to\infty}(\ln t)}{x-X_{\frac{\theta}{8}}(t)+Z(t)},
     \end{equation*}
     where $o_{t\to\infty}(\ln t)$ is uniform in $x$ in the region.
    For $x\in \left[X_{\frac{\theta}{8}}(t)+1,\infty \right)$, we have
     \begin{equation*}
    \mathcal{Q}[\overline{u}](t,x)\le\sup_{y\in[x-1,x+1]}\partial_{xx}w_r(t,y)\le \frac{\theta Z(t)}{4\left[x-X_{\frac{\theta}{8}}(t)+Z(t)-1\right]^3}\le \frac{o_{t\to\infty}(\ln t)}{x-X_{\frac{\theta}{8}}(t)+Z(t)}.
     \end{equation*}
    As a result, we have 
     \begin{equation*}
   II_1\le \frac{o_{t\to\infty}(\ln t)}{x-X_{\frac{\theta}{8}}(t)+Z(t)}.
     \end{equation*}
    
    \par
    For $1\le y\le {\frac{1}{2}\left(x-X_{\frac{\theta}{8}}(t)+Z(t)\right)}$, if $x\ge X_{\frac{\theta}{8}}(t)+Z(t)$ then $x-y\ge X_{\frac{\theta}{8}}(t)$ and 
     \begin{equation*}
    \begin{aligned} 
    \overline{u}(t,x-y)-\overline{u}(t,x)&=\frac{\theta Z(t)y}{8\left(x-y-X_{\frac{\theta}{8}}(t)+Z(t)\right)\left(x-X_{\frac{\theta}{8}}(t)+Z(t)\right)}\\
    &\le \frac{\theta Z(t)y}{4\left(x-X_{\frac{\theta}{8}}(t)+Z(t)\right)^2}.
        \end{aligned}
     \end{equation*}
    Let us show that for $X_{\frac{\theta}{8}}(t)\le x\le X_{\frac{\theta}{8}}(t)+Z(t)$ we also have 
     \begin{equation*}\overline{u}(t,x-y)-\overline{u}(t,x)\le \frac{\theta Z(t)y}{\left[x-X_{\frac{\theta}{8}}(t)+Z(t)\right]^2}.  \end{equation*}
    We claim that 
     \begin{equation*}
    \sup_{\left[X_{\frac{\theta}{8}}(t)-\frac{1}{2}Z(t),X_{\frac{\theta}{8}}(t)+Z(t)\right]}|\partial_x\overline{u}(t,\cdot)|\le \frac{2}{rt\ln t}.
     \end{equation*}
    Indeed, for $z\in \left[X_{\frac{\theta}{8}}(t)-\frac{1}{2}Z(t),X_{\frac{\theta}{8}}(t)\right]$,
     \begin{equation*}
    |\partial_x\overline{u}(t,z)|=\frac{\theta}{8\ln t\left[z-X_{\frac{\theta}{2}}(t)+1\right]}+\frac{1}{rt\ln t}\le \frac{\theta}{8\ln t\left[X_{\frac{\theta}{8}}(t)-X_{\frac{\theta}{2}}(t)-\frac{1}{2}Z(t)+1\right]}+\frac{1}{rt\ln t}.
     \end{equation*}
    It follows from $X_{\frac{\theta}{8}}(t)-X_{\frac{\theta}{2}}(t)\ge \frac{3\theta}{16}rt\ln t$ for $t\ge t_1$ and $Z(t)\le \frac{\theta r t\ln t}{8}$ for $t\ge e$ that for $t\ge t_1$,
     \begin{equation*}
X_{\frac{\theta}{8}}(t)-X_{\frac{\theta}{2}}(t)-\frac{1}{2}Z(t)\ge\frac{\theta}{8}rt\ln t.
     \end{equation*}
    Thus, for $t\ge t_1$ and $z\in \left[X_{\frac{\theta}{8}}(t)-\frac{1}{2}Z(t),X_{\frac{\theta}{8}}(t)\right]$,
     \begin{equation*}
    |\partial_x\overline{u}(t,x)|\le \frac{2}{rt\ln t}.
     \end{equation*}
  For $z\in\left[X_{\frac{\theta}{8}}(t),X_{\frac{\theta}{8}}(t)+Z(t)\right]$, by \Cref{prop:Zte} we also have
     \begin{equation*}
    |\partial_x\overline{u}(t,z)|=\frac{\theta Z(t)}{8\left[z-X_{\frac{\theta}{8}}(t)+Z(t)\right]^{2}}\le \frac{\theta}{8 Z(t)}\le \frac{2}{rt\ln t}.
     \end{equation*}
    By \Cref{prop:Zte}, for $X_{\frac{\theta}{8}}(t)\le x\le X_{\frac{\theta}{8}}(t)+Z(t)$, we have 
     \begin{equation*}
    \frac{\theta Z(t)}{\left[x-X_{\frac{\theta}{8}}(t)+Z(t)\right]^2}\ge \frac{\theta}{4Z(t)}\ge \frac{2}{rt\ln t}.
     \end{equation*}
 For $t\ge  t_1$ and $X_{\frac{\theta}{8}}(t)\le x\le X_{\frac{\theta}{8}}(t)+Z(t)$, by $X_{\frac{\theta}{8}}(t)-\frac{1}{2}Z(t)\le x-y\le X_{\frac{\theta}{8}}(t)+Z(t)$,
    we then have 
  \begin{equation*}
    \overline{u}(t,x-y)-\overline{u}(t,x)= -\int_{x-y}^{x}\partial_x\overline{u}(t,z) \dd z \le  \frac{\theta Z(t)y}{\left[x-X_{\frac{\theta}{8}}(t)+Z(t)\right]^2}. 
  \end{equation*}
    Thus, using $J(y)\le \j_0 y^{-2}$ for $y\ge 1$,
     \begin{equation*}
     II_2\le \frac{\j_0\theta Z(t)\ln \left[x-X_{\frac{\theta}{8}}(t)+Z(t)\right]}{\left(x-X_{\frac{\theta}{8}}(t)+Z(t)\right)^2}.
     \end{equation*}
   Since $z\mapsto \frac{\ln z}{z}$ is decreasing on $[e,\infty)$, for $t$ large enough, say $t\ge \bar t_2\ge t_2$, we have 
     \begin{equation*}
    II_2\le \frac{\j_0\theta Z(t)\ln \left[x-X_{\frac{\theta}{8}}(t)+Z(t)\right]}{\left(x-X_{\frac{\theta}{8}}(t)+Z(t)\right)^2}\le \frac{\j_0\theta \ln Z(t)}{x-X_{\frac{\theta}{8}}(t)+Z(t)}\le \frac{2\j_0\theta \ln t}{x-X_{\frac{\theta}{8}}(t)+Z(t)}.
     \end{equation*}
    For $II_3$, using $J(y)\le \j_0 y^{-2}$ for $y\ge 1$ again, we have
     \begin{equation*}
    II_3\le \frac{4\j_0}{x-X_{\frac{\theta}{8}}(t)+Z(t)}.
     \end{equation*}
    Therefore, collecting the above estimations of $II_1$, $II_2$ and $II_3$, for all $t\ge \bar t_2$, we arrive at
     \begin{equation*}
    \mathcal{D}[\overline{u}](t,x)\le\frac{\mathcal{C}_1\ln t}{x-X_{\frac{\theta}{8}}(t)+Z(t)},
     \end{equation*}
    by $\mathcal{C}_1=2\j_0\theta+4\j_0+1$.
\end{proof}

Now, let us show that $\overline{u}$ is a supersolution to \eqref{ceq}. 
\begin{proposition}\label{prop-super-cc}
    For $t\ge \bar t:=\max\{t_1, \bar t_0,\bar t_1,\bar t_2\}$, if $r\ge \bar r:=\max\left\{2\j_0,2\theta \rho+ \frac{8\mathcal{C}}{\theta},\frac{96\j_0}{\theta},\frac{256\mathcal{C}_1}{\theta^2} \right\}$, then the function $\overline{u}(t,\cdot)$ is a supersolution to \eqref{ceq} on $[X_1(t),+\infty)$.
\end{proposition}
\begin{proof}[{\bf Proof of \Cref{prop-super-cc}}]
We split space into sub-zones as follows.
\par
  {\bf \# Zone 1: $\left[X_1(t), X_{\frac{\theta}{2}}(t)\right]$.} 
By Hypothesis \ref{hypo_f}, there is $\rho >0$ such that $f(z)\le \rho z^2$ for all $z\in(0,1)$.
In view of \Cref{pw1} and \Cref{dec}, for $t\ge t_1$ we have
 \begin{equation*}
\partial_t \overline{u}-\D[\overline{u}]-f(\overline{u})\ge\frac{r}{2\theta} \overline{u}^2 -\mathcal{C}-\rho\overline{u}^2\ge \left(\frac{r}{2\theta}-\rho\right)\overline{u}^2 -\mathcal{C}\ge \left(\frac{r}{2\theta}-\rho\right)\frac{\theta^2}{4} -\mathcal{C}\ge 0,
 \end{equation*}
as long as $r\ge 2\theta \rho+ \frac{8\mathcal{C}}{\theta}$.
\par {\bf \# Zone 2: $\left[X_{ \frac{\theta}{2}}(t), X_{ \frac{\theta}{8}}(t)\right]$.} 
Since $f(\overline{u}(t,\cdot))=0$ for $x\ge X_{\frac{\theta}{2}}(t)$, by Proposition \ref{pw1} and Proposition \ref{dec}, $t\ge \max\{t_1,\bar t_1\}$ we have
 \begin{equation*}
\begin{aligned}
&\partial_t \overline{u}-\D[\overline{u}]-f(\overline{u})\ge \frac{\frac{\theta}{8} r-12\j_0}{x-X_{\frac{\theta}{2}}(t)+1}+ \left(1-\frac{2\j_0}{r}\right)\frac{1 }{t}\ge 0,
\end{aligned}
 \end{equation*}
by taking $r\ge \max\{2\j_0,\frac{96\j_0}{\theta}\}$.

\par {\bf \# Zone 3: $\left[ X_{ \frac{\theta}{8}}(t),\infty\right)$.} 
By Proposition \ref{pw1} and Proposition \ref{dec}, $t\ge \max\{t_1,\bar t_2\}$, we achieve
 \begin{equation*}
\begin{aligned}
\partial_t \overline{u}-\D[\overline{u}]-f(\overline{u})&\ge  \frac{ \left( \frac{\theta^2r }{256} -\mathcal{C}_1\right)\ln t}{x-X_{\frac{\theta}{8}}(t)+Z(t)}\ge 0,
\end{aligned}
 \end{equation*}
by taking $r\ge \frac{256\mathcal{C}_1}{\theta^2}$.
\par Therefore, by taking $r\ge \bar r= \max\left\{2\j_0,2\theta \rho+ \frac{8\mathcal{C}}{\theta},\frac{96\j_0}{\theta},\frac{256\mathcal{C}_1}{\theta^2} \right\}$, we complete the proof.
\end{proof}

\bibliographystyle{plain}
\bibliography{cleaned_biblio}

\end{document}